\documentclass{amsart}
\usepackage{moreverb,url}
\usepackage{amssymb,amsmath,amsfonts,mathrsfs}
\usepackage{bm}
\usepackage{float}
\usepackage{verbatim}
\usepackage[square,numbers]{natbib}
\usepackage{stmaryrd}
\usepackage{fancyhdr}
\usepackage{syntax,etoolbox}
\usepackage{graphicx}
\usepackage[dvipsnames]{xcolor}
\usepackage{hyperref}
\usepackage{subcaption}
\usepackage[font=scriptsize, labelfont=bf]{caption}
\hypersetup{colorlinks=true,allcolors=blue}
\usepackage[square,numbers]{natbib}
\usepackage[left=2.5cm, right=2.5cm, bottom=3cm]{geometry}
\usepackage{lineno}

\newtheorem{theorem}{Theorem}[section]
\newtheorem{lemma}{Lemma}[section]

\newtheorem{innercustomgeneric}{\customgenericname}
\providecommand{\customgenericname}{}
\newcommand{\newcustomproblem}[2]{%
	\newenvironment{#1}[1]
	{%
		\renewcommand\customgenericname{#2}%
		\renewcommand\theinnercustomgeneric{##1}%
		\innercustomgeneric
	}
	{\endinnercustomgeneric}
}

\newcustomproblem{customprob}{Problem}

\newcommand*{\bqed}{\hfill\ensuremath{\blacksquare}}%

\def\dd{\, \mathrm{d}}

\begin{document}
	
	\today
	
	
	\title[An obstacle problem for generalised membranes]{On the justification of Koiter's model for generalised membrane shells of the first kind confined in a half-space}
	
	
	\author{Paolo Piersanti}
	\address{School of Science and Engineering, The Chinese University of Hong Kong (Shenzhen), 2001 Longxiang Blvd., Longgang District, Shenzhen, China}
	\email{ppiersanti@cuhk.edu.cn}

	\begin{abstract}
		In this paper we justify Koiter's model for linearly elastic generalised membrane shells of the first kind subjected to remaining confined in a prescribed half-space. After showing that the confinement condition considered in this paper is equivalent, under the validity of the Kirchhoff-Love assumptions, to the classical Signorini condition, we formulate the corresponding obstacle problem for a three-dimensional linearly elastic generalised membrane shell of the first kind, and we perform a rigorous asymptotic analysis as the thickness approaches to zero on the unique solution for one such model. We show - and this constitutes the first novelty in this paper - that the solution to the three-dimensional obstacle problem converges to the unique solution of a two-dimensional model consisting of a set of variational inequalities that are posed over the abstract completion of a non-empty, closed and convex set. The two-dimensional limit model with Koiter's model as a point of departure is characterised by the same variational inequalities as the two-dimensional model obtained with the variational formulation of a three-dimensional linearly elastic generalised membrane shell subjected to remaining confined in a prescribed half-space as a point of departure, with the noticeable difference that the sets where solutions for these two-dimensional limit models are sought do not coincide, in general.
		
		In order to complete the justification of Koiter's model for linearly elastic generalised membrane shells of the first kind subjected to remaining confined in a half-space, we give sufficient conditions ensuring that the sets where solutions for these two-dimensional limit models are sought coincide. The latter constitutes the second main novelty presented in this work.
		
		\smallskip
		\noindent \textbf{Keywords.} Variational Inequalities $\cdot$ Generalised Membrane Shells $\cdot$ Asymptotic Analysis $\cdot$ Obstacle Problems
		
		\smallskip
		\noindent \textbf{MSC 2020.} 35J86, 47H07, 74B05.
	\end{abstract}
	
	\maketitle
	
	\section{Introduction}
\label{Sec:0}

The contact of this elastic films with rigid foundations arises in many applicative fields such as Engineering, Material Science, Medicine and Biology. For instance, the description of the motion inside the human heart of the three aorta valves, which can be regarded as linearly elastic shells, is governed by a mathematical model built up in a way such that each valve remains confined in a certain portion of space without penetrating or being penetrated by the other two valves. In this direction we cite the recent references~\cite{Quarteroni2021-3}, \cite{Quarteroni2021-2} and~\cite{Quarteroni2021-1}.

The geometrical deformation of a linearly elastic shell is modelled, in general, via the classical three-dimensional energy for linearised elasticity (cf., e.g., \cite{Ciarlet1988}). However, this model has some drawbacks, as it cannot be used to study models where the shell under consideration is made of a non-homogeneous or anisotropic material (cf., e.g., \cite{CaillerieSanchez1995a} and~\cite{CaillerieSanchez1995b}), or its thickness varies periodically (cf., e.g., \cite{TelegaLewinski1998a} and~\cite{TelegaLewinski1998b}).

To overcome the intrinsic limitations of this model, W.T.~Koiter devised, in the papers~\cite{Koiter,Koiter1970}, an alternative model that was meant to provide information on the deformation of a linearly elastic shell by solely considering the deformation of its middle surface. Another important feature of Koiter's model is that it is valid for all types of linearly elastic shells.

The rigorous mathematical justification of Koiter's model for linearly elastic shells in the case where no obstacles are taken into account was established by P.G.~Ciarlet and his collaborators in the papers~\cite{CiaDes1979, CiaLods1996a,CiaLods1996b,CiaLods1996c,CiaLods1996d,CiaLodsMia1996}. We also mention the papers~\cite{MeiPie2024,PPS2024,PPS2024-2,PS,Shen2018,Shen2019,Shen2020}, which are about the numerical analysis and simulations of problems arising in shell theory. For the justification of the time-dependent Koiter's model when the action of temperature is considered and no obstacles are taken into account, we refer to~\cite{Pie-2021}. 

The study of obstacle problems for thin elastic films was initiated by A.~L\'{e}ger and B.~Miara in the papers~\cite{LegMia2008,LegMia2018}, where the authors focused on linearly elastic shallow shells and linearly elastic flexural shells subjected to a constraint for the transverse component of the deformation, imposing that the deformed reference configuration must remain confined in a prescribed half-space.

The papers~\cite{CiaPie2018b,CiaPie2018bCR,CiaMarPie2018b,CiaMarPie2018,PieJDE2022,Pie2023} consider a more general confinement condition, as they discuss the justification of Koiter's model in the case where the shell, regardless of whether it is an elliptic membrane or a flexural shell, remains confined in a general half-space. The confinement condition considered in the aforementioned papers is different from the usual Signorini condition that is favoured in many textbooks on contact mechanics (cf., e.g.,~\cite{KikuchiOden1988}). Indeed, the Signorini condition solely requires the boundary of the deformed three-dimensional elastic body to obey the constraint whereas the confinement condition we will be considering applies to \emph{all the points} of the deformed reference configuration. The latter approach seems to be more amenable in the context of a dimension-reduction analysis, where we expect the two-dimensional limit model, defined over a surface in the Euclidean space, to obey the constraint at all the points of one such surface.

The papers~\cite{CiaMarPie2018b,CiaMarPie2018,PieJDE2022} discuss the identification of a two-dimensional limit model obtained as a result of a rigorous asymptotic analysis as the thickness approaches zero departing from the three-dimensional obstacle problem for a linearly elastic elliptic membrane shell constrained to remain confined in a prescribed half-space. The latter category of shells is characterised by the fact that the Gaussian curvature of the middle surface is positive at each point of the definition domain, and by the fact that the displacement vanishes on the entire lateral face. The boundary conditions characterising linearly elastic elliptic membrane shells allow to bypass all the analytical difficulties and technicalities that are instead faced when considering linearly elastic generalised membrane shells of the first kind.
The paper~\cite{Pie2023} discusses the identification of a two-dimensional limit model obtained as a result of a rigorous asymptotic analysis as the thickness approaches zero departing from the three-dimensional obstacle problem for a linearly elastic flexural shell constrained to remain confined in a prescribed half-space. In this case, the space of solutions for the model under consideration is characterised by the fact that admissible displacements are \emph{inextensional}, in the sense that the linearised change of metric tensor at any of these admissible displacements vanishes. This feature of admissible displacements renders the analysis challenging, and requires the development of a new strategy based on the penalty method and on the formulation of a uniformity property (Theorem~3.3 in~\cite{Pie2023}) stating that the convergence process as the penalty parameter approaches zero is independent of the thickness under the assumption of higher regularity for the solution of the three-dimensional obstacle problem.
The papers~\cite{CiaPie2018b,CiaPie2018bCR} complement the results obtained in~\cite{CiaMarPie2018b,CiaMarPie2018,PieJDE2022,Pie2023} by verifying that the two-dimensional limit models obtained as a result of a rigorous asymptotic analysis as the thickness approaches zero departing from Koiter's model for shells subjected to remaining confined in a half-space coincide with the two-dimensional limit models obtained departing from the three-dimensional obstacle problem for either linearly elastic elliptic membrane shells or linearly elastic flexural shells constrained to remain confined in the \emph{same} half-space as Koiter's model. The case of generalised membrane shells of the first kind is also analysed, though at the time the papers~\cite{CiaPie2018b,CiaPie2018bCR} were published there was no record in the literature concerning the identification of a two-dimensional model  obtained as a result of a rigorous asymptotic analysis as the thickness approaches zero departing from the three-dimensional obstacle problem for a linearly elastic generalised membrane shell of the first kind. The purpose of this paper is to complete the study initiated in~\cite{CiaPie2018b,CiaPie2018bCR,CiaMarPie2018b,CiaMarPie2018,PieJDE2022,Pie2023}, by showing that the two-dimensional limit model obtained departing from the three-dimensional obstacle problem for a linearly elastic generalised membrane shell of the first kind subjected to remaining confined in a half-space coincides with the two-dimensional limit model obtained departing from Koiter's model~\cite{CiaPie2018b,CiaPie2018bCR}.

To the best of our knowledge, the justification of Koiter's model for linearly elastic generalised membrane shells appears to be an open problem, to-date. In order to achieve this goal, we will first perform a rigorous asymptotic analysis starting with the classical energy of three-dimensional linearised elasticity for which the admissible competitors to the role of minimiser are sought in a non-empty, closed, and convex subspace of a suitable space, that takes into account the constraint according to which the shell must remain confined in a prescribed half-space.

We will recover the same two-dimensional limit-model as the one recovered in~\cite{CiaPie2018b}, although the set of admissible solutions for the two-dimensional limit-model recovered here is in general larger than the set identified in~\cite{CiaPie2018b}. We then identify reasonable sufficient conditions for which the two sets coincide. The latter property is established via a \emph{density argument}.

We will also show in passing that the confinement condition we are considering is actually \emph{equivalent} to the Signorini condition, provided that the Kirchhoff-Love assumptions are hypothesised.

For completeness, we also mention the recent papers~\cite{PWDT2D} and~\cite{PWDT3D}, which provide examples of non-linear variational models arising in biology, as well as the paper~\cite{PieTem2023}, that discusses the modelling and analysis of the melting of a shallow ice sheet by means of a set of doubly non-linear parabolic variational inequalities. We also mention the recent works~\cite{HeaNai24, LK24} that explore new ideas in contact problems in non-linear elasticity.

The paper is divided into six sections (including this one). In section~\ref{Sec:1} we recall some background and notation. In section~\ref{Sec:2} we formulate a three-dimensional obstacle problem for a general linearly elastic shell. In section~\ref{Sec:3} we, more specifically, consider the case where the linearly elastic shell under consideration is a linearly elastic generalised membrane of the first kind. We then scale the three-dimensional obstacle problem in a way such that it is posed over a domain that is independent of the thickness parameter. In section~\ref{Sec:4}, a rigorous asymptotic analysis is carried out and the sought set of \emph{abstract} two-dimensional variational inequalities is recovered. We also observe that the concept of solution for the two-dimensional limit model recovered as a result of the previous asymptotic analysis is different from the concept of solution for the two-dimensional limit model recovered upon completion of a rigorous asymptotic analysis starting with Koiter's model subjected to the same confinement condition as the one considered in this paper, even though the \emph{abstract} variational inequalities are the same.
In section~\ref{Sec:5}, we give sufficient conditions ensuring that the concept of solution for the two-dimensional limit model recovered starting with the three-dimensional obstacle problem for linearly elastic generalised membrane shells of the first kind and the concept of solution for the same two-dimensional limit model recovered starting with Koiter's model subjected to the same confinement condition as the one considered in this paper are actually identical.

The main results presented in this paper are Theorem~\ref{asymptotics} and Theorem~\ref{th:density}. In Theorem~\ref{asymptotics}, we identify the two-dimensional limit model obtained as a result of a rigorous asymptotic analysis as the thickness approaches zero departing from the three-dimensional obstacle problem for a linearly elastic generalised membrane of the first kind introduced in section~\ref{Sec:2} and section~\ref{Sec:3}. In Theorem~\ref{th:density} we establish a density result that unifies the result obtained in Theorem~\ref{asymptotics} with those in~\cite{CiaPie2018b,CiaPie2018bCR}, thus completing the justification of Koiter's model for linearly elastic generalised shells of the first kind subjected to remaining confined in a half-space.

\section{Geometrical preliminaries}
\label{Sec:1}

For details about the classical notions of differential geometry used in this section and the next one, see, e.g.~\cite{Ciarlet2000} or \cite{Ciarlet2005}.

Greek indices and exponents, except $\varepsilon$ and $\nu$ in $\partial_{\nu}$, take their values in the set $\{1,2\}$, while Latin indices or exponents, except when they are used for indexing sequences, take their values in the set $\{1,2,3\}$, and the summation convention with respect to repeated indices or exponents is systematically used in conjunction with these two rules. The notation $\mathbb{E}^3$ designates the three-dimensional Euclidean space. The Euclidean inner product and the vector product of $\bm{u}, \bm{v} \in \mathbb{E}^3$ are denoted $\bm{u} \cdot \bm{v}$ and $\bm{u} \wedge \bm{v}$, respectively. The Euclidean norm of $\bm{u} \in \mathbb{E}^3$ is denoted $|\bm{u}|$. The notation $\delta^j_i$ designates the Kronecker symbol.

Given an open subset $\Omega$ of $\mathbb{R}^n$, the notations $L^2(\Omega)$, $H^1(\Omega)$, or $H^2 (\Omega)$, denote the standard Lebesgue and Sobolev spaces, and the notation $\mathcal{D}(\Omega)$ denotes the space of all functions that are infinitely differentiable over $\Omega$ and have compact supports in $\Omega$. The notation $\|\cdot\|_X$ designates the norm of a normed vector space $X$. Given a vector space $X$, the corresponding product space $X \times X \times X$ is denoted by $\bm{X}$.

The abbreviations \emph{``a.a.''} and \emph{``a.e.''} mean \emph{almost all} and \emph{almost everywhere}, respectively.

The boundary $\Gamma$ of an open subset $\Omega$ in $\mathbb{R}^n$ is said to be Lipschitz-continuous if the following conditions are satisfied: Given an integer $s\ge 1$, there exist constants $\alpha_1>0$ and $L>0$, and a finite number of local coordinate systems, with coordinates $\bm{\phi}'_r=(\phi_1^r, \dots, \phi_{n-1}^r) \in \mathbb{R}^{n-1}$ and $\phi_r=\phi_n^r$, sets $\tilde{\omega}_r:=\{\bm{\phi}_r \in\mathbb{R}^{n-1}; |\bm{\phi}_r|<\alpha_1\}$, $1 \le r \le s$, and corresponding functions
$$
\tilde{\theta}_r:\tilde{\omega}_r\to\mathbb{R},
$$
such that
$$
\Gamma=\bigcup_{r=1}^s \{(\bm{\phi}'_r,\phi_r); \bm{\phi}'_r \in \tilde{\omega}_r \textup{ and }\phi_r=\tilde{\theta}_r(\bm{\phi}'_r)\},
$$
and 
$$
|\tilde{\theta}_r(\bm{\phi}'_r)-\tilde{\theta}_r(\bm{\upsilon}'_r)|\le L |\bm{\phi}'_r-\bm{\upsilon}'_r|, \textup{ for all }\bm{\phi}'_r, \bm{\upsilon}'_r \in \tilde{\omega}_r, \textup{ and all }1\le r\le s.
$$

The second last formula takes into account overlapping local charts, while the last set of inequalities express the Lipschitz continuity of the mappings $\tilde{\theta}_r$.

An open set $\Omega$ is said to be locally on the same side of its boundary if, in addition, there exists a constant $\alpha_2>0$ such that
\begin{align*}
	\{(\bm{\phi}'_r,\phi_r);\bm{\phi}'_r \in\tilde{\omega}_r \textup{ and }\tilde{\theta}_r(\bm{\phi}'_r) < \phi_r < \tilde{\theta}_r(\bm{\phi}'_r)+\alpha_2\}&\subset \Omega,\quad\textup{ for all } 1\le r\le s,\\
	\{(\bm{\phi}'_r,\phi_r);\bm{\phi}'_r \in\tilde{\omega}_r \textup{ and }\tilde{\theta}_r(\bm{\phi}'_r)-\alpha_2 < \phi_r < \tilde{\theta}_r(\bm{\phi}'_r)\}&\subset \mathbb{R}^n\setminus\overline{\Omega},\quad\textup{ for all } 1\le r\le s.
\end{align*}

A \emph{domain} in $\mathbb{R}^n$ is a bounded Lipschitz domain, namely, a bounded and connected open subset $\Omega$ of $\mathbb{R}^n$, whose boundary $\partial \Omega$ is Lipschitz-continuous, the set $\Omega$ being locally on a single side of $\partial \Omega$. We denote by $\partial_{\nu}$ the normal derivative along the boundary $\partial\Omega$. For more details, see, e.g. Section~8.2 in~\cite{Ciarlet2025}.

Let $\omega$ be a domain in $\mathbb{R}^2$, let $y = (y_\alpha)$ denote a generic point in $\overline{\omega}$, and let $\partial_\alpha := \partial / \partial y_\alpha$ and $\partial_{\alpha\beta} := \partial^2/\partial y_\alpha \partial y_\beta$. A mapping $\bm{\theta} \in \mathcal{C}^1(\overline{\omega}; \mathbb{E}^3)$ is an \emph{immersion} if the two vectors
$$
\bm{a}_\alpha(y) := \partial_\alpha \bm{\theta}(y)
$$
are linearly independent at each point $y \in \overline{\omega}$. Then the image $\bm{\theta}(\overline{\omega})$ of the set $\overline{\omega}$ under the mapping $\bm{\theta}$ is a \emph{surface in} $\mathbb{E}^3$, equipped with $y_1, y_2$ as its \emph{curvilinear coordinates}. Given any point $y\in \overline{\omega}$, the vectors $\bm{a}_\alpha(y)$ span the \emph{tangent plane} to the surface $\bm{\theta}(\overline{\omega})$ at the point $\bm{\theta}(y)$, the unit vector
$$
\bm{a}_3(y) := \frac{\bm{a}_1(y) \wedge \bm{a}_2(y)}{|\bm{a}_1(y) \wedge \bm{a}_2(y)|}
$$
is normal to $\bm{\theta}(\overline{\omega})$ at $\bm{\theta} (y)$, the three vectors $\bm{a}_i(y)$ form the \emph{covariant} basis at $\bm{\theta}(y)$, and the three vectors $\bm{a}^j(y)$ defined by the relations
$$
\bm{a}^j(y) \cdot \bm{a}_i(y) = \delta^j_i
$$
form the \emph{contravariant} basis at $\bm{\theta}(y)$; note that the vectors $\bm{a}^\beta (y)$ also span the tangent plane to $\bm{\theta}(\overline{\omega})$ at $\bm{\theta}(y)$ and that $\bm{a}^3(y) = \bm{a}_3(y)$.

The \emph{first fundamental form} of the surface $\bm{\theta} (\overline{\omega})$ is then defined by means of its \emph{covariant components}
$$
a_{\alpha\beta} := \bm{a}_\alpha \cdot \bm{a}_\beta = a_{\beta\alpha} \in \mathcal{C}^0(\overline{\omega}),
$$
or by means of its \emph{contravariant components}
$$
a^{\alpha\beta}:= \bm{a}^\alpha \cdot \bm{a}^\beta = a^{\beta\alpha}\in \mathcal{C}^0(\overline{\omega}).
$$

Note that the symmetric matrix field $(a^{\alpha\beta})$ is then the inverse of the matrix field $(a_{\alpha\beta})$, that $\bm{a}^\beta = a^{\alpha\beta}\bm{a}_\alpha$ and $\bm{a}_\alpha = a_{\alpha\beta} \bm{a}^\beta$, and that the \emph{area element} along $\bm{\theta}(\overline{\omega})$ is given at each point $\bm{\theta}(y)$, for each $y \in \overline{\omega}$, by $\sqrt{a(y)}\dd y$, where
$$
a := \det(a_{\alpha\beta}) \in \mathcal{C}^0(\overline{\omega}).
$$

Given an immersion $\bm{\theta} \in \mathcal{C}^2(\overline{\omega}; \mathbb{E}^3)$, the \emph{second fundamental form} of the surface $\bm{\theta}(\overline{\omega})$ is defined by means of its \emph{covariant components}
$$
b_{\alpha\beta}:= \partial_\alpha \bm{a}_\beta \cdot \bm{a}_3 = -\bm{a}_\beta \cdot \partial_\alpha \bm{a}_3 = b_{\beta\alpha} \in \mathcal{C}^0(\overline{\omega}),
$$
or by means of its \emph{mixed components}
$$
b^\beta_\alpha := a^{\beta\sigma} b_{\alpha \sigma} \in \mathcal{C}^0(\overline{\omega}),
$$
and the \emph{Christoffel symbols} associated with the immersion $\bm{\theta}$ are defined by
$$
\Gamma^\sigma_{\alpha\beta}:= \partial_\alpha \bm{a}_\beta \cdot \bm{a}^\sigma = \Gamma^\sigma_{\beta\alpha} \in \mathcal{C}^0(\overline{\omega}).
$$

Given an injective immersion $\bm{\theta} \in \mathcal{C}^2 (\overline{\omega}; \mathbb{E}^3)$ and a vector field $\bm{\eta} = (\eta_i) \in \mathcal{C}^1(\overline{\omega};\mathbb{R}^3)$, the vector field
$$
\tilde{\bm{\eta}} := \eta_i \bm{a}^i
$$
can be viewed as a \emph{displacement field of the surface} $\bm{\theta}(\overline{\omega})$, thus defined by means of its \emph{covariant components} $\eta_i$ over the vectors $\bm{a}^i$ of the contravariant bases along the surface. If the norms $\|\eta_i\|_{\mathcal{C}^1(\overline{\omega})}$ are small enough, the mapping $(\bm{\theta} + \eta_i \bm{a}^i) \in \mathcal{C}^1(\overline{\omega}; \mathbb{E}^3)$ is also an injective immersion, so that the set $(\bm{\theta} + \eta_i \bm{a}^i) (\overline{\omega})$ is also a surface in $\mathbb{E}^3$, equipped with the same curvilinear coordinates as those of the surface $\bm{\theta} (\overline{\omega})$, called the \emph{deformed surface} corresponding to the displacement field $\tilde{\bm{\eta}} = \eta_i \bm{a}^i$. One can then define the first fundamental form of the deformed surface by means of its covariant components
$$
a_{\alpha\beta}(\bm{\eta}):= (\bm{a}_\alpha + \partial_\alpha \tilde{\bm{\eta}}) \cdot (\bm{a}_\beta + \partial_\beta \tilde{\bm{\eta}})= a_{\alpha\beta} + \bm{a}_\alpha \cdot \partial_\beta \tilde{\bm{\eta}} + \partial_\alpha \tilde{\bm{\eta}} \cdot \bm{a}_\beta + \partial_\alpha \tilde{\bm{\eta}} \cdot \partial_\beta \tilde{\bm{\eta}}.
$$

The \emph{linear part with  respect to} $\tilde{\bm{\eta}}$ in the difference $\frac{1}{2}(a_{\alpha\beta}(\bm{\eta}) - a_{\alpha\beta})$ is called the \emph{linearised change of metric tensor} associated with the displacement field $\eta_i \bm{a}^i$, the covariant components of which are defined by:
$$
\gamma_{\alpha\beta}(\bm{\eta}) := \dfrac{1}{2} \left(\bm{a}_\alpha \cdot \partial_\beta \tilde{\bm{\eta}} + \partial_\alpha \tilde{\bm{\eta}} \cdot \bm{a}_\beta\right) = \dfrac{1}{2}(\partial_\beta \eta_\alpha + \partial_\alpha \eta_\beta ) - \Gamma^\sigma_{\alpha\beta} \eta_\sigma - b_{\alpha\beta} \eta_3 = \gamma_{\beta\alpha} (\bm{\eta}).
$$

The \emph{linear part with  respect to} $\tilde{\bm{\eta}}$ in the difference $(b_{\alpha\beta}(\bm{\eta}) - b_{\alpha\beta})$ is called the \emph{linearised change of curvature tensor} associated with the displacement field $\eta_i \bm{a}^i$, the covariant components of which are formally defined by:
\begin{align*}
\rho_{\alpha\beta}(\bm{\eta}) &:= (\partial_{\alpha\beta}\tilde{\bm{\eta}}-\Gamma_{\alpha\beta}^\sigma \partial_\sigma \tilde{\bm{\eta}})\cdot \bm{a}_3 = \rho_{\beta\alpha}(\bm{\eta})\\
&= \partial_{\alpha\beta}\eta_3 -\Gamma_{\alpha\beta}^\sigma \partial_\sigma\eta_3 -b_\alpha^\sigma b_{\sigma\beta}\eta_3
+b_\alpha^\sigma(\partial_\beta \eta_\sigma-\Gamma_{\beta\sigma}^\tau\eta_\tau)+b_\beta^\tau(\partial_\alpha\eta_\tau-\Gamma_{\alpha\tau}^\sigma\eta_\sigma)
+(\partial_\alpha b_\beta^\tau+\Gamma_{\alpha\sigma}^\tau b_\beta^\sigma-\Gamma_{\alpha\beta}^\sigma b_\sigma^\tau)\eta_\tau.
\end{align*}

\section{A general three-dimensional obstacle problem for a linearly elastic shell} \label{Sec:2}

Let $\omega$ be a domain in $\mathbb{R}^2$, let $\gamma:= \partial \omega$, and let $\gamma_0$ be a non-empty relatively open subset of $\gamma$. For each $\varepsilon > 0$, we define the sets
$$
\Omega^\varepsilon := \omega \times (- \varepsilon , \varepsilon) \quad \textup{ and } \quad \Gamma^\varepsilon_0 := \gamma_0 \times [-\varepsilon,\varepsilon],
$$
we let $x^\varepsilon = (x^\varepsilon_i)$ designate a generic point in the set $\overline{\Omega^\varepsilon}$, and we let $\partial^\varepsilon_i := \partial / \partial x^\varepsilon_i$. Hence we also have $x^\varepsilon_\alpha = y_\alpha$ and $\partial^\varepsilon_\alpha = \partial_\alpha$.

Given an \emph{injective} immersion $\bm{\theta} \in \mathcal{C}^3(\overline{\omega}; \mathbb{E}^3)$ and $\varepsilon > 0$, consider a \emph{shell} with \emph{middle surface} $\bm{\theta} (\overline{\omega})$ and with \emph{constant thickness} $2 \varepsilon$. This means that the \emph{reference configuration} of the shell is the set $\bm{\Theta}^\varepsilon(\overline{\Omega^\varepsilon})$, where the mapping $\bm{\Theta}^\varepsilon:\overline{\Omega^\varepsilon} \to \mathbb{E}^3$ is defined by
$$
\bm{\Theta}^\varepsilon(x^\varepsilon) := \bm{\theta}(y) + x^\varepsilon_3 \bm{a}^3(y), \quad \textup{ for all } x^\varepsilon = (y,x^\varepsilon_3) \in \overline{\Omega^\varepsilon}.
$$

One can then show (cf., e.g., Theorem 3.1-1 of~\cite{Ciarlet2000}) that, if $\varepsilon > 0$ is small enough, such a mapping $\bm{\Theta}^\varepsilon \in \mathcal{C}^2(\overline{\Omega^\varepsilon}; \mathbb{E}^3)$ is also an injective immersion, in the sense that the three vectors
$$
\bm{g}^\varepsilon_i(x^\varepsilon) := \partial^\varepsilon_i \bm{\Theta}^\varepsilon(x^\varepsilon),
$$
are linearly independent at each point $x^\varepsilon \in \overline{\Omega^\varepsilon}$. The vectors $\{\bm{g}^\varepsilon_i\}_{i=1}^3$ constitute the \emph{covariant basis} at the point $\bm{\Theta}^\varepsilon(x^\varepsilon)$, while the three vectors $\bm{g}^{j,\varepsilon}(x^\varepsilon)$ defined by the relations
$$
\bm{g}^{j,\varepsilon}(x^\varepsilon) \cdot \bm{g}^\varepsilon_i(x^\varepsilon) = \delta^j_i,
$$
constitute the \emph{contravariant basis} at the same point. It will be implicitly assumed in the sequel that $\varepsilon > 0$ \emph{is small enough so that $\bm{\Theta}^\varepsilon:\overline{\Omega^\varepsilon} \to \mathbb{E}^3$} is an injective immersion.

One then defines the \emph{metric tensor of the reference configuration} $\bm{\Theta}^\varepsilon(\overline{\Omega^\varepsilon})$ by means of its \emph{covariant components}
$$
g^\varepsilon_{ij} := \bm{g}^\varepsilon_i \cdot \bm{g}^\varepsilon_j \in \mathcal{C}^1(\overline{\Omega^\varepsilon}),
$$
or by means of its \emph{contravariant components}
$$
g^{ij,\varepsilon} := \bm{g}^{i,\varepsilon} \cdot \bm{g}^{j,\varepsilon} \in \mathcal{C}^1(\overline{\Omega^\varepsilon}).
$$

Note that the symmetric matrix field $(g^{ij,\varepsilon})$ is then the inverse of the matrix field $(g^\varepsilon_{ij})$, that $\bm{g}^{j,\varepsilon} = g^{ij,\varepsilon} \bm{g}^\varepsilon_i$ and $\bm{g}^\varepsilon_i = g^\varepsilon_{ij} \bm{g}^{j,\varepsilon}$, and that the \emph{volume element} in $\bm{\Theta}^\varepsilon(\overline{\Omega^\varepsilon})$ is given at each point $\bm{\Theta}^\varepsilon(x^\varepsilon)$, for all $x^\varepsilon \in \overline{\Omega^\varepsilon}$, by $\sqrt{g^\varepsilon (x^\varepsilon)} \dd x^\varepsilon$, where
$$
g^\varepsilon := \det (g^\varepsilon_{ij}) \in \mathcal{C}^1(\overline{\Omega^\varepsilon}).
$$

One also defines the \emph{Christoffel symbols} associated with the immersion $\bm{\Theta}^\varepsilon$ by
$$
\Gamma^{p,\varepsilon}_{ij}:= \partial_i^\varepsilon \bm{g}^\varepsilon_j \cdot \bm{g}^{p,\varepsilon} = \Gamma^{p,\varepsilon}_{ji} \in \mathcal{C}^0(\overline{\Omega^\varepsilon}).
$$
Note that $\Gamma^{3,\varepsilon}_{\alpha 3} = \Gamma^{p, \varepsilon}_{33} = 0$.

Given a vector field $\bm{v}^\varepsilon = (v^\varepsilon_i) \in \mathcal{C}^1(\overline{\Omega^\varepsilon};\mathbb{R}^3)$, the associated vector field
$$
\tilde{\bm{v}}^\varepsilon := v^\varepsilon_i \bm{g}^{i,\varepsilon},
$$
can be viewed as a \emph{displacement field} of the reference configuration $\bm{\Theta}^\varepsilon(\overline{\Omega^\varepsilon})$ of the shell, thus defined by means of its \emph{covariant components} $v^\varepsilon_i$ over the vectors $\bm{g}^{i,\varepsilon}$ of the contravariant bases in the reference configuration.

If the norms $\|v^\varepsilon_i\|_{\mathcal{C}^1(\overline{\Omega^\varepsilon})}$ are small enough, the mapping $(\bm{\Theta}^\varepsilon + v^\varepsilon_i \bm{g}^{i,\varepsilon})$ is also an injective immersion, so that one can also define the  \emph{metric tensor of the deformed configuration} $(\bm{\Theta}^\varepsilon + v^\varepsilon_i \bm{g}^{i,\varepsilon})(\overline{\Omega^\varepsilon})$ by means of its covariant components
\begin{align*}
g^\varepsilon_{ij}(v^\varepsilon) := (\bm{g}^\varepsilon_i + \partial^\varepsilon_i \tilde{\bm{v}}^\varepsilon) \cdot (\bm{g}^\varepsilon_j + \partial^\varepsilon_j \tilde{\bm{v}}^\varepsilon)
= g^\varepsilon_{ij} + \bm{g}^\varepsilon_i \cdot \partial_j^\varepsilon \tilde{\bm{v}}^\varepsilon + \partial^\varepsilon_i \tilde{\bm{v}}^\varepsilon \cdot \bm{g}^\varepsilon_j + \partial_i^\varepsilon \tilde{\bm{v}}^\varepsilon \cdot \partial_j^\varepsilon \tilde{\bm{v}}^\varepsilon.
\end{align*}

The linear part with respect to $\tilde{\bm{v}}^\varepsilon$ in the difference $\dfrac{1}{2}(g^\varepsilon_{ij} (\bm{v}^\varepsilon) - g^\varepsilon_{ij})$ is called the \emph{linearised strain tensor} associated with the displacement field $v^\varepsilon_i \bm{g}^{i, \varepsilon}$, the covariant components of which are defined by:
$$
e^\varepsilon_{i\|j}(\bm{v}^\varepsilon) := \frac{1}{2} \left(\bm{g}^\varepsilon_i \cdot \partial_j^\varepsilon \tilde{\bm{v}}^\varepsilon + \partial^\varepsilon_i \tilde{\bm{v}}^\varepsilon \cdot \bm{g}^\varepsilon_j\right) = \frac{1}{2} (\partial^\varepsilon_j v^\varepsilon_i + \partial^\varepsilon_i v^\varepsilon_j) - \Gamma^{p,\varepsilon}_{ij} v^\varepsilon_p = e_{j\|i}^\varepsilon(\bm{v}^\varepsilon).
$$
The functions $e^\varepsilon_{i\|j}(\bm{v}^\varepsilon)$ are called the \emph{linearised strains in curvilinear coordinates} associated with the displacement field $v^\varepsilon_i \bm{g}^{i,\varepsilon}$.

We assume throughout this paper that, for each $\varepsilon > 0$, the reference configuration $\bm{\Theta}^\varepsilon(\overline{\Omega^\varepsilon})$ of the shell is a \emph{natural state} (i.e., stress-free) and that the material constituting the shell is \emph{homogeneous}, \emph{isotropic}, and \emph{linearly elastic}. The behaviour of such an elastic material is thus entirely governed by its two \emph{Lam\'{e} constants} $\lambda \ge 0$ and $\mu > 0$ (for details, see, e.g., Section~3.8 of~\cite{Ciarlet1988}).

We will also assume that the shell is subjected to \emph{applied body forces} whose density per unit volume is defined by means of its covariant components $f^{i, \varepsilon} \in L^2(\Omega^\varepsilon)$, and to a \emph{homogeneous boundary condition of place} along the portion $\Gamma^\varepsilon_0$ of its lateral face (i.e., the displacement vanishes on $\Gamma^\varepsilon_0$).

The confinement condition considered in this paper, which was originally suggested by Brezis \& Stampacchia~\cite{BrezisStampacchia1968}, does not take any traction forces into account. One reason for neglecting traction forces is that, since this portion of the boundary is not \emph{a priori} specified, i.e., it is an \emph{unknown} in such obstacle problem , there should be no traction forces acting on the portion of the boundary of the three-dimensional shell in contact with the obstacle.

In this paper we consider a specific \emph{obstacle problem} for such a shell, in the sense that the shell is subjected to a \emph{confinement condition}, expressing that any \emph{admissible displacement vector field} $v^\varepsilon_i \bm{g}^{i, \varepsilon}$ must be such that all the points of the corresponding deformed configuration remain in a prescribed \emph{half-space} of the form
$$
\mathbb{H}:=\{x\in\mathbb{E}^3;\boldsymbol{Ox}\cdot\bm{q}\ge 0\},
$$
where $\bm{q}$ is a \emph{unit-vector} which is given once and for all, and which is thus orthogonal to the plane associated with the half-space where the linearly elastic shell is required to remain confined.
In other words, any admissible displacement field must satisfy
\begin{equation}
	\label{cc-original}
	\left(\bm{\Theta}^\varepsilon(x^\varepsilon)+v^\varepsilon_i(x^\varepsilon)\bm{g}^{i,\varepsilon}(x^\varepsilon)\right)\cdot\bm{q}\ge 0,
\end{equation}
for all $x^\varepsilon \in \overline{\Omega^\varepsilon}$, or possibly only for almost all $x^\varepsilon \in \Omega^\varepsilon$ when the covariant components $v^\varepsilon_i$ are required to belong to the Sobolev space $H^1(\Omega^\varepsilon)$ as in Theorem \ref{t:2} below.

It will of course be assumed that the reference configuration satisfies the confinement condition, i.e., that
$$
\bm{\Theta}^\varepsilon(\overline{\Omega^\varepsilon}) \subset \mathbb{H}.
$$

It is to be emphasised that the above confinement condition \emph{definitely departs} from the usual \emph{Signorini condition} favoured by most authors, who usually require that only the points of the undeformed and deformed ``lower face'' $\omega \times \{-\varepsilon\}$ of the reference configuration satisfy the confinement condition (see, e.g., \cite{LegMia2008}, \cite{LegMia2018}, \cite{Rodri2018}). Clearly, the confinement condition considered in the present paper, which is inspired by the formulation proposed by Br\'ezis \& Stampacchia~\cite{BrezisStampacchia1968}, is more physically realistic, since a Signorini condition imposed only on the lower face of the reference configuration does not prevent -- at least mathematically -- other points of the deformed reference configuration to ``cross'' the plane  $\{x \in \mathbb{E}^3; \; \bm{Ox} \cdot \bm{q} = 0\}$ and then to end up on the ``other side'' of this plane.

Such a confinement condition renders the asymptotic analysis substantially more difficult, however, as the constraint now bears on a vector field, namely the displacement vector field of the reference configuration, \emph{instead of on only a single component} of this field.

The mathematical modelling of such an \emph{obstacle problem for a linearly elastic shell} is then straightforward. Indeed, \emph{apart from} the confinement condition, the rest, i.e., the \emph{function space} and the expression of the \emph{quadratic energy} $J^\varepsilon$, are classical (see, e.g.~\cite{Ciarlet2000}). More specifically, let
$$
A^{ijk\ell,\varepsilon} := \lambda g^{ij,\varepsilon} g^{k\ell,\varepsilon} + \mu \left( g^{ik,\varepsilon} g^{j\ell,\varepsilon} + g^{i\ell,\varepsilon} g^{jk,\varepsilon}\right) =A^{jik\ell,\varepsilon} =  A^{k\ell ij,\varepsilon},
$$
denote the contravariant components of the \emph{elasticity tensor} of the elastic material constituting the shell. Then the unknown of the problem, which is the vector field $\bm{u}^\varepsilon = (u^\varepsilon_i)$ where the functions $u^\varepsilon_i : \overline{\Omega^\varepsilon} \to \mathbb{R}$ are the three covariant components of the unknown ``\emph{three-dimensional}'' displacement vector field $u^\varepsilon_i \bm{g}^{i, \varepsilon}$ of the reference configuration of the shell, should minimise the \emph{energy} $J^\varepsilon:\bm{H}^1(\Omega^\varepsilon) \to \mathbb{R}$ defined by
$$
J^\varepsilon(\bm{v}^\varepsilon) := \frac{1}{2} \int_{\Omega^\varepsilon} A^{ijk\ell,\varepsilon} e^\varepsilon_{k\|\ell}(\bm{v}^\varepsilon)e^\varepsilon_{i\|j}(\bm{v}^\varepsilon) \sqrt{g^\varepsilon} \dd x^\varepsilon - \int_{\Omega^\varepsilon} f^{i,\varepsilon} v^\varepsilon_i \sqrt{g^\varepsilon} \dd x^\varepsilon,
$$
for each $\bm{v}^\varepsilon = (v^\varepsilon_i) \in \bm{H}^1(\Omega^\varepsilon)$
over the \emph{set of admissible displacements} defined by:
$$
\bm{U}(\Omega^\varepsilon) := \{\bm{v}^\varepsilon = (v^\varepsilon_i) \in \bm{H}^1 (\Omega^\varepsilon) ; \bm{v}^\varepsilon = \textbf{0} \text{ on } \Gamma^\varepsilon_0 \textup{ and } (\bm{\Theta}^\varepsilon(x^\varepsilon) + v^\varepsilon_i (x^\varepsilon) \bm{g}^{i, \varepsilon} (x^\varepsilon)) \cdot \bm{q} \ge 0 \textup{ for a.a. } x^\varepsilon \in \Omega^\varepsilon\}.
$$

The solution to this \emph{minimisation problem} exists and is unique, and it can be also characterised as the unique solution of the following problem:

\begin{customprob}{$\mathcal{P}(\Omega^\varepsilon)$}\label{problem0}
	Find $\bm{u}^\varepsilon =(u_i^\varepsilon)\in \bm{U}(\Omega^\varepsilon)$ that satisfies the variational inequalities:
	$$
	\int_{\Omega^\varepsilon}
	A^{ijk\ell, \varepsilon} e^\varepsilon_{k\| \ell}  (\bm{u}^\varepsilon)
	\left( e^\varepsilon_{i\| j}  (\bm{v}^\varepsilon) -  e^\varepsilon_{i\| j}  (\bm{u}^\varepsilon)  \right) \sqrt{g^\varepsilon} \dd x^\varepsilon \ge \int_{\Omega^\varepsilon} f^{i , \varepsilon} (v^\varepsilon_i - u^\varepsilon_i)\sqrt{g^\varepsilon} \dd x^\varepsilon,
	$$
	for all $\bm{v}^\varepsilon = (v^\varepsilon_i) \in \bm{U}(\Omega^\varepsilon)$.
	\bqed	
\end{customprob}

The following result then follows from classical arguments (see, e.g., Theorem~4.8-2 in~\cite{Ciarlet2025}).
\begin{theorem} 
\label{t:2}
 The quadratic minimisation problem: Find a vector field $\bm{u}^\varepsilon \in \bm{U}(\Omega^\varepsilon)$ such that
$$
J^\varepsilon(\bm{u}^\varepsilon) = \inf_{\bm{v}^\varepsilon \in \bm{U} (\Omega^\varepsilon)} J^\varepsilon (\bm{v}^\varepsilon),
$$
has one and only one solution. Besides, $\bm{u}^\varepsilon$ is also the unique solution of Problem~\ref{problem0}.
\qed
\end{theorem}

Since $\bm{\theta}(\overline{\omega})\subset\bm{\Theta}^\varepsilon(\overline{\Omega^\varepsilon})$, it evidently follows that $\bm{\theta}(y)\cdot\bm{q}\ge 0$ for all $y\in \overline{\omega}$. But in fact, a stronger property holds (cf. Lemma~2.1 of~\cite{CiaMarPie2018}).

\begin{lemma}
\label{lem:1}
Let $\omega$ be a domain in $\mathbb{R}^2$, let $\bm{\theta} \in \mathcal{C}^1(\overline{\omega};\mathbb{E}^3)$ be an injective immersion, let $\bm{q} \in \mathbb{E}^3$ be a given unit-vector, and let $\varepsilon > 0$. Then the inclusion
$$
\bm{\Theta}^\varepsilon(\overline{\Omega^\varepsilon})\subset \mathbb{H} = \{x\in\mathbb{E}^3;\;\boldsymbol{Ox} \cdot \bm{q} \ge 0\}
$$
implies that:
$$
d:=\min_{y \in \overline{\omega}}(\bm{\theta}(y)\cdot\bm{q})>0.
$$
\qed
\end{lemma}

\section{The scaled three-dimensional problem for a family of generalised membrane shells} 
\label{Sec:3}

In section~\ref{Sec:2}, we considered an obstacle problem for ``general'' linearly elastic shells. From now on, we will restrict ourselves  to a specific class of shells. Prior to giving the definition of linearly elastic generalised membrane shell, we recall the definitions of the remaining two categories of linearly elastic shells, namely, linearly elastic flexural shells, and linearly elastic elliptic membrane shells.

A shell is classified as a linearly elastic flexural shell if it satisfies two additional conditions: First, $\emptyset \neq \gamma_0 \subset \gamma$, meaning that a homogeneous boundary condition is imposed over a non-zero area portion of the entire lateral face $\gamma_0 \times \left[-\varepsilon,\varepsilon\right]$ of the shell; and second, the space
\begin{equation*}
	\bm{V}_F(\omega):=\{\bm{\eta}=(\eta_i) \in H^1(\omega)\times H^1(\omega)\times H^2(\omega); \gamma_{\alpha\beta}(\bm{\eta})=0 \textup{ in }\omega
	\textup{ and }\eta_i=\partial_{\nu}\eta_3=0 \textup{ on }\gamma_0\},
\end{equation*}
contains non-zero functions, i.e., $\bm{V}_F(\omega)\neq\{\bm{0}\}$.

Linearly elastic shells are classified as \emph{membranes}, provided $\bm{V}_F(\omega)=\{\bm{0}\}$. Among linearly elastic membrane shells, we can distinguish two sub-families: elliptic membrane shells and generalised membrane shells.

A linearly elastic membrane shell is classified as \emph{elliptic membrane shell} if the following two additional assumptions are satisfied: \emph{first}, $\gamma_0 = \gamma$, i.e., the homogeneous boundary condition of place is imposed over the \emph{entire lateral face} $\gamma \times \left[-\varepsilon,\varepsilon\right]$ of the shell, and \emph{second}, its middle surface $\bm{\theta}(\overline{\omega})$ is \emph{elliptic}, according to the definition given in section~\ref{Sec:1}.

We are now in position to state the definition of linearly elastic generalised membranes (cf., e.g., Section~5.1 of~\cite{Ciarlet2000}). A linearly elastic shell is said to be a \emph{linearly elastic generalised membrane shell} if the following \emph{two additional assumptions} are simultaneously satisfied: \emph{First}, $\text{length}\,\gamma_0>0$ (an assumption that is satisfied if $\gamma_0$ is a non-empty relatively open subset of $\gamma$, as assumed here). \emph{Second}, the \emph{space of admissible linearised inextensional displacements} $\bm{V}_F(\omega)$ \emph{does not contain any non-zero vectors}, i.e.,
$$
\bm{V}_F(\omega) = \{\bm{0}\},
$$
\emph{but, the shell is not a linearly elastic elliptic membrane shell} (note that, in light of the results in~\cite{Ciarlet2000}, the space $\bm{V}_F(\omega)$ indeed reduces to $\{\bm{0}\}$ if the shell is a linearly elastic elliptic membrane one).
The \emph{second} condition in the definition of a linearly elastic generalised membrane shell, viz.,
$$
\bm{V}_F(\omega)=\{\bm{0}\},
$$
is equivalent to stating that the \emph{semi-norm $\left|\cdot\right|_\omega^M$ defined by
	$$
	|\bm{\eta}|_\omega^M:=\left\{\sum_{\alpha,\beta}|\gamma_{\alpha\beta}(\bm{\eta})|_{L^2(\omega)}^2\right\}^{1/2},
	$$
	for each $\bm{\eta}=(\eta_i) \in H^1(\omega) \times H^1(\omega) \times L^2(\omega)$ becomes a norm over the space} 
$$
\bm{V}_K(\omega)=\{\bm{\eta}=(\eta_i) \in H^1(\omega)\times H^1(\omega) \times H^2(\omega); \eta_i=\partial_{\nu}\eta_3=0 \textup{ on }\gamma_0\}.
$$

Generalised membrane shells are themselves subdivided into \emph{two categories} described in terms of the spaces:
\begin{align*}
	\bm{V}(\omega)&:=\{\bm{\eta}=(\eta_i) \in \bm{H}^1(\omega);\bm{\eta}={\bm{0}} \textup{ on }\gamma_0\},\\
	\bm{V}_{0}(\omega)&:=\{\bm{\eta}=(\eta_i)\in\bm{V}(\omega);\gamma_{\alpha\beta}(\bm{\eta})=0 \textup{ in }\omega\}.
\end{align*}

A linearly elastic generalised membrane shell is \emph{``of the first kind''} if 
$$
\bm{V}_{0}(\omega)=\{\bm{0}\},
$$
or, equivalently, \emph{if the semi-norm} $\left|\cdot\right|_\omega^M$ \emph{is already a norm over the space} $\bm{V}(\omega)$ (hence, \emph{a fortiori}, over the space $\bm{V}_K(\omega) \subset \bm{V}(\omega)$).
Otherwise, i.e., if 
$$
{\bm{V}_F(\omega)=\{\bm{0}\}} \quad\textup{but}\quad {\bm{V}_{0}(\omega)\neq \{\bm{0}\}},
$$
or, equivalently, if the semi-norm $\left|\cdot\right|_\omega^M$ is a norm over $\bm{V}_K(\omega)$ but \emph{not} over $\bm{V}(\omega)$, the linearly elastic shell is a linearly elastic generalised membrane shell \emph{``of the second kind''}.
In this paper we shall only consider linearly elastic generalised membrane shells ``of the first kind'' as they are the most frequently encountered in the practice.

In this paper, we consider the \emph{obstacle problem} as defined in section~\ref{Sec:2} \emph{for a family of linearly elastic generalised membrane shells of the first kind}, all sharing the \emph{same middle surface} and whose thickness $2 \varepsilon > 0$ is considered as a ``small'' parameter approaching zero. Our objective then consists in performing an \emph{asymptotic analysis as} $\varepsilon \to 0$, so as to seek whether we can identify a \emph{two-dimensional limit model}. To this end, we shall resort to a (by now classical) approach due to Ciarlet \& Lods in~\cite{CiaLods1996d} (see also Theorem~5.6-1 of~\cite{Ciarlet2000}). To begin with, we ``scale'' Problem~\ref{problem0}, over a \emph{fixed domain} $\Omega$, i.e., independent of $\varepsilon$, using appropriate \emph{scalings on the unknowns} and \emph{assumptions on the data}.
Define
$$
\Omega := \omega \times (-1,1)\quad\textup{ and }\quad\Gamma_0:=\gamma_0\times [-1,1],
$$
let $x=(x_i)$ denote a generic point in the set $\overline{\Omega}$, and let $\partial_i:=\partial/\partial x_i$. With each point $x=(x_i)\in\overline{\Omega}$, we associate  the point $x^\varepsilon=(x^\varepsilon_i)$ defined by
$$
x^\varepsilon_\alpha:=x_\alpha=y_\alpha \quad\text{ and }\quad x^\varepsilon_3:=\varepsilon x_3,
$$
so that $\partial^\varepsilon_\alpha=\partial_\alpha$ and $\partial^\varepsilon_3=\frac{1}{\varepsilon} \partial_3$. To the unknown $\bm{u}^\varepsilon=(u^\varepsilon_i)$ and to the vector fields $\bm{v}^\varepsilon=(v^\varepsilon_i)$ appearing in the formulation of Problem~\ref{problem0} we then associate the \emph{scaled unknown} $\bm{u}(\varepsilon)=(u_i(\varepsilon))$ and the \emph{scaled vector fields} $\bm{v}=(v_i)$ by letting
$$
u_i(\varepsilon)(x):=u^\varepsilon_i(x^\varepsilon)\text{ and }v_i(x):=v^\varepsilon_i(x^\varepsilon),
$$
at each $x\in \overline{\Omega}$. Finally, we \emph{assume} that there exist functions $f^i \in L^2(\Omega)$ \emph{independent on} $\varepsilon$ such that the following \emph{assumptions on the applied body forces} hold:
\begin{equation*}
f^{i, \varepsilon} (x^\varepsilon) = f^i(x),\quad \text{ for a.a. } x \in \Omega.
\end{equation*}

Note that the independence on $\varepsilon$ of the Lam\'{e} constants assumed in section~\ref{Sec:2} in the formulation of Problem~\ref{problem0} implicitly constitutes another \emph{assumption on the data}.

In view of the proposed scaling, we define the following ``scaled'' counterparts of the various functions and vector fields introduced in section~\ref{Sec:1}:
\begin{align*}
	\bm{g}^i(\varepsilon)(x)&:=\bm{g}^{i,\varepsilon}(x^\varepsilon),\quad \textup{ for all } x\in \overline{\Omega},\\
	g(\varepsilon)(x) &:= g^\varepsilon(x^\varepsilon)\text{ and } A^{ijk\ell}(\varepsilon)(x):=A^{ijk\ell, \varepsilon}(x^\varepsilon),\quad\textup{ for all } x\in \overline{\Omega},\\
	e_{\alpha\|\beta}(\varepsilon;\bm{v})&:=\frac{1}{2}(\partial_\beta v_\alpha + \partial_\alpha v_\beta)-\Gamma^k_{\alpha\beta}(\varepsilon) v_k = e_{\beta\|\alpha}(\varepsilon;\bm{v}),\\
	e_{\alpha\|3}(\varepsilon;\bm{v})=e_{3\|\alpha}(\varepsilon;\bm{v}) &:=\frac{1}{2}\left(\frac{1}{\varepsilon}\partial_3 v_\alpha+\partial_\alpha v_3\right)-\Gamma^\sigma_{\alpha3}(\varepsilon) v_\sigma,\\
	e_{3\| 3}(\varepsilon;\bm{v}) &:= \frac{1}{\varepsilon}\partial_3 v_3,
\end{align*}
where
$$
\Gamma^p_{ij}(\varepsilon)(x):=\Gamma^{p,\varepsilon}_{ij}(x^\varepsilon),\quad \textup{ for all } x\in\overline{\Omega}.
$$

Define the space
$$
\bm{V}(\Omega):=\{\bm{v}=(v_i)\in\bm{H}^1(\Omega);\bm{v}=\bm{0}\textup{ on }\gamma_0\times (-1,1)\},
$$
and, for each $\varepsilon > 0$,  define the set:
\begin{align*}
	\bm{U}(\varepsilon;\Omega)&:=\{\bm{v}=(v_i)\in\bm{V}(\Omega);\big(\bm{\theta}(y)+\varepsilon x_3 \bm{a}_3(y)+v_i(x)\bm{g}^i(\varepsilon)(x)\big)\cdot\bm{q}\ge 0 \textup{ for a.a. } x=(y,x_3)\in\Omega\}.
\end{align*}

We now introduce the ``scaled'' version of Problem~\ref{problem0}, denoted in what follows by $\mathcal{P}(\varepsilon;\Omega)$:

\begin{customprob}{$\mathcal{P}(\varepsilon; \Omega)$}
\label{problem0scaled}
Find $\bm{u}(\varepsilon)=(u_i(\varepsilon)) \in \bm{U}(\varepsilon; \Omega)$ that satisfies the variational inequalities:
\begin{align*}
	\int_{\Omega} A^{ijk\ell}(\varepsilon) e_{k\| \ell} (\varepsilon; \bm{u}(\varepsilon)) \left(e_{i\| j}(\varepsilon; \bm{v}) - e_{i\|j} (\varepsilon; \bm{u}(\varepsilon))\right) \sqrt{g(\varepsilon)} \dd x
	\ge \int_{\Omega} f^i (v_i - u_i(\varepsilon)) \sqrt{g(\varepsilon)} \dd x,
\end{align*}
for all $\bm{v}=(v_i) \in \bm{U}(\varepsilon;\Omega)$.
\bqed
\end{customprob}

Then, the following existence result holds:

\begin{theorem} \label{t:3}
The scaled unknown $\bm{u}(\varepsilon)$ is the unique solution of the variational Problem~\ref{problem0scaled}.
\end{theorem}
\begin{proof}
The variational Problem~\ref{problem0scaled} simply constitutes a re-writing of the variational Problem~\ref{problem0}, this time in terms of the scaled unknown $\bm{u}(\varepsilon)$, of the vector fields $\bm{v}$, and of the functions $f^i$, which are now all defined over the domain $\Omega$. Then the assertion follows from this observation.
\end{proof}

The functions $e_{i\|j}(\varepsilon;\bm{v})$ appearing in Problem~\ref{problem0scaled} are called the \emph{scaled linearised strains in curvilinear coordinates} associated with the scaled displacement vector field $v_i \bm{g}^i(\varepsilon)$.

For later purposes (like in Lemma~\ref{lem:2} below), we also let
$$
\bm{g}_i(\varepsilon)(x):=\bm{g}^\varepsilon_i(x^\varepsilon)\textup{ at each }x\in\overline{\Omega}.
$$

It is immediately verified (cf., e.g., \cite{Ciarlet2000}) that \emph{other} assumptions on the data are possible that would give rise to the \emph{same} problem over the fixed domain $\Omega$. For instance, should the Lam\'{e} constants (now denoted) $\lambda^\varepsilon$ and $\mu^\varepsilon$ appearing in Problem~\ref{problem0} be of the form $\lambda^\varepsilon = \varepsilon^t \lambda$ and $\mu^\varepsilon = \varepsilon^t \mu$, where $\lambda \ge 0$ and $\mu$ are constants independent of $\varepsilon$ and $t$ is an arbitrary real number, the \emph{same} Problem~\ref{problem0scaled} arises if we assume that the components of the applied body force density are now of the form
$$
f^{i, \varepsilon} (x^\varepsilon) = \varepsilon^{t} f^i(x), \text{ for a.a. } x\in \Omega,
$$
where the functions $f^i \in L^2(\Omega)$ are independent of $\varepsilon$.

The next lemma assembles various results regarding the asymptotic behaviour as $\varepsilon \to 0$ of the functions and vector fields appearing in the formulation of Problem~\ref{problem0scaled}. These properties will be repeatedly used in the proof of the convergence theorem (Theorem~\ref{asymptotics}).

In the statement of the next lemma (for the proof, see Theorems~3.3-1 and~3.3-2 of~\cite{Ciarlet2000}), the notation ``$\mathcal{O}(\varepsilon)$'', or ``$\mathcal{O}(\varepsilon^2)$'', stands for a remainder that is of order $\varepsilon$, or $\varepsilon^2$, with respect to the sup-norm over the set $\overline{\Omega}$, and any function, or vector-valued function, of the variable $y \in \overline{\omega}$, such as $a^{\alpha\beta}, b_{\alpha\beta}, \bm{a}^i$, etc. (all these are defined in section~\ref{Sec:1}) is identified with the function, or vector-valued function, of $x = (y, x_3) \in \overline{\Omega} = \overline{\omega} \times [-1,1]$ that takes the same value at $x_3 = 0$ and is independent of $x_3 \in \left[-1, 1\right]$; for brevity, this extension from $\overline{\omega} $ to $\overline{\Omega}$ is designated with the same notation. In what follows, the number $\varepsilon_0$ is defined as in Theorem~3.1-1 of~\cite{Ciarlet2000}.

\begin{lemma} \label{lem:2}
Let $\bm{\theta} \in \mathcal{C}^3(\overline{\omega};\mathbb{E}^3)$ be an injective immersion.
Then, the functions $A^{ijk\ell}(\varepsilon) = A^{jik\ell}(\varepsilon) = A^{k\ell ij}(\varepsilon)$ satisfy the following properties:
$$
A^{ijk\ell}(\varepsilon)=A^{ijk\ell}(0)+\mathcal{O}(\varepsilon) , \quad A^{\alpha\beta\sigma3}(\varepsilon)=A^{\alpha333}(\varepsilon)=0,
$$
for all $0<\varepsilon\le\varepsilon_0$, where
\begin{align*}
A^{\alpha\beta\sigma\tau}(0)&=\lambda a^{\alpha\beta}a^{\sigma\tau}+\mu(a^{\alpha\sigma}a^{\beta\tau}+a^{\alpha\tau}a^{\beta\sigma}),\\
A^{\alpha\beta 33} (0) &= \lambda a^{\alpha\beta}, \quad A^{\alpha 3\sigma 3}(0)=\mu a^{\alpha\sigma}, \quad A^{3333}(0)=\lambda+2\mu,
\end{align*}
and there exists a constant $C_e > 0$ such that
$$
\sum_{i,j}|t_{ij}|^2\le C_e A^{ijk\ell}(\varepsilon)(x) t_{k\ell} t_{ij},
$$
for all $0<\varepsilon\le\varepsilon_0$, all $x \in \overline{\Omega}$, and all symmetric matrices $(t_{ij})$.

The functions $\Gamma^p_{ij} (\varepsilon)$ and $g(\varepsilon)$ satisfy the following properties:
\begin{align*}
\Gamma^\sigma_{\alpha\beta}(\varepsilon) &= \Gamma^\sigma_{\alpha\beta} - \varepsilon x_3 (\partial_\alpha b^\sigma_\beta + \Gamma^\sigma_{\alpha\tau} b^\tau_\beta - \Gamma^\tau_{\alpha\beta} b^\sigma_\tau) + \mathcal{O}(\varepsilon^2),\\
\Gamma^3_{\alpha\beta}(\varepsilon) &= b_{\alpha\beta} - \varepsilon x_3 b^\sigma_\alpha b_{\sigma \beta},\\
\partial_3 \Gamma^p_{\alpha\beta}(\varepsilon) &= \mathcal{O}(\varepsilon),\\
\Gamma^\sigma_{\alpha 3}(\varepsilon) &= - b^\sigma_\alpha - \varepsilon x_3 b^\tau_\alpha b^\sigma_\tau + \mathcal{O}(\varepsilon^2),\\
\Gamma^3_{\alpha 3}(\varepsilon) &= \Gamma^p_{33}(\varepsilon) =0,\\
g(\varepsilon) &= a + \mathcal{O}(\varepsilon),
\end{align*}
for all $0<\varepsilon \le \varepsilon_0$ and all $x \in \overline{\Omega}$. In particular then, there exist constants $g_0$ and $g_1$ such that
$$
0 < g_0 \le g(\varepsilon)(x) \le g_1,\quad \textup{ for all }0<\varepsilon \le \varepsilon_0 \textup{ and all }x \in \overline{\Omega}.
$$

The vector fields $\bm{g}_i (\varepsilon)$ and $\bm{g}^j(\varepsilon)$ satisfy the following properties:
\begin{align*}
\bm{g}_\alpha(\varepsilon) &= \bm{a}_\alpha - \varepsilon x_3 b^\sigma_\alpha \bm{a}_\sigma,\\
\bm{g}^\alpha(\varepsilon) &= \bm{a}^\alpha + \varepsilon x_3 b^\alpha_\sigma \bm{a}^\sigma + \mathcal{O}(\varepsilon^2),\\
\bm{g}_3(\varepsilon)&=\bm{a}_3=\bm{a}^3=\bm{g}^3(\varepsilon).
\end{align*}
\qed
\end{lemma}

We recall (cf., e.g., \cite{Ciarlet2000}) that the various relations and estimates in Lemma~\ref{lem:2} hold for \emph{any} family of linearly elastic shell. Define the sets $\Gamma^\varepsilon_{+}:=\omega\times\{\varepsilon\}$ and $\Gamma^\varepsilon_{-}:=\omega\times\{-\varepsilon\}$.

The next result relates the confinement condition we are considering in this paper with the classical \emph{Signorini condition} (cf., e.g., Chapter~6 in~\cite{KikuchiOden1988}), according to which:
\begin{equation*}
	(\bm{\Theta}^\varepsilon(y,x_3^\varepsilon) +v_i^\varepsilon(y,x_3^\varepsilon)\bm{g}^{i,\varepsilon}(y,x_3^\varepsilon)) \cdot\bm{q} \ge 0, \quad\textup{ for a.a. }(y,x_3^\varepsilon) \in \Gamma^\varepsilon_+ \cup \Gamma^\varepsilon_{-}.
\end{equation*}

In order to establish this result, we will need to resort to the \emph{Kirchhoff-Love assumptions} (cf., e.g., page~336 in~\cite{Ciarlet2000}), according to which \emph{any point on a normal to the middle surface remains on the normal to the deformed middle surface after the deformation has taken place, and the distance between such a point and the middle surface remains constant before and after the deformation}. 

Let us recall that the widely accepted Kirchhoff-Love assumptions, which are of \emph{geometrical nature} of the shell, constitutes one of the pillars that led Koiter to the derivation of his model~\cite{Koiter,Koiter1970}.

To this end, we restrict ourselves to considering \emph{smooth} displacement fields. This choice is justified, as it is related to the fact that if the norms $\|v^\varepsilon_i\|_{\mathcal{C}^1(\overline{\Omega^\varepsilon})}$ are small enough, the mapping $(\bm{\Theta}^\varepsilon + v^\varepsilon_i \bm{g}^{i,\varepsilon})$ is also an injective immersion (cf. section~\ref{Sec:1}).

\begin{lemma}
	\label{Signorini}
	Let $\bm{\theta}\in\mathcal{C}^3(\overline{\omega};\mathbb{E}^3)$ be an injective immersion and let $\varepsilon>0$ be small enough so that $\bm{\Theta}^\varepsilon\in\mathcal{C}^3(\overline{\Omega^\varepsilon};\mathbb{E}^3)$ is also an injective immersion.
	Let $\bm{v}^\varepsilon \in \bm{\mathcal{C}}^1(\overline{\Omega^\varepsilon})$ be such that $(\bm{\Theta}^\varepsilon+v_i^\varepsilon\bm{g}^{i,\varepsilon})$ is an injective immersion.
	
	Under the validity of the Kirchhoff-Love assumptions, the following conditions are equivalent:
	\begin{itemize}
		\item[(i)] $\bm{v}^\varepsilon \in \bm{U}(\Omega^\varepsilon)$;
		\item[(ii)] $\bm{v}^\varepsilon=\bm{0}$ on $\Gamma_0^\varepsilon$ and $(\bm{\Theta}^\varepsilon(y,x_3^\varepsilon) +v_i^\varepsilon(y,x_3^\varepsilon)\bm{g}^{i,\varepsilon}(y,x_3^\varepsilon)) \cdot\bm{q}\ge 0$, for all $(y,x_3^\varepsilon) \in \Gamma^\varepsilon_+ \cup \Gamma^\varepsilon_{-}$.
	\end{itemize}
\end{lemma}
\begin{proof}
	Let us first show that (i) implies (ii).
	If a point $(\hat{y},\hat{x}_3^\varepsilon) \in \Gamma^\varepsilon_+ \cup \Gamma^\varepsilon_{-}$ is such that
	$$
	(\bm{\theta}(\hat{y})+\hat{x}_3^\varepsilon \bm{a}^3(\hat{y}) +v_i^\varepsilon(\hat{y},\hat{x}_3^\varepsilon)\bm{g}^{i,\varepsilon}(\hat{y},\hat{x}_3^\varepsilon)) \cdot\bm{q} <0,
	$$
	then there exists a ball $B_\rho$ centred at the given point $(\hat{y},\hat{x}_3^\varepsilon)$ with small enough radius $\rho>0$ such that:
	$$
	(\bm{\theta}(y)+ x_3^\varepsilon \bm{a}^3(y) +v_i^\varepsilon(y,x_3^\varepsilon)\bm{g}^{i,\varepsilon}(y,x_3^\varepsilon)) \cdot\bm{q} <0,
	$$
	for all $(y,x_3^\varepsilon) \in \overline{\Omega^\varepsilon}\cap B_\rho$. The latter contradicts~(i), namely, the fact that $\bm{v}^\varepsilon \in \bm{U}(\Omega^\varepsilon)$. Therefore~(ii) holds.
	
	Let us next show that the converse implication holds. Thanks to the assumed Kirchhoff-Love assumptions, for each $y \in\overline{\omega}$, if the closed segment $\{\bm{\Theta}^\varepsilon(y,x_3^\varepsilon), -\varepsilon \le x_3^\varepsilon \le \varepsilon\}$ is orthogonal to the plane that is tangent to $\bm{\theta}(\overline{\omega})$ at $\bm{\theta}(y)$, then the segment
	$$
	(1-r)\left(\bm{\theta}(y) -\varepsilon\bm{a}^3(y)+v_i^\varepsilon(y,-\varepsilon) \bm{g}^{i,\varepsilon}(y,-\varepsilon)\right)
	+r\left(\bm{\theta}(y) +\varepsilon\bm{a}^3(y)+v_i^\varepsilon(y,\varepsilon) \bm{g}^{i,\varepsilon}(y,\varepsilon)\right),
	$$
	where $0\le r \le 1$ is also orthogonal to the tangent plane to the deformed middle surface at the point $(\bm{\theta}(y)+v_i^\varepsilon(y,0)\bm{g}^{i,\varepsilon}(y,0))$.
	
	Since we only consider displacements for which the corresponding deformation is an injective immersion, we conclude that for all $y_1, y_2 \in \overline{\omega}$, with $y_1 \neq y_2$,
	$$
	\left(\bigcup_{0\le r\le 1} s_{y_1}^\varepsilon(r)\right) \cap \left(\bigcup_{0\le r\le 1} s_{y_2}^\varepsilon(r)\right) = \emptyset,
	$$
	i.e., the segments corresponding to distinct points in $\overline{\omega}$ do not intersect. Therefore, each point in the deformed reference configuration of the shell belongs to one and only one segment orthogonal to the middle surface of the deformed shell.
	
	Since (ii) is assumed to hold, it follows that:
	\begin{equation*}
		\label{confsegm}
		\left[(1-r)\left(\bm{\theta}(y) -\varepsilon\bm{a}^3(y)+v_i^\varepsilon(y,-\varepsilon) \bm{g}^{i,\varepsilon}(y,-\varepsilon)\right)
		+r\left(\bm{\theta}(y) +\varepsilon\bm{a}^3(y)+v_i^\varepsilon(y,\varepsilon) \bm{g}^{i,\varepsilon}(y,\varepsilon)\right)\right] \cdot \bm{q} \ge 0,
	\end{equation*}
	for all $0\le r \le 1$, for all $y \in \overline{\omega}$. Noting that $y\in\overline{\omega}$, we infer that $\bm{v}^\varepsilon \in \bm{U}(\Omega^\varepsilon)$.
\end{proof}

Note that since $\Gamma^\varepsilon_+ \cup \Gamma^\varepsilon_{-}$ clearly contains the portion of the boundary of the three-dimensional elastic shell that engages contact with the obstacle, the confinement condition considered in this paper implies the validity of the classical Signorini condition.

When one considers a family of linearly elastic \emph{generalised membrane shells} (not necessarily of the first kind) whose thickness $2 \varepsilon$ approaches zero, a specific \emph{Korn's inequality in curvilinear coordinates} (cf., e.g., Theorem~4.1 of~\cite{CiaLodsMia1996} or Theorem~5.3-1 of~\cite{Ciarlet2000}) holds over the \emph{fixed} domain $\Omega = \omega \times (-1,1)$, according to the following theorem. That the constant $C_1$ that appears in this inequality is \emph{independent of} $\varepsilon > 0$ plays a key role in the asymptotic analysis of such a family (see part~(i) of the proof of Theorem~\ref{asymptotics}). A similar inequality of Korn's type was established by Kohn \& Vogelius~\cite{KV1985}.

\begin{theorem}\label{t:4}
Let $\bm{\theta} \in \mathcal{C}^3(\overline{\omega};\mathbb{E}^3)$ be an immersion.
Recall that $\textup{length }\gamma_0>0$.
Define the space
$$
\bm{V} (\Omega) := \{\bm{v} = (v_i) \in \bm{H}^1(\Omega) ; \; \bm{v} = \bm{0} \textup{ on }  \gamma_0 \times [-1,1]\}.
$$
Then there exist constants $\varepsilon_0 > 0$ and $C_1 > 0$ such that
$$
  \left\{ \sum_i \left\| v_i \right\|^2_{H^1(\Omega)}\right\}^{1/2} \le \dfrac{C_1}{\varepsilon} \left\{ \sum_{i,j} \left\| e_{i\|j} (\varepsilon; \bm{v}) \right\|^2_{L^2(\Omega)} \right\}^{1/2}
$$
for all $0 < \varepsilon \leq \varepsilon_0$ and all $\bm{v}=(v_i) \in \bm{V} (\Omega)$.
\qed
\end{theorem}

In order to derive \emph{a priori estimates} for the family $\{\bm{u}(\varepsilon)\}_{\varepsilon>0}$ of solutions for Problem~\ref{problem0}, the applied body forces need to be \emph{admissible} in the sense of~\cite{Ciarlet2000}. By this we mean that, for a given $\varepsilon>0$, there exist functions $F^{ij}(\varepsilon)=F^{ji}(\varepsilon) \in L^2(\Omega)$ such that:
\begin{equation}
	\label{adm-1}
\begin{cases}
&\int_{\Omega} f^i v_i \sqrt{g(\varepsilon)} \dd x = \int_{\Omega} F^{ij}(\varepsilon) e_{i\|j}(\varepsilon;\bm{v}) \sqrt{g(\varepsilon)} \dd x,\quad \textup{ for all }\bm{v} = (v_i) \in \bm{V}(\Omega) \textup{ and all } 0<\varepsilon\le \varepsilon_0,\\
\\
&F^{ij}(\varepsilon) \to F^{ij} \quad\textup{ in }L^2(\Omega) \textup{ as }\varepsilon\to 0.
\end{cases}
\end{equation}

The justification of this definition can be found in, for instance, Section~5.5 of~\cite{Ciarlet2000}. If the applied body forces are admissible, there exists a constant $\kappa_0>0$ independent of $\varepsilon$ such that
\begin{equation}
	\label{adm-2}
\left|\int_{\Omega} F^{ij}(\varepsilon) e_{i\|j}(\varepsilon;\bm{v}) \sqrt{g(\varepsilon)} \dd x\right| \le \kappa_0\sqrt{g_1} \left\{\sum_{i,j}\left\|e_{i\|j}(\varepsilon;\bm{v})\right\|_{L^2(\Omega)}^2\right\}^{1/2},
\end{equation}
for all $0<\varepsilon<\varepsilon_0$ and for all $\bm{v}\in\bm{V}(\Omega)$.

The difficulty in the study of linearly elastic generalised membranes lies in the fact that the asymptotic analysis as $\varepsilon \to 0$ involves, \emph{already in the obstacle free} case, certain \emph{abstract completions}. More precisely, for linearly elastic generalised membrane shells of the first kind, the results of the asymptotic analysis are posed (cf., e.g., Theorem~5.6-1 of~\cite{Ciarlet2000}) over the space
\begin{equation*}
	\bm{V}_M^\sharp(\omega):=\overline{\bm{V}(\omega)}^{|\cdot|_\omega^M},
\end{equation*}
whereas, for linearly elastic generalised membrane shells of the second kind, the results of the asymptotic analysis are posed (cf., e.g., Theorem~5.6-2 of~\cite{Ciarlet2000}) over the space
\begin{equation*}
	\dot{\bm{V}}_M^\sharp(\omega):=\overline{\bm{V}(\omega)/\bm{V}_0(\omega)}^{|\cdot|_\omega^M}.
\end{equation*}

In what follows, given a function $v \in L^2(\Omega)$, we denote by $\bar{v}$ the average of $v$ with respect to the transverse variable, i.e., 
\begin{equation*}
	\bar{v}:=\dfrac{1}{2}\int_{-1}^{1} v(\cdot,x_3) \dd x_3,
\end{equation*}
and we recall that the properties for averages with respect to the transverse variable are established in Section~4.2 of~\cite{Ciarlet2000}. In the same spirit, we denote the average with respect to the transverse variable of a vector field $\bm{v}=(v_i) \in \bm{L}^2(\Omega)$ by $\bar{\bm{v}}$, defined component-wise by $\bar{\bm{v}}:=\left(\bar{v}_i\right)$.

To carry out the asymptotic analysis for a family of linearly elastic generalised membrane shells, we shall introduce the space
\begin{equation*}
	\bm{V}_0(\Omega):=\{\bm{v} \in \bm{H}^1(\Omega);\bm{v}=\bm{0} \textup{ on }\Gamma_0, \partial_3\bm{v}=\bm{0} \textup{ in }\Omega, \textup{ and } \gamma_{\alpha\beta}(\bar{\bm{v}})=0 \textup{ in }\omega\},
\end{equation*}
which is the ``three-dimensional analogue'' of the space $\bm{V}_0(\omega)$ introduced beforehand. We also define the semi-norm $|\cdot|_\Omega^M$ by:
\begin{equation*}
	|\bm{v}|_\Omega^M:=\left\{\|\partial_3\bm{v}\|_{L^2(\Omega)}^2+\left(|\bar{\bm{v}}|_\omega^M\right)^2\right\}^{1/2},\quad\textup{ for all } \bm{v} \in \bm{V}(\Omega).
\end{equation*}

We note in passing that $\bm{V}_0(\omega) =\{\bm{0}\}$ if and only if $\bm{V}_0(\Omega)=\{\bm{0}\}$, or, equivalently, that $|\cdot|_\Omega^M$ is a norm over the space $\bm{V}(\Omega)$ if and only if $|\cdot|_\omega^M$ is a norm over the space $\bm{V}(\omega)$. In other words, a linearly elastic generalised membrane shell is of the first kind if and only if $\bm{V}_0(\Omega)=\{\bm{0}\}$.

\section{Rigorous asymptotic analysis for a family of linearly elastic generalised membranes of the first kind}
\label{Sec:4}

In this section and in the next one, we restrict our attention to \emph{linearly elastic generalised membrane shells of the first kind}.

This section is devoted to showing, in the same spirit as~\cite{CiaLods1996d} (see also Theorem~5.6-1 of~\cite{Ciarlet2000}), that the solutions $\bm{u}(\varepsilon)$ of the (scaled) three-dimensional problems~\ref{problem0scaled} converge, as $\varepsilon$ approaches zero, to the solution of a two-dimensional limit variational problem, denoted by $\mathcal{P}_M^\sharp(\omega)$ in what follows.

We will show that the solutions $\bm{u}(\varepsilon)$ for Problem~\ref{problem0scaled} converge in the \emph{completion of the space} $\bm{V}(\Omega)$ \emph{with respect to the norm} $|\cdot|_\Omega^M$. As a result, the two-dimensional limit model that we shall recover will be an \emph{abstract variational inequality} whose solution lies in the abstract completion of a non-empty, closed, and convex set with respect to a certain norm.

We denote by $B_M^\sharp(\cdot,\cdot):\bm{V}_M^\sharp(\omega) \times \bm{V}_M^\sharp(\omega) \to \mathbb{R}$ the unique continuous extension to the space $\bm{V}_M^\sharp(\omega) \times \bm{V}_M^\sharp(\omega)$ of the bilinear form defining the energy of a linearly elastic elliptic membrane shells, denoted and defined by:
\begin{equation*}
(\bm{\eta},\bm{\xi}) \in \bm{V}(\omega) \times \bm{V}(\omega) \mapsto B_M(\bm{\eta},\bm{\xi}):=
\int_{\omega} a^{\alpha\beta\sigma\tau} \gamma_{\sigma\tau}(\bm{\eta}) \gamma_{\alpha\beta}(\bm{\xi}) \sqrt{a} \dd y.
\end{equation*}

For each $\varepsilon>0$, we define the linear mapping $L(\varepsilon):\bm{V}(\Omega) \to \mathbb{R}$ by:
\begin{equation*}
	L(\varepsilon)(\bm{v}):=\int_{\Omega} f^i v_i \sqrt{g(\varepsilon)} \dd x,\quad\textup{ for all }\bm{v}=(v_i)\in \bm{V}(\Omega).
\end{equation*}

The linear mapping $L_M^\sharp:\bm{V}_M^\sharp(\omega) \to \mathbb{R}$, whose definition will be introduced in the proof of the next theorem, accounts for the asymptotic behaviour of the admissible applied body forces $F^{ij}(\varepsilon)$ introduced beforehand.
Finally, we define the set
\begin{equation*}
\bm{U}(\omega):=\{\bm{\eta}=(\eta_i) \in \bm{H}^1(\omega); \eta_i=0 \textup{ on }\gamma_0 \textup{ and } (\bm{\theta}(y)+\eta_i(y)\bm{a}^i(y))\cdot\bm{q} \ge 0 \textup{ for a.a. }y\in \omega\},
\end{equation*}
and define the corresponding abstract completion with respect to the norm $|\cdot|_\omega^M$ as follows:
\begin{equation*}
	\bm{U}_M^\sharp(\omega):=\overline{\bm{U}(\omega)}^{|\cdot|_\omega^M}.
\end{equation*}

We are thus in a position to state the two-dimensional limit problem, denoted by~$\mathcal{P}_M^\sharp(\omega)$, as follows:

\begin{customprob}{$\mathcal{P}_M^\sharp(\omega)$}
\label{problemLim}
Find $\bm{\zeta} \in \bm{U}_M^\sharp(\omega)$ that satisfies the following variational inequalities:
\begin{equation*}
B_M^\sharp(\bm{\zeta},\bm{\eta}-\bm{\zeta}) \ge L_M^\sharp(\bm{\eta}-\bm{\zeta}),\quad\textup{ for all }\bm{\eta} \in \bm{U}_M^\sharp(\omega).
\end{equation*}
\bqed
\end{customprob}

By contrast with the cases where the linearly elastic shell under consideration is a linearly elastic elliptic membrane shells~\cite{CiaMarPie2018b,CiaMarPie2018} or a linearly elastic flexural shell~\cite{Pie2023}, the existence and uniqueness of solutions for Problem~\ref{problemLim} is not straightforward, as it relies on a preliminary estimate that we will establish in the forthcoming theorem, which constitutes the first new result contained in this paper.

First, we need a preparatory lemma, whose proof hinges on the Kirchhoff-Love assumptions, thanks to which it was possible to relate the confinement condition considered in this paper with the classical Signorini condition (cf. Lemma~\ref{Signorini}).
Since we consider infinitesimal displacements, it is licit to \emph{assume} that if
	\begin{equation}
		\label{dpcmp}
		\min_{y\in \overline{\omega}}(\bm{a}^3(y)\cdot\bm{q})>0,
	\end{equation}
	then there exists a constant $M>0$ independent of the choice of the admissible displacement $\bm{\eta}$ for which the following relation holds as well
	\begin{equation}
		\label{a3eta}
		\bm{a}_3(\bm{\eta})(y)\cdot\bm{q} \ge M,\quad\textup{ for a.a. } y\in\omega,
	\end{equation}
	where $\bm{a}_3(\bm{\eta})$ is the usual unit vector normal to the deformed middle surface corresponding to the displacement occurring at the middle surface incurs. We recall that the relation~\eqref{dpcmp} is one of the sufficient conditions ensuring the ``density property'' formulated in~\cite{CiaMarPie2018b,CiaMarPie2018}. We also observe that condition~\eqref{a3eta} holds when $\overline{\bm{u}(\varepsilon)} \in \bm{H}^3(\omega)$. However, in general, the latter higher regularity does not hold~\cite{Uraltseva1987}, and we thus assume that~\eqref{dpcmp} implies~\eqref{a3eta} since the admissible displacement fields are assumed to be ``small''.
	
	We also observe that assuming that the solution of Problem~\ref{problem0scaled} satisfies the Kirchhoff-Love assumptions is a judicious assumption in light of the error estimates established by Lods \& Mardare~\cite{LodsMar2000}.

\begin{lemma}
	\label{average-constraint}
	Let $\bm{u}(\varepsilon) \in \bm{U}(\varepsilon;\Omega)$ be the unique solution of Problem~\ref{problem0scaled}. 
	Assume that $\bm{u}(\varepsilon)$ satisfies the Kirchhoff-Love assumptions.
	Assume that:
	\begin{equation*}
		\min_{y\in \overline{\omega}}(\bm{a}^3(y)\cdot\bm{q})>0.
	\end{equation*}
	
	Then, it follows that $\overline{\bm{u}(\varepsilon)} \in \bm{U}(\omega)$.
\end{lemma}
\begin{proof}
	Thanks to the Kirchhoff-Love assumptions, the following formula holds (cf., e.g., page~372 in~\cite{Ciarlet2000}):
	\begin{equation*}
		u_i(\varepsilon)\bm{g}^i(\varepsilon)=u_i(\varepsilon)(\cdot,0)\bm{a}^i+\varepsilon x_3 \{\bm{a}_3(\bm{u}(\varepsilon)(\cdot,0))-\bm{a}_3\},\quad\textup{ a.e. in } \Omega.
	\end{equation*}
	
	The properties of the covariant and contravariant bases (cf., e.g., Lemma~\ref{lem:2} and~\cite{Ciarlet2005}) give at once that:
	\begin{equation*}
		\begin{aligned}
			u_\beta(\varepsilon)&=u_\beta(\varepsilon)(\cdot,0) -\varepsilon x_3 u_i(\varepsilon)(\cdot,0) b_{\beta\tau} a^{i\tau}+\varepsilon x_3 \bm{a}_3(\bm{u}(\varepsilon)(\cdot,0))\cdot\bm{a}_\beta\\
			&\quad-\varepsilon^2 x_3^2 b_{\beta\tau}\bm{a}_3(\bm{u}(\varepsilon)(\cdot,0))\cdot\bm{a}^\tau,\\
			u_3(\varepsilon)&=u_3(\varepsilon)(\cdot,0)+\varepsilon x_3 \bm{a}_3(\bm{u}(\varepsilon)(\cdot,0)) \cdot \bm{a}_3 -\varepsilon x_3.
		\end{aligned}
	\end{equation*}
	
	Therefore, the average of $\bm{u}(\varepsilon)$ satisfies the following relations:
	\begin{equation}
		\label{avg}
		\begin{aligned}
			\overline{u_\beta(\varepsilon)}&=u_\beta(\varepsilon)(\cdot,0)-\dfrac{1}{3}\varepsilon^2 b_{\beta\tau} \bm{a}_3(\bm{u}(\varepsilon)(\cdot,0)) \cdot\bm{a}^\tau,\\
			\overline{u_3(\varepsilon)}&=u_3(\varepsilon)(\cdot,0).
		\end{aligned}
	\end{equation}
	
	Note also that $\overline{\bm{u}(\varepsilon)}=\bm{0}$ on $\gamma_0$ since $\bm{u}(\varepsilon)\in\bm{U}(\varepsilon;\Omega)$ vanishes on $\Gamma_0$.
	In view of~\eqref{avg}, the following relations hold for a.a. $y\in\omega$:
	\begin{equation*}
		\begin{aligned}
			&\left(\bm{\theta}(y)+\overline{u_i(\varepsilon)}(y)\bm{a}^i(y)\right)\cdot\bm{q}=\left(\bm{\theta}(y)+u_i(\varepsilon)(y,0)\bm{a}^i\right)\cdot\bm{q}-\dfrac{1}{3}\varepsilon^2 b_{\beta\tau}(\bm{a}_3(\bm{u}(\varepsilon)(y,0)) \cdot\bm{a}^\tau) \bm{a}^\beta\cdot\bm{q}\\
			&=\left(\bm{\theta}(y)+u_i(\varepsilon)(y,0)\bm{a}^i-\varepsilon\bm{a}_3(\bm{u}(\varepsilon)(y,0))\right)\cdot\bm{q}+\varepsilon\bm{a}_3(\bm{u}(\varepsilon)(y,0))\cdot\bm{q}-\dfrac{1}{3}\varepsilon^2 b_{\beta\tau}(\bm{a}_3(\bm{u}(\varepsilon)(y,0)) \cdot\bm{a}^\tau) (\bm{a}^\beta\cdot\bm{q}).
		\end{aligned}
	\end{equation*}
	
	In light of the Kirchhoff-Love assumptions, we observe that the first term on the right-hand side of the last expression represents the deformation of the lower face of the shell at the given point $y \in \omega$, since all the points along a normal to the middle surface line remain on a normal to the deformed middle surface. Therefore, by Lemma~\ref{Signorini}:
	\begin{equation*}
		\left(\bm{\theta}(y)+u_i(\varepsilon)(y,0)\bm{a}^i-\varepsilon\bm{a}_3(\bm{u}(\varepsilon)(y,0))\right)\cdot\bm{q}\ge 0,\quad\textup{ for a.a. }y\in\omega.
	\end{equation*}
	
	In light of~\eqref{a3eta}, we infer that:
	\begin{equation*}
		\bm{a}_3(\bm{u}(\varepsilon)(y,0))\cdot\bm{q} \ge M, \quad\textup{ for a.a. }y\in\omega,
	\end{equation*}
	where the constant $M$ is independent of the displacement and is therefore independent of $\varepsilon$.
	
	Consequently, if we choose $\varepsilon>0$ sufficiently small, the smoothness of $b_{\beta\tau}$, the smoothness of $\bm{a}^\tau$, and the fact that $\bm{a}_3(\bm{u}(\varepsilon)(y,0))$ has unit norm (cf., e.g., formula~(1.3.4) in~\cite{BerBoi1982}) combined imply that
	\begin{equation*}
		\left(\bm{\theta}(y)+\overline{u_i(\varepsilon)}(y)\bm{a}^i(y)\right)\cdot\bm{q}\ge 0, \quad\textup{ for a.a. }y\in\omega,
	\end{equation*}
	as it was to be proved.
\end{proof}

We now establish the first main result of this paper.

\begin{theorem}
\label{asymptotics}
Assume that $\bm{\theta} \in \mathcal{C}^3(\overline{\omega};\mathbb{E}^3)$ is an injective immersion. Consider a family of linearly elastic generalised membrane shells of the first kind, each of which having the same middle surface $\bm{\theta}(\overline{\omega})$, and thickness $2\varepsilon$ approaching zero. Assume that each of the linearly elastic generalised membrane shells of the first kind under consideration is subjected to a boundary condition of place along a portion of its lateral face having $\bm{\theta}(\gamma_0)$ as its middle curve, and that each one of these shells is subjected to applied body forces that are admissible (cf. section~\ref{Sec:3}).

Assume that~\eqref{dpcmp} holds, namely:
\begin{equation*}
	\min_{y \in \overline{\omega}} (\bm{a}^3(y)\cdot\bm{q}) >0.
\end{equation*}

For each $0<\varepsilon\le \varepsilon_0$, let $\bm{u}(\varepsilon) \in \bm{U}(\Omega)$ denote the unique solution of Problem~\ref{problem0scaled}.
Assume that $\bm{u}(\varepsilon)$ satisfies the Kirchhoff-Love assumptions.
Define the space
\begin{equation*}
\bm{V}_M^\sharp(\Omega):=\overline{\bm{V}(\Omega)}^{|\cdot|_\Omega^M},
\end{equation*}
and define the set
\begin{equation*}
	\bm{U}_M^\sharp(\omega):=\overline{\bm{U}(\omega)}^{|\cdot|_\omega^M} \subset \bm{V}_M^\sharp(\omega).
\end{equation*}

Then there exist $\bm{u} \in \bm{V}_M^\sharp(\Omega)$ and $\bm{\zeta} \in \bm{U}_M^\sharp(\omega)$ such that:
\begin{equation*}
	\begin{aligned}
		\bm{u}(\varepsilon) &\to \bm{u},\quad\textup{ in }\bm{V}_M^\sharp(\Omega) \textup{ as } \varepsilon\to 0,\\
		\overline{\bm{u}(\varepsilon)} &\to \bm{\zeta},\quad\textup{ in }\bm{V}_M^\sharp(\omega) \textup{ as }\varepsilon\to0.
	\end{aligned}
\end{equation*}

Define
\begin{equation*}
	\begin{aligned}
		a^{\alpha\beta\sigma\tau}&:=\dfrac{4\lambda\mu}{\lambda+2\mu}a^{\alpha\beta}a^{\sigma\tau}+2\mu(a^{\alpha\sigma}a^{\beta\tau}+a^{\alpha\tau}a^{\beta\sigma}),\\
		\varphi^{\alpha\beta}&:=\int_{-1}^{1} \left\{F^{\alpha\beta}-\dfrac{\lambda}{\lambda+2\mu}a^{\alpha\beta}F^{33}\right\} \dd x_3 \in L^2(\omega),\\
		L_M(\bm{\eta})&:=\int_{\omega} \varphi^{\alpha\beta} \gamma_{\alpha\beta}(\bm{\eta})\sqrt{a} \dd y,\quad\textup{ for all }\bm{\eta} \in\bm{V}(\omega),
	\end{aligned}
\end{equation*}
where the functions $F^{ij} \in L^2(\Omega)$ are those entering the definition of admissible forces, and $B_M^\sharp$ and $L_M^\sharp$ denote the unique continuous extensions from $\bm{V}(\omega)$ to $\bm{V}_M^\sharp(\omega)$ of the bilinear form $B_M(\cdot,\cdot)$ and of the linear form $L_M$, respectively.
Then, the limit $\bm{\zeta}$ belongs to the set $\bm{U}_M^\sharp(\omega)$, and is the unique solution of Problem~\ref{problemLim}.
\end{theorem}
\begin{proof}
In what follows, we let for brevity:
$$
e_{i\|j}(\varepsilon):=e_{i\|j}(\varepsilon;\bm{u}(\varepsilon)).
$$

The outline of the proof, which is broken into ten parts numbered (i)--(x), is inspired by the proof of Theorem~5.6-1 of~\cite{Ciarlet2000} (itself adapted from the original result by Ciarlet \& Lods contained in~\cite{CiaLods1996d}), where no confinement conditions are imposed. We will see that taking into account the confinement condition requires extra care in some parts of the asymptotic analysis.

(i) \emph{There exists $0<\varepsilon_1 \le \varepsilon_0$, and there exists a constant $c_0>0$ independent of $\varepsilon$ such that, for all $0<\varepsilon\le\varepsilon_1$, we have that:
\begin{equation*}
	|\bm{v}|_\Omega^M \le c_0 \left\{\sum_{i,j}\|e_{i\|j}(\varepsilon;\bm{v})\|_{L^2(\Omega)}^2\right\}^{1/2},\quad\textup{ for all }\bm{v} \in \bm{V}(\Omega).
\end{equation*}
}

The proof is the same as that of part~(i) in Theorem~5.6-1 of~\cite{Ciarlet2000} and, for this reason, is omitted.

(ii) \emph{The semi-norms $|\bm{u}(\varepsilon)|_\Omega^M$, $\left|\overline{\bm{u}(\varepsilon)}\right|_\omega^M$, and the norms $\|\varepsilon\bm{u}(\varepsilon)\|_{\bm{H}^1(\Omega)}$ and $\|e_{i\|j}(\varepsilon)\|_{L^2(\Omega)}$ are bounded independently of $0<\varepsilon\le\varepsilon_1$. Therefore, up to a subsequence, there exist $\bm{u}\in\bm{V}_M^\sharp(\Omega)$, $\bm{u}^{-1}=(u_i^{(-1)}) \in \bm{V}(\Omega)$, $e_{i\|j} \in L^2(\Omega)$, and $\bm{\zeta} \in \bm{U}_M^\sharp(\omega)$ for which the following convergences take place as $\varepsilon\to0$:
\begin{equation*}
	\begin{aligned}
		\bm{u}(\varepsilon) &\rightharpoonup \bm{u},\quad\textup{ in } \bm{V}_M^\sharp(\Omega),\\
		\varepsilon\bm{u}(\varepsilon) &\rightharpoonup \bm{u}^{(-1)},\quad\textup{ in } \bm{H}^1(\Omega),\\
		e_{i\|j}(\varepsilon) &\rightharpoonup e_{i\|j},\quad\textup{ in } \bm{L}^2(\Omega),\\
		\partial_3 u_3(\varepsilon) &=\varepsilon e_{3\|3}(\varepsilon) \to 0,\quad\textup{ in } L^2(\Omega),\\
		\overline{\bm{u}(\varepsilon)} &\rightharpoonup \bm{\zeta},\quad\textup{ in }\bm{U}_M^\sharp(\omega).
	\end{aligned}
\end{equation*}
}

Letting $\bm{v}=\bm{0}$ in the variational inequalities of Problem~\ref{problem0scaled}, and recalling that the applied body forces are admissible (see~\eqref{adm-1} and~\eqref{adm-2}), we obtain, as a consequence of part~(i), that
\begin{equation*}
\begin{aligned}
	&\dfrac{\sqrt{g_0}}{C_e c_0^2}\left(|\bm{u}(\varepsilon)|_\Omega^M\right)^2 \le \dfrac{\sqrt{g_0}}{C_e} \sum_{i,j}\|e_{i\|j}(\varepsilon)\|_{L^2(\Omega)}^2 \le \int_{\Omega} A^{ijk\ell}(\varepsilon) e_{k\|\ell}(\varepsilon) e_{i\|j}(\varepsilon) \sqrt{g(\varepsilon)} \dd x\\
	&\le \int_{\Omega} F^{ij}(\varepsilon) e_{i\|j}(\varepsilon) \sqrt{g(\varepsilon)} \dd x \le \kappa_0 \sqrt{g_1} \left\{\sum_{i,j} \|e_{i\|j}(\varepsilon)\|_{L^2(\Omega)}^2\right\}^{1/2},
\end{aligned}
\end{equation*}
for some $\kappa_0>0$ independent of $\varepsilon$. As a result,
\begin{equation*}
	\begin{aligned}
		\left\{\|e_{i\|j}(\varepsilon)\|_{L^2(\Omega)}\right\}_{\varepsilon>0} &\textup{ is bounded independently of } 0<\varepsilon\le \varepsilon_1,\\
		\left\{|\bm{u}(\varepsilon)|_\Omega^M\right\}_{\varepsilon>0} &\textup{ is bounded independently of } 0<\varepsilon\le \varepsilon_1.
	\end{aligned}
\end{equation*}

An application of Theorem~4.2-1 in~\cite{Ciarlet2000} gives that
\begin{equation*}
	\left|\overline{\bm{u}(\varepsilon)}\right|_\omega^M \le |\bm{u}(\varepsilon)|_\Omega^M,
\end{equation*}
and an application of Theorem~\ref{t:4} implies the boundedness of the sequence $\{\varepsilon\|\bm{u}(\varepsilon)\|_{\bm{H}^1(\Omega)}\}_{\varepsilon>0}$ independently of $0<\varepsilon\le\varepsilon_1$.

In order to show that $\bm{\zeta} \in \bm{U}_M^\sharp(\omega)$, it suffices to verify that $\bm{U}_M^\sharp(\omega)$ is convex (note that $\bm{U}_M^\sharp(\omega)$ is - by definition - (strongly) closed with respect to the norm $|\cdot|_\omega^M$), so that a classical result in Functional Analysis ensures that $\bm{U}_M^\sharp(\omega)$ is also \emph{weakly closed} (cf., e.g., \cite{Brez11}).

Let $\bm{\zeta}_1, \bm{\zeta}_2 \in \bm{U}_M^\sharp(\omega)$, and let $0\le r\le 1$.
Consider sequences $\{\bm{\zeta}_1^{(k)}\}_{k\ge 0}$ and $\{\bm{\zeta}_2^{(k)}\}_{k\ge0}$ in $\bm{U}(\omega)$ that satisfy:
\begin{equation*}
	\begin{aligned}
		|\bm{\zeta}_1^{(k)}-\bm{\zeta}_1|_\omega^M & \to 0, \quad\textup{ as } k\to\infty,\\
		|\bm{\zeta}_2^{(k)}-\bm{\zeta}_2|_\omega^M & \to 0, \quad\textup{ as } k\to\infty.
	\end{aligned}
\end{equation*}

An application of the triangle inequality gives:
\begin{equation*}
	\left|\left(r \bm{\zeta}_1^{(k)}+(1-r)\bm{\zeta}_2^{(k)}\right) - \left(r \bm{\zeta}_1+(1-r)\bm{\zeta}_2\right)\right|_\omega^M \le r |\bm{\zeta}_1^{(k)}-\bm{\zeta}_1|_\omega^M+(1-r) |\bm{\zeta}_2^{(k)}-\bm{\zeta}_2|_\omega^M \to 0, \quad\textup{ as }k\to\infty.
\end{equation*}

Combining the previous convergence and the convexity of the set $\bm{U}(\omega)$ shows that
\begin{equation*}
	\left(r \bm{\zeta}_1+(1-r)\bm{\zeta}_2\right) \in \bm{U}_M^\sharp(\omega),
\end{equation*}
and the sought convexity is thus obtained, showing that $\bm{U}_M^\sharp(\omega)$ is weakly closed.

As a result, combining the fact that the elements $\overline{\bm{u}(\varepsilon)} \in \bm{U}(\omega)$ (cf. Lemma~\ref{average-constraint}) are bounded independently of $0<\varepsilon\le \varepsilon_1$ with respect to the norm $|\cdot|_\omega^M$ with the fact that $\bm{U}_M^\sharp(\omega)$ is weakly closed in turn gives that
\begin{equation*}
	\overline{\bm{u}(\varepsilon)} \rightharpoonup \bm{\zeta},\quad \textup{ in }\bm{U}_M^\sharp(\omega)\textup{ as }\varepsilon\to 0.
\end{equation*}

(iii) \emph{There exists a constant $0<\varepsilon_2\le \varepsilon_1$ such that, given any $0 < \varepsilon \le \varepsilon_2$ and any vector field $\bm{\varphi}=(\varphi_i)\in\boldsymbol{\mathcal{D}}(\Omega)$, one can find a vector field $\bm{v}(\varepsilon;\bm{\varphi})= (v_i(\varepsilon;\bm{\varphi})):\Omega\to\mathbb{R}^3$ with the following properties:}
$$
\bm{v}(\varepsilon;\bm{\varphi}) \in \bm{U}(\varepsilon;\Omega) \quad \textup{ and } \quad \partial_3 v_i (\varepsilon;\bm{\varphi}) = \varphi_i \textup{ in } \Omega, \textup{ for all } 0 <\varepsilon\le\varepsilon_2.
$$

\emph{Moreover, there exists a constant $C(\bm{\varphi})$ such that:}
$$
\|\bm{v} (\varepsilon; \bm{\varphi})\|_{\bm{H}^1(\Omega)} \le C(\bm{\varphi}),\quad \textup{ for all } 0 < \varepsilon \le \varepsilon_2.
$$

The proof is identical to that of part~(iv) in~\cite{CiaMarPie2018} and, for this reason, is omitted.

(iv) \emph{The weak limits $e_{i\|j} \in L^2(\Omega)$ found in part~(ii) take the following forms:}
\begin{equation*}
	\begin{aligned}
		e_{\alpha\|3}&=\dfrac{1}{2\mu} a_{\alpha\beta} F^{\beta3},\\
		e_{3\|3}&=-\dfrac{\lambda}{\lambda+2\mu}a^{\alpha\beta}e_{\alpha\|\beta}+\dfrac{F^{33}}{\lambda+2\mu}.
	\end{aligned}
\end{equation*}

Let the vector field $\bm{\varphi}=(\varphi_i) \in \bm{\mathcal{D}}(\Omega)$ be given, for a given $0<\varepsilon\le \varepsilon_2$, let $\bm{v}(\varepsilon;\bm{\varphi})$ be the corresponding vector field defined in part~(iii). Letting $\bm{v}=\bm{v}(\varepsilon;\bm{\varphi})$ in the variational inequalities of Problem~\ref{problem0scaled}, multiplying the aforementioned variational inequalities by $\varepsilon>0$, and recalling that $A^{\alpha\beta\sigma3}(\varepsilon)=A^{\alpha333}(\varepsilon)=0$, we obtain:
\begin{equation}
	\label{varineq1}
	\begin{aligned}
		&\int_{\Omega}A^{\alpha\beta\sigma\tau}(\varepsilon) e_{\sigma\|\tau}(\varepsilon) \left\{\varepsilon e_{\alpha\|\beta}(\varepsilon;\bm{v}(\varepsilon;\bm{\varphi}))-\varepsilon e_{\alpha\|\beta}(\varepsilon)\right\} \sqrt{g(\varepsilon)} \dd x\\
		&\quad+\int_{\Omega} A^{\sigma\tau33}(\varepsilon) e_{3\|3}(\varepsilon)\left\{\varepsilon e_{\sigma\|\tau}(\varepsilon;\bm{v}(\varepsilon;\bm{\varphi})) -\varepsilon e_{\sigma\|\tau}(\varepsilon)\right\} \sqrt{g(\varepsilon)} \dd x\\
		&\quad+4\int_{\Omega}A^{\alpha3\sigma3}(\varepsilon) e_{\sigma\|3}(\varepsilon)\left\{\varepsilon e_{\alpha\|3}(\varepsilon;\bm{v}(\varepsilon;\bm{\varphi}))-\varepsilon e_{\alpha\|3}(\varepsilon)\right\} \sqrt{g(\varepsilon)} \dd x\\
		&\quad+\int_{\Omega}A^{33\sigma\tau}(\varepsilon) e_{\sigma\|\tau}(\varepsilon) \left\{\varepsilon e_{3\|3}(\varepsilon;\bm{v}(\varepsilon;\bm{\varphi})) -\varepsilon e_{3\|3}(\varepsilon)\right\} \sqrt{g(\varepsilon)} \dd x\\
		&\quad+\int_{\Omega}A^{3333}(\varepsilon) e_{3\|3}(\varepsilon) \left\{\varepsilon e_{3\|3}(\varepsilon;\bm{v}(\varepsilon;\bm{\varphi}))-\varepsilon e_{3\|3}(\varepsilon)\right\} \sqrt{g(\varepsilon)} \dd x\\
		&\ge \int_{\Omega} F^{\alpha\beta}(\varepsilon) \left\{\varepsilon e_{\alpha\|\beta}(\varepsilon;\bm{v}(\varepsilon;\bm{\varphi}))-\varepsilon e_{\alpha\|\beta}(\varepsilon)\right\} \sqrt{g(\varepsilon)} \dd x\\
		&\quad+2\int_{\Omega} F^{\alpha3}(\varepsilon)\left\{\varepsilon e_{\alpha\|3}(\varepsilon;\bm{v}(\varepsilon;\bm{\varphi}))-\varepsilon e_{\alpha\|3}(\varepsilon)\right\} \sqrt{g(\varepsilon)} \dd x\\
		&\quad+\int_{\Omega} F^{33}(\varepsilon) \left\{\varepsilon e_{3\|3}(\varepsilon;\bm{v}(\varepsilon;\bm{\varphi}))-\varepsilon e_{3\|3}(\varepsilon)\right\} \sqrt{g(\varepsilon)} \dd x.
	\end{aligned}
\end{equation}

Together, the relations $\partial_3 v_i(\varepsilon;\bm{\varphi})=\varphi_i$ and the boundedness of the norms $\|\bm{v}(\varepsilon;\bm{\varphi})\|_{\bm{H}^1(\Omega)}$ independently of $0<\varepsilon\le\varepsilon_2$ imply that:
\begin{equation}
	\label{conv}
	\begin{aligned}
		\varepsilon e_{\alpha\|\beta}(\varepsilon;\bm{v}(\varepsilon;\bm{\varphi})) & \to 0,\quad\textup{ in }L^2(\Omega),\\
		\varepsilon e_{\alpha\|3}(\varepsilon;\bm{v}(\varepsilon;\bm{\varphi})) & \to \dfrac{\varphi_\alpha}{2},\quad\textup{ in }L^2(\Omega),\\
		\varepsilon e_{3\|3}(\varepsilon;\bm{v}(\varepsilon;\bm{\varphi})) &= \varphi_3,\quad\textup{ a.e. in }\Omega \textup{ for all }0<\varepsilon\le\varepsilon_2.
	\end{aligned}
\end{equation}

Letting $\varepsilon\to0$ in~\eqref{varineq1}, recalling the definition of admissible applied body forces~\eqref{adm-1}, and using the convergences established in part~(ii) and in~\eqref{conv}, we obtain that:
\begin{equation}
\label{varineq2}
\begin{aligned}
	&\int_{\Omega}2\mu a^{\alpha\sigma} e_{\sigma\|3} \varphi_\alpha 
	\sqrt{a} \dd x +\int_{\Omega}\left(\lambda a^{\sigma\tau}e_{\sigma\|\tau}+(\lambda+2\mu)e_{3\|3}\right) \varphi_3 \sqrt{a} \dd x\\
	&\ge \int_{\Omega} F^{\alpha3} \varphi_\alpha \sqrt{a} \dd x + \int_{\Omega} F^{33} \varphi_3 \sqrt{a} \dd x.
\end{aligned}
\end{equation}

As inequality~\eqref{varineq2} holds for any vector field $\bm{\varphi}=(\varphi_i) \in\bm{\mathcal{D}}(\Omega)$, we infer that, if we specialise such vector field in such a way that $\varphi_3=0$ and $\varphi_\alpha$ varies in $\mathcal{D}(\Omega)$, then:
\begin{equation*}
2\mu
\begin{pmatrix}
	a^{11}&a^{12}\\
	a^{12}&a^{22}
\end{pmatrix}
\begin{pmatrix}
	e_{1\|3}\\
	e_{2\|3}
\end{pmatrix}
=
\begin{pmatrix}
F^{13}\\
F^{23}
\end{pmatrix},\quad\textup{ a.e. in }\Omega.
\end{equation*}

The invertibility of the symmetric and positive-definite matrix $(a^{\alpha\beta})$ in turn gives that:
\begin{equation*}
		e_{\alpha\|3}=\dfrac{1}{2\mu} a_{\alpha\beta} F^{\beta3}.
\end{equation*}

The relations
\begin{equation*}
	e_{3\|3}=-\dfrac{\lambda}{\lambda+2\mu}a^{\alpha\beta}e_{\alpha\|\beta}+\dfrac{F^{33}}{\lambda+2\mu},
\end{equation*}
can be likewise established, by letting $\varphi_\alpha=0$ and letting $\varphi_3 \in \mathcal{D}(\Omega)$ vary arbitrarily.

(v) \emph{The whole family $\{\bm{u}(\varepsilon)\}_{\varepsilon>0}$ satisfies:}
\begin{equation*}
	\left\{\overline{e_{\alpha\|\beta}(\varepsilon)} - \gamma_{\alpha\beta}\left(\overline{\bm{u}(\varepsilon)}\right)\right\} \to 0,\quad\textup{ in }L^2(\omega) \textup{ as }\varepsilon\to 0.
\end{equation*}

\emph{Consequently, the subsequence considered in part~(ii) satisfies:}
\begin{equation*}
	\gamma_{\alpha\beta}\left(\overline{\bm{u}(\varepsilon)}\right) \rightharpoonup \overline{e_{\alpha\|\beta}},\quad\textup{ in }L^2(\omega) \textup{ as }\varepsilon\to 0.
\end{equation*}

The proof is identical to that of part~(iv) in Theorem~5.6-1 of~\cite{Ciarlet2000} and is, for this reason, omitted.

(vi) \emph{The subsequence $\{\bm{u}(\varepsilon)\}_{\varepsilon>0}$ found in part~(ii) is such that:}
\begin{equation*}
	\begin{aligned}
		\varepsilon\bm{u}(\varepsilon) &\rightharpoonup \bm{0},\quad\textup{ in }\bm{H}^1(\Omega),\\
		\partial_3u_\alpha(\varepsilon) &\rightharpoonup 0,\quad\textup{ in } L^2(\Omega),
	\end{aligned}
\end{equation*}
\emph{as $\varepsilon\to 0$. Furthermore, the weak limits $e_{\alpha\|\beta}$ are independent of the transverse variable $x_3$.}

The proof is identical to that of part~(vi) of Theorem~5.6-1 of~\cite{Ciarlet2000} and, for this reason, is omitted.

(vii) \emph{The limits $e_{\alpha\|\beta}$ obtained in part~(ii) satisfy the following variational inequalities:}
\begin{equation*}
	\int_{\omega} a^{\alpha\beta\sigma\tau} \overline{e_{\sigma\|\tau}} (\gamma_{\alpha\beta}(\bm{\eta})-\overline{e_{\alpha\|\beta}}) \sqrt{a} \dd y \ge \int_{\omega} \varphi^{\alpha\beta} (\gamma_{\alpha\beta}(\bm{\eta})-\overline{e_{\alpha\|\beta}}) \sqrt{a} \dd y,\quad\textup{ for all }\bm{\eta} \in\bm{U}(\omega),
\end{equation*}
\emph{where}
\begin{equation*}
	\varphi^{\alpha\beta}:=\int_{-1}^{1} \left\{F^{\alpha\beta}-\dfrac{\lambda}{\lambda+2\mu}a^{\alpha\beta} F^{33}\right\} \dd x_3 \in L^2(\omega),
\end{equation*}
\emph{the functions $F^{ij} \in L^2(\Omega)$ being those used in the definition of admissible applied body forces (cf.~\eqref{adm-1} and~\eqref{adm-2}).}

For a given $\bm{\eta}=(\eta_i) \in \bm{U}(\omega)$, and for each $0<\varepsilon\le \varepsilon_2$, define the vector field $\bm{v}(\varepsilon;\bm{\eta})=(v_i(\varepsilon;\bm{\eta}))$ as follows:
\begin{equation*}
	v_i(\varepsilon;\bm{\eta}):=\left[(1-\sqrt{\varepsilon}) \eta_j\bm{a}^j\right]\cdot\bm{g}_i(\varepsilon).
\end{equation*}

Since $\bm{\eta} \in \bm{U}(\omega)$, since the vector fields $\bm{a}^j$ are of class $\bm{\mathcal{C}}^2(\overline{\omega})$, and since the vector fields $\bm{g}_i(\varepsilon)$ are of class $\bm{\mathcal{C}}^1(\overline{\Omega})$, we obtain  (cf., e.g., Proposition~9.4 of~\cite{Brez11}) that $
\bm{v}(\varepsilon;\bm{\eta}) \in \bm{H}^1(\Omega)$. Moreover, the trace of each component $v_i(\varepsilon;\bm{\eta})$ vanishes along $\gamma_0 \times (-1,1)$. That $\bm{\theta} \cdot\bm{q}>0$ in $\overline{\omega}$ implies that, for a.a. $x\in\Omega$,
\begin{equation*}
	\begin{aligned}
		&\left[\bm{\theta}+\varepsilon x_3 \bm{a}^3+v_i(\varepsilon;\bm{\eta})\bm{g}^i(\varepsilon)\right] \cdot\bm{q}\\
		&=\left[\bm{\theta}+\varepsilon x_3 \bm{a}^3+(\eta_j \bm{a}^j \cdot \bm{g}_i(\varepsilon))\bm{g}^i(\varepsilon)\right]\cdot\bm{q}\\
		&\quad-\sqrt{\varepsilon}(\eta_j\bm{a}^j \cdot\bm{g}_i(\varepsilon)) (\bm{g}^i(\varepsilon) \cdot\bm{q})\\
		&=(1-\sqrt{\varepsilon})\left(\bm{\theta}+\eta_j\bm{a}^j\right) \cdot\bm{q} +\sqrt{\varepsilon}(\bm{\theta}\cdot\bm{q}+\sqrt{\varepsilon} x_3 \bm{a}^3\cdot\bm{q}),
	\end{aligned}
\end{equation*}
where $(\bm{\theta}+\eta_j\bm{a}^j)\cdot\bm{q}\ge 0$ a.e. in $\omega$, $\bm{\theta}\cdot\bm{q}>0$ in $\overline{\omega}$, and $\sqrt{\varepsilon}x_3 \bm{a}^3\cdot\bm{q}=\mathcal{O}(\sqrt{\varepsilon})$.
Besides, the latter quantity is greater or equal than zero provided that $\varepsilon$ is sufficiently small. This implies that $\bm{v}(\varepsilon;\bm{\eta}) \in \bm{U}(\varepsilon;\Omega)$.
Since
\begin{equation*}
v_j(\varepsilon;\bm{\eta}) \bm{g}^j(\varepsilon) = \left(\left( (1-\sqrt{\varepsilon}) \eta_i \bm{a}^i\right) \cdot \bm{g}_j (\varepsilon)\right)\bm{g}^j(\varepsilon) = (1-\sqrt{\varepsilon})\eta_j\bm{a}^j, 
\end{equation*}
it follows that:
\begin{equation*}
	\|v_j(\varepsilon;\bm{\eta})\bm{g}^j (\varepsilon)-\eta_j \bm{a}^j\|_{\bm{H}^1(\Omega)} = \sqrt{\varepsilon}\|\eta_i \bm{a}^i\|_{\bm{H}^1(\Omega)} \to 0,\quad \textup{ as } \varepsilon \to 0.
\end{equation*}

This convergence in turn implies that:
\begin{equation*}
	v_i(\varepsilon;\bm{\eta})=\left(v_j (\varepsilon;\bm{\eta}) \bm{g}^j(\varepsilon)\right) \cdot \bm{g}_i(\varepsilon) \to (\eta_j\bm{a}^j) \cdot \bm{a}_i = \eta_i,\quad \textup{ in } H^1(\Omega) \textup{ as } \varepsilon \to 0.
\end{equation*}

Finally, using Lemma~\ref{lem:2} and noting that
\begin{align*}
	v_\alpha (\varepsilon; \bm{\eta}) &= \left( (1-\sqrt{\varepsilon} ) \eta_j \bm{a}^j \right) \cdot \bm{g}_\alpha (\varepsilon) = (1- \sqrt{\varepsilon}) \left( \eta_\alpha - \varepsilon x_3 b^\sigma_\alpha \eta_\sigma \right), \\
	v_3 (\varepsilon; \bm{\eta}) &= \left( (1-\sqrt{\varepsilon} ) \eta_j \bm{a}^j \right) \cdot \bm{g}_3 (\varepsilon) = (1- \sqrt{\varepsilon}) \eta_3,
\end{align*}
we deduce from the definition of the scaled strains $e_{i\|j}(\varepsilon;\bm{v})$, that
\begin{equation}
	\label{conv2}
	\begin{aligned}
	e_{\alpha\|\beta}(\varepsilon; \bm{v}(\varepsilon;\bm{\eta})) &= (1-\sqrt{\varepsilon})\left(\gamma_{\alpha\beta}(\bm{\eta}) + \varepsilon x_3 \mathcal{R}_{\alpha\beta}(\varepsilon;\bm{\eta})\right) \to \gamma_{\alpha\beta}(\bm{\eta}),\quad\textup{ in }L^2(\Omega)\textup{ as }\varepsilon\to 0\\
	e_{\alpha\|3}(\varepsilon;\bm{v}(\varepsilon;\bm{\eta})) &= (1-\sqrt{\varepsilon}) \left(\dfrac{1}{2}\left(b^\sigma_\alpha \eta_\sigma + \partial_\alpha \eta_3\right)+ \varepsilon x_3 \mathcal{R}_{\alpha3}(\varepsilon;\bm{\eta})\right),\\
	e_{3\|3}(\varepsilon;\bm{v}(\varepsilon; \bm{\eta})) &= 0,
\end{aligned}
\end{equation}
for some functions $\mathcal{R}_{\alpha j} (\varepsilon;\bm{\eta}) \in L^2(\Omega)$ that are bounded in $L^2(\Omega)$ independently of $\varepsilon$.

Given $\bm{\eta} \in \bm{U}(\omega)$, letting $\bm{v}=\bm{v}(\varepsilon;\bm{\eta})$ in the variational inequalities gives:
\begin{equation}
	\label{varineq3}
	\begin{aligned}
		0\le&\int_{\Omega}A^{ijk\ell}(\varepsilon) e_{k\|\ell}(\varepsilon)e_{i\|j}(\varepsilon) \sqrt{g(\varepsilon)} \dd x\\
		&\quad-2\int_{\Omega}A^{ijk\ell}(\varepsilon) e_{k\|\ell}(\varepsilon) e_{i\|j} \sqrt{g(\varepsilon)} \dd x\\
		&\quad+\int_{\Omega} A^{ijk\ell}(\varepsilon) e_{k\|\ell} e_{i\|j} \sqrt{g(\varepsilon)} \dd x\\
		&\le \int_{\Omega}A^{ijk\ell}(\varepsilon) e_{k\|\ell}(\varepsilon)e_{i\|j}(\varepsilon;\bm{v}(\varepsilon;\bm{\eta})) \sqrt{g(\varepsilon)} \dd x\\
	&\quad-2\int_{\Omega}A^{ijk\ell}(\varepsilon) e_{k\|\ell}(\varepsilon) e_{i\|j} \sqrt{g(\varepsilon)} \dd x\\
	&\quad+\int_{\Omega} A^{ijk\ell}(\varepsilon) e_{k\|\ell} e_{i\|j} \sqrt{g(\varepsilon)} \dd x\\
	&\quad-\int_{\Omega} F^{ij}(\varepsilon) \left(e_{i\|j}(\varepsilon;\bm{v}(\varepsilon;\bm{\eta}))-e_{i\|j}(\varepsilon)\right) \sqrt{g(\varepsilon)} \dd x\\
	&=\int_{\Omega}A^{\alpha\beta\sigma\tau}(\varepsilon) e_{\sigma\|\tau}(\varepsilon)e_{\alpha\|\beta}(\varepsilon;\bm{v}(\varepsilon;\bm{\eta})) \sqrt{g(\varepsilon)} \dd x\\
	&\quad+\int_{\Omega}A^{\sigma\tau33}(\varepsilon) e_{\sigma\|\tau}(\varepsilon)e_{3\|3}(\varepsilon;\bm{v}(\varepsilon;\bm{\eta})) \sqrt{g(\varepsilon)}\dd x\\
	&\quad+4\int_{\Omega}A^{\alpha3\sigma3}(\varepsilon) e_{\sigma\|3}(\varepsilon)e_{\alpha\|3}(\varepsilon;\bm{v}(\varepsilon;\bm{\eta})) \sqrt{g(\varepsilon)} \dd x\\
	&\quad+\int_{\Omega}A^{33\sigma\tau}(\varepsilon) e_{3\|3}(\varepsilon)e_{\sigma\|\tau}(\varepsilon;\bm{v}(\varepsilon;\bm{\eta})) \sqrt{g(\varepsilon)} \dd x\\
	&\quad+\int_{\Omega}A^{3333}(\varepsilon) e_{3\|3}(\varepsilon)e_{3\|3}(\varepsilon;\bm{v}(\varepsilon;\bm{\eta})) \sqrt{g(\varepsilon)} \dd x\\
	&\quad-2\int_{\Omega}A^{ijk\ell}(\varepsilon) e_{k\|\ell}(\varepsilon) e_{i\|j} \sqrt{g(\varepsilon)} \dd x\\
	&\quad+\int_{\Omega} A^{ijk\ell}(\varepsilon) e_{k\|\ell} e_{i\|j} \sqrt{g(\varepsilon)} \dd x\\
	&\quad-\int_{\Omega} F^{\alpha\beta}(\varepsilon) \left(e_{\alpha\|\beta}(\varepsilon;\bm{v}(\varepsilon;\bm{\eta}))-e_{\alpha\|\beta}(\varepsilon)\right) \sqrt{g(\varepsilon)} \dd x\\
	&\quad-2\int_{\Omega} F^{\alpha3}(\varepsilon) \left(e_{\alpha\|3}(\varepsilon;\bm{v}(\varepsilon;\bm{\eta}))-e_{\alpha\|3}(\varepsilon)\right) \sqrt{g(\varepsilon)} \dd x\\
	&\quad-\int_{\Omega} F^{33}(\varepsilon) \left(e_{3\|3}(\varepsilon;\bm{v}(\varepsilon;\bm{\eta}))-e_{3\|3}(\varepsilon)\right) \sqrt{g(\varepsilon)} \dd x.
\end{aligned}
\end{equation}

Together, the assumed admissibility of the applied body forces (cf.~\eqref{adm-1} and~\eqref{adm-2}), the weak convergences established in part~(ii), the relations satisfied by the weak limits $e_{i\|j}$, and the asymptotic behaviour of the vector fields $\bm{v}(\varepsilon;\bm{\eta})$ as $\varepsilon\to 0$ (viz. \eqref{conv2}) give:
\begin{equation*}
	\begin{aligned}
		\int_{\Omega} F^{\alpha\beta}(\varepsilon) \left(e_{\alpha\|\beta}(\varepsilon;\bm{v}(\varepsilon;\bm{\eta}))-e_{\alpha\|\beta}(\varepsilon)\right) \sqrt{g(\varepsilon)} \dd x &\to \int_{\Omega} F^{\alpha\beta} (\gamma_{\alpha\beta}(\bm{\eta}) - e_{\alpha\|\beta}) \sqrt{a} \dd x,\\
		-2\int_{\Omega} F^{\alpha3}(\varepsilon) \left(e_{\alpha\|3}(\varepsilon;\bm{v}(\varepsilon;\bm{\eta}))-e_{\alpha\|3}(\varepsilon)\right) \sqrt{g(\varepsilon)} \dd x &\to \int_{\Omega} F^{\alpha3} \left(-(b_\alpha^\sigma \eta_\sigma+\partial_\alpha\eta_3)+\dfrac{a_{\alpha\beta}}{\mu}F^{\beta3}\right) \sqrt{a} \dd x,\\
		-\int_{\Omega}F^{33}(\varepsilon)\left(e_{3\|3}(\varepsilon;\bm{v}(\varepsilon;\bm{\eta}))-e_{3\|3}(\varepsilon)\right) \sqrt{g(\varepsilon)} \dd x&\to\int_{\Omega}F^{33}\left(-\dfrac{\lambda}{\lambda+2\mu}a^{\alpha\beta} e_{\alpha\|\beta}+\dfrac{F^{33}}{\lambda+2\mu}\right)\sqrt{a} \dd x,\\
		-\int_{\Omega} A^{ijk\ell}(\varepsilon) e_{k\|\ell} e_{i\|j} \sqrt{g(\varepsilon)} \dd x &\to -\int_{\omega} a^{\alpha\beta\sigma\tau} \overline{e_{\sigma\|\tau}}\, \overline{e_{\alpha\|\beta}} \sqrt{a} \dd y\\
		&\qquad-\int_{\Omega} \left\{\dfrac{a_{\alpha\beta}}{\mu}F^{\alpha3} F^{\beta3}+\dfrac{|F^{33}|^2}{\lambda+2\mu}\right\} \sqrt{a} \dd y, \quad\textup{ thanks to part (vi),}\\
		\int_{\Omega}A^{\alpha\beta\sigma\tau}(\varepsilon) e_{\sigma\|\tau}(\varepsilon)e_{\alpha\|\beta}(\varepsilon;\bm{v}(\varepsilon;\bm{\eta})) \sqrt{g(\varepsilon)} \dd x &\to 
		\int_{\Omega}\{\lambda a^{\alpha\beta}a^{\sigma\tau}+\mu(a^{\alpha\sigma}a^{\beta\tau}+a^{\alpha\tau}a^{\beta\sigma})\} e_{\sigma\|\tau} \gamma_{\alpha\beta}(\bm{\eta}) \sqrt{a} \dd x,\\
		\int_{\Omega}A^{\sigma\tau33}(\varepsilon) e_{\sigma\|\tau}(\varepsilon)e_{3\|3}(\varepsilon;\bm{v}(\varepsilon;\bm{\eta})) \sqrt{g(\varepsilon)}\dd x &=0,\\
		4\int_{\Omega}A^{\alpha3\sigma3}(\varepsilon) e_{\sigma\|3}(\varepsilon)e_{\alpha\|3}(\varepsilon;\bm{v}(\varepsilon;\bm{\eta})) \sqrt{g(\varepsilon)} \dd x &\to \int_{\Omega} F^{\alpha3}(b_\alpha^\sigma \eta_\sigma+\partial_\alpha\eta_3) \sqrt{a} \dd x,\\
		\int_{\Omega}A^{33\sigma\tau}(\varepsilon) e_{3\|3}(\varepsilon)e_{\sigma\|\tau}(\varepsilon;\bm{v}(\varepsilon;\bm{\eta})) \sqrt{g(\varepsilon)} \dd x &\to \int_{\Omega} \lambda a^{\sigma\tau}\left(-\dfrac{\lambda}{\lambda+2\mu}a^{\alpha\beta} e_{\alpha\|\beta}+\dfrac{F^{33}}{\lambda+2\mu}\right) \gamma_{\sigma\tau}(\bm{\eta}) \sqrt{a} \dd x,\\
		\int_{\Omega}A^{3333}(\varepsilon) e_{3\|3}(\varepsilon)e_{3\|3}(\varepsilon;\bm{v}(\varepsilon;\bm{\eta})) \sqrt{g(\varepsilon)} \dd x &=0.
	\end{aligned}
\end{equation*}

Therefore, combining~\eqref{varineq3} with the above relations and passing to the averages gives:
\begin{equation*}
\int_{\omega} a^{\alpha\beta\sigma\tau} \overline{e_{\sigma\|\tau}} \left(\gamma_{\alpha\beta}(\bm{\eta})-\overline{e_{\alpha\|\beta}}\right) \sqrt{a} \dd y \ge \int_{\omega} \left(\int_{-1}^{1} F^{\alpha\beta}-\dfrac{\lambda}{\lambda+2\mu} a^{\alpha\beta}F^{33} \dd x_3\right) \left(\gamma_{\alpha\beta}(\bm{\eta})-\overline{e_{\alpha\|\beta}}\right) \sqrt{a} \dd y,
\end{equation*}
which is the sought inequality.

(viii) \emph{The following convergences, so far only known to be weak, are actually strong:}
\begin{equation*}
	\begin{aligned}
		e_{i\|j}(\varepsilon) &\to e_{i\|j},\quad\textup{ in }L^2(\Omega),\\
		\varepsilon\bm{u}(\varepsilon) &\to \bm{0},\quad\textup{ in }\bm{H}^1(\Omega),\\
		\gamma_{\alpha\beta}\left(\overline{\bm{u}(\varepsilon)}\right) &\to \overline{e_{\alpha\|\beta}},\quad\textup{ in }L^2(\omega),\\
		\overline{\bm{u}(\varepsilon)} &\to \bm{\zeta},\quad\textup{ in }\bm{U}_M^\sharp(\omega).
	\end{aligned}
\end{equation*}

Consider the restriction of the linearised strain tensor to the space $\bm{V}(\omega)$, still denoted and defined by:
\begin{equation*}
	\gamma_{\alpha\beta}(\bm{\eta})=\dfrac{1}{2}(\partial_\beta \eta_\alpha + \partial_\alpha \eta_\beta ) - \Gamma^\sigma_{\alpha\beta} \eta_\sigma - b_{\alpha\beta} \eta_3, \quad\textup{ for all }\bm{\eta}=(\eta_i)\in\bm{V}(\omega).
\end{equation*}

Combining the linearity and continuity of such an operator with the assumption that the semi-norm $|\cdot|_\omega^M$ is actually a norm over $\bm{V}(\omega)$ (since the linearly elastic generalised membrane shell under consideration is of the first kind), allows us to apply Theorem~3.1-1 in~\cite{Ciarlet2025}, so as to infer that there exists a unique linear and continuous extension that we denote by:
\begin{equation*}
	\gamma_{\alpha\beta}^{\sharp}:\bm{V}_M^\sharp(\omega) \to L^2(\omega).
\end{equation*}

Recall that the following convergences were established in part~(ii) and~(v), respectively:
\begin{equation*}
	\begin{aligned}
		\overline{\bm{u}(\varepsilon)} \rightharpoonup \bm{\zeta}&,\quad\textup{ in } \bm{U}_M^\sharp(\omega) \textup{ as }\varepsilon\to 0,\\
		\gamma_{\alpha\beta}\left(\overline{\bm{u}(\varepsilon)}\right) \rightharpoonup \overline{e_{\alpha\|\beta}}&,\quad\textup{ in }L^2(\omega) \textup{ as }\varepsilon\to 0.
	\end{aligned}
\end{equation*}

The linearity and continuity of $\gamma_{\alpha\beta}^{\sharp}$, together with the uniqueness of the weak limit (cf., e.g., \cite{Brez11}), allow us to infer that:
\begin{equation*}
	\gamma_{\alpha\beta}^{\sharp}(\bm{\zeta}) = \overline{e_{\alpha\|\beta}},\quad \textup{ in } L^2(\omega).
\end{equation*}

In light of the definition of $\bm{U}_M^\sharp(\omega)$, given a vector field $\bm{\zeta} \in \bm{U}_M^\sharp(\omega)$ satisfying the variational inequalities in part~(vii), there exists a sequence $\{\bm{\zeta}_\ell\}_{\ell\ge 1} \subset \bm{U}(\omega)$ such that:
\begin{equation*}
	\bm{\zeta}_\ell \to \bm{\zeta},\quad\textup{ in }\bm{U}_M^\sharp(\omega) \textup{ as }\ell\to\infty,
\end{equation*}
so that:
\begin{equation}
	\label{eq2}
	\gamma_{\alpha\beta}^{\sharp}(\bm{\zeta}_\ell) \to \overline{e_{\alpha\|\beta}},\quad\textup{ in }L^2(\omega) \textup{ as }\ell\to\infty.
\end{equation}

The same computations as in part~(vii) and the above properties of the extension $\gamma_{\alpha\beta}^{\sharp}$ lead to the following estimates
\begin{equation}
	\label{varineq4}
	\begin{aligned}
		0&\le \dfrac{\sqrt{g_0}}{C_e}\limsup_{\varepsilon\to 0}\left(\sum_{i,j}\|e_{i\|j}(\varepsilon)-e_{i\|j}\|_{L^2(\Omega)}^2\right) \le 
		\int_{\omega} a^{\alpha\beta\sigma\tau} \overline{e_{\sigma\|\tau}} \left(\gamma_{\alpha\beta}^{\sharp}(\bm{\zeta}_\ell)-\overline{e_{\alpha\|\beta}}\right) \sqrt{a} \dd y\\
		&\quad - \int_{\omega} \left(\int_{-1}^{1} F^{\alpha\beta}-\dfrac{\lambda}{\lambda+2\mu} a^{\alpha\beta}F^{33} \dd x_3\right) \left(\gamma_{\alpha\beta}^{\sharp}(\bm{\zeta}_\ell)-\overline{e_{\alpha\|\beta}}\right) \sqrt{a} \dd y,
	\end{aligned}
\end{equation}
which hold for each $\ell\ge 1$. Thanks to~\eqref{eq2}, the right-hand side in~\eqref{varineq4} tends to zero as $\ell\to\infty$.
We have thus established that:
\begin{equation*}
	\limsup_{\varepsilon\to 0} \|e_{i\|j}(\varepsilon)-e_{i\|j}\|_{L^2(\Omega)} =0, \quad\textup{ for all } i,j.
\end{equation*}

The remaining strong convergences can be established in the same fashion as in part~(vii) in Theorem~5.6-1 of~\cite{Ciarlet2000}. In particular, the elements $\overline{\bm{u}(\varepsilon)} \in \bm{U}(\omega)$ strongly converge to $\bm{\zeta} \in \bm{U}_M^\sharp(\omega)$ as $\varepsilon\to0$.

(ix) \emph{The limit $\bm{\zeta} \in \bm{U}_M^\sharp(\omega)$ found in part~(viii) is the unique solution of Problem~\ref{problemLim}. Consequently, the entire sequence $\{\overline{\bm{u}(\varepsilon)}\}_{\varepsilon>0}$ converges to $\bm{\zeta}$ in $\bm{V}_M^\sharp(\omega)$ as $\varepsilon\to0$.}

Fix $\bm{\eta} \in\bm{U}(\omega)$, and recall that $\overline{\bm{u}(\varepsilon)} \in\bm{U}(\omega)$. The strong convergences obtained in part~(viii) imply that:
\begin{equation}
	\label{conv3}
	\begin{aligned}
		\int_{\omega}a^{\alpha\beta\sigma\tau} \gamma_{\sigma\tau}\left(\overline{\bm{u}(\varepsilon)}\right) \gamma_{\alpha\beta}(\bm{\eta}) \sqrt{a} \dd y &\to \int_{\omega}a^{\alpha\beta\sigma\tau} \overline{e_{\sigma\|\tau}} \gamma_{\alpha\beta}(\bm{\eta}) \sqrt{a} \dd y,\\
		\int_{\omega}a^{\alpha\beta\sigma\tau} \gamma_{\sigma\tau}\left(\overline{\bm{u}(\varepsilon)}\right) \gamma_{\alpha\beta}\left(\overline{\bm{u}(\varepsilon)}\right) \sqrt{a} \dd y &\to \int_{\omega}a^{\alpha\beta\sigma\tau} \overline{e_{\sigma\|\tau}}\,\overline{e_{\alpha\|\beta}} \sqrt{a} \dd y.
	\end{aligned}
\end{equation}

The convergences in~\eqref{conv3} and the conclusion of part~(vii) imply that, on the one hand:
\begin{equation}
	\label{varineq5}
	\begin{aligned}
		\lim_{\varepsilon\to 0} &B_M\left(\overline{\bm{u}(\varepsilon)},\bm{\eta}-\overline{\bm{u}(\varepsilon)}\right) = \int_{\omega}a^{\alpha\beta\sigma\tau} \overline{e_{\sigma\|\tau}} \left(\gamma_{\alpha\beta}(\bm{\eta}) - \overline{e_{\alpha\|\beta}}\right) \sqrt{a} \dd y\\
		&\ge \int_{\omega} \varphi^{\alpha\beta} \left(\gamma_{\alpha\beta}(\bm{\eta}) - \overline{e_{\alpha\|\beta}}\right) \sqrt{a} \dd y=\lim_{\varepsilon\to 0}L_M\left(\bm{\eta}-\overline{\bm{u}(\varepsilon)}\right)=L_M^\sharp(\bm{\eta}-\bm{\zeta}).
	\end{aligned}
\end{equation}

On the other hand:
\begin{equation}
	\label{varineq6}
	\lim_{\varepsilon\to 0} B_M\left(\overline{\bm{u}(\varepsilon)},\bm{\eta}-\overline{\bm{u}(\varepsilon)}\right) = B_M^\sharp(\bm{\zeta},\bm{\eta}-\bm{\zeta}).
\end{equation}

Therefore, combining~\eqref{varineq5} and~\eqref{varineq6} and the definition of continuous extensions $B_M^\sharp$ and $L_M^\sharp$ gives:
\begin{equation*}
	B_M^\sharp(\bm{\zeta},\bm{\eta}-\bm{\zeta}) \ge L_M^\sharp(\bm{\eta}-\bm{\zeta}),\quad\textup{ for all }\bm{\eta}\in\bm{U}_M^\sharp(\omega).
\end{equation*}

Noting that the set $\bm{U}_M^\sharp(\omega)$ is non-empty, closed and convex (viz. part~(ii)) and that a straightforward density argument shows that the bilinear form $B_M^\sharp$ and $L_M^\sharp$ satisfy the assumptions of Theorem~4.8-2 in~\cite{Ciarlet2025}, we infer that there exists a unique solution for Problem~\ref{problemLim}, and that the entire sequence $\{\overline{\bm{u}(\varepsilon)}\}_{\varepsilon>0}$ converges to $\bm{\zeta}$ in $\bm{V}_M^\sharp(\omega)$ as $\varepsilon\to 0$.

(x) \emph{Conclusion of the proof.}
The proof of parts~(ix), (x) and~(xi) of Theorem~5.6-1 in~\cite{Ciarlet2000} can now be re-used verbatim and their conclusions thus follow, completing the proof.
\end{proof}

We complete this section with some remarks on the admissibility of the assumption made in Theorem~\ref{asymptotics} according to which the solution of Problem~\ref{problem0} must satisfy the Kirchhoff-Love assumptions in order to carry out the asymptotic analysis. 
		
The primary reason for resorting to the Kirchhoff-Love assumptions is an analytical necessity, as the limit passage for the displacement field takes place in the abstract completion $\bm{V}_M^\sharp(\Omega)$. This functional setting prevents us from applying the Rellich-Kondrachov theorem, which was the key for establishing, in the papers~\cite{CiaPie2018b,CiaMarPie2018,PieJDE2022}, that the limit vector field respects the confinement condition.
The Kirchhoff-Love assumptions allow us to bypass this analytical issue and we are able to establish in Lemma~\ref{average-constraint} that $\overline{\bm{u}(\varepsilon)}\in\bm{U}(\omega)$.

Beyond this analytical necessity, employing the Kirchhoff-Love assumptions is a recognised technique in shell theory, particularly for thin structures. It is consistent with adopting the linearised Kirchhoff-Love assumptions (cf., e.g., equation~(1.3.19) in~\cite{BerBoi1982}) within the framework of linearised elasticity. This approach is not without precedent; for instance, in the study of non-linear elastic models for liquid crystal polymer networks, Nochetto and collaborators~\cite{BNY24} used a similar assumption to resolve analogous mathematical complexities for a problem different from the one we are considering here.

A way to avoid hypothesising the validity of the Kirchhoff-Love assumptions amounts to replacing the original confinement condition with the following physically sound condition:
\begin{equation}
	\label{cc-new}
	\left(\bm{\theta}(y)+x_3^\varepsilon\bm{a}^3(y)+v_i^\varepsilon(x^\varepsilon)\bm{a}^i(y)\right)\cdot\bm{q}\ge 0,\quad\textup{ for a.a. }x^\varepsilon=(y,x_3^\varepsilon)\in\Omega^\varepsilon.
\end{equation}

Note that the vectors $\bm{g}^{i,\varepsilon}(x^\varepsilon)$ appearing in the original constraint formulation~\eqref{cc-original} have been replaced by $\bm{a}^i(y)$ in~\eqref{cc-new}. The confinement condition~\eqref{cc-new}, although \emph{slightly} less accurate than the original one, is physically sound as it suggests, in line with the linearly elastic shell theory, that the shell under consideration is thin and its thickness remains constant under deformation. If a vector field $\bm{v}^\varepsilon=(v_i^\varepsilon)$ satisfies the confinement condition~\eqref{cc-new}, it is straightforward to see that the average across the thickness $\overline{\bm{v}}^\varepsilon$ satisfies the constraint characterising the set $\bm{U}(\omega)$ \emph{without} having to hypothesise the validity of the Kirchhoff-Love assumptions for one such displacement vector field. Indeed, taking the average across the thickness in the left-hand side of~\eqref{cc-new} gives:
\begin{equation*}
	\left(\bm{\theta}(y)+\overline{v}_i^\varepsilon(y)\bm{a}^i(y)\right)\cdot\bm{q}\ge 0,\quad\textup{ for a.a. }y\in\omega.
\end{equation*}

\section{Justification of Koiter's model for linearly elastic generalised membrane shells of the first kind}
\label{Sec:5}

It remains to show that the solution $\bm{u}^\varepsilon$ of Problem~\ref{problem0} \emph{asymptotically behaves like the solution of Koiter's model}. We recall that Koiter's model, which is due to Koiter~\cite{Koiter,Koiter1970}, is based on the Kirchhoff-Love assumptions, which are assumptions of \emph{geometrical nature}, and on the assumption that the estimates derived by F. John~\cite{John1965,John1971} are valid. The most remarkable feature of Koiter's model is that it is a two-dimensional model that is meant to approximate the \emph{original} three-dimensional linearly elastic energy regardless the type of shell under consideration.
For a more detailed description of Koiter's model we refer the reader to Chapter~7 of~\cite{Ciarlet2000}.

The natural functional space where Koiter's model is defined is given by:
\begin{equation*}
	\bm{V}_K(\omega):=\{\bm{\eta}=(\eta_i) \in H^1(\omega) \times H^1(\omega) \times H^2(\omega); \eta_i=\partial_{\nu}\eta_3=0 \textup{ on }\gamma_0\}.
\end{equation*}

For each $\varepsilon>0$, the energy associated with Koiter's model for a linearly elastic generalised membrane shell of the first kind takes the following form:
\begin{equation*}
	J_K^\varepsilon(\bm{\eta}):=\dfrac{\varepsilon}{2}\int_{\omega} a^{\alpha\beta\sigma\tau} \gamma_{\sigma\tau}(\bm{\eta}) \gamma_{\alpha\beta}(\bm{\eta}) \sqrt{a} \dd y +\dfrac{\varepsilon^3}{6}\int_{\omega}a^{\alpha\beta\sigma\tau} \rho_{\sigma\tau}(\bm{\eta})\rho_{\alpha\beta}(\bm{\eta}) \sqrt{a} \dd y -\varepsilon\int_{\omega} \varphi^{\alpha\beta} \gamma_{\alpha\beta}(\bm{\eta}) \sqrt{a} \dd y,
\end{equation*}
for all $\bm{\eta}=(\eta_i) \in \bm{V}_K(\omega)$.

If the shell is required to remain confined in a prescribed half-space, the minimisers for the energy functional $J_K^\varepsilon$ are to be sought in the following non-empty, closed and convex subset of the space $\bm{V}_K(\omega)$ introduced beforehand:
\begin{equation*}
	\bm{U}_K(\omega):=\{\bm{\eta}=(\eta_i)\in \bm{V}_K(\omega); (\bm{\theta}+\eta_i \bm{a}^i)\cdot\bm{q}\ge 0 \textup{ a.e. in }\omega\}.
\end{equation*}

Let us now recall the statement of Theorem~4.1 in~\cite{CiaPie2018b} concerning the identification of a two-dimensional limit model for linearly elastic generalised membrane shells of the first kind subjected to remaining confined in a half-space orthogonal to a unit-vector $\bm{q}$.

\begin{theorem}
	\label{Koitergeneralised}
	Let $\omega$ be a domain in $\mathbb{R}^2$ and let $\bm{\theta} \in \mathcal{C}^3(\overline{\omega};\mathbb{E}^3)$ be an injective immersion such that $\min_{y \in \overline{\omega}}(\bm{\theta} \cdot\bm{q})=d>0$. Consider a family of linearly elastic generalised membrane shells of the first kind with thickness $2\varepsilon$ approaching zero and with each having the same middle surface $\bm{\theta}(\overline{\omega})$, and assume that each shell is subject to a boundary condition of place along a portion of its lateral face, whose middle curve is the set $\bm{\theta}(\gamma_0)$.
	
	Define the spaces
	\begin{align*}
		\bm{V}(\omega)&:=\{\bm{\eta}=(\eta_i) \in \bm{H}^1(\omega);\bm{\eta}=\bm{0} \textup{ on }\gamma_0\},\\
		\bm{V}_M^\sharp(\omega)&:=\textup{completion of $\bm{V}(\omega)$ with respect to $\left|\cdot\right|_\omega^M$},
	\end{align*}
	and, for each $\varepsilon>0$, let $\bm{\zeta}_K^\varepsilon$ denote the unique minimiser of the energy functional $J_K^\varepsilon$ over the set $\bm{U}_K(\omega)$.
	Then the following convergence holds:
	$$
	\bm{\zeta}_K^\varepsilon \to \bm{\zeta},\quad\textup{ in } \bm{V}_M^\sharp(\omega) \textup{ as }\varepsilon \to 0,
	$$
	where $\bm{\zeta}$ denote the unique solution of the following variational problem:
	Find 
	$$
	\bm{\zeta} \in \bm{U}^\sharp(\omega):=\textup{closure of $\bm{U}_K(\omega)$ with respect to $\left|\cdot\right|_\omega^M$}
	$$
	that satisfies the variational inequalities:
	$$
	B_M^\sharp(\bm{\zeta}, \bm{\eta}-\bm{\zeta}) \ge L_M^\sharp(\bm{\eta}-\bm{\zeta}) \quad\textup{ for all }\bm{\eta}=(\eta_i) \in \bm{U}^\sharp(\omega),
	$$
	where $B_M^\sharp(\cdot,\cdot)$ and $L_M^\sharp$ designate the unique continuous linear extensions from $\bm{V}(\omega)$ to $\bm{V}_M^\sharp(\omega)$ of the bilinear form $B_M(\cdot,\cdot)$, and of the linear form $L_M$ defined by
	$$
	L_M(\bm{\eta}):=\int_\omega \varphi^{\alpha\beta}\gamma_{\alpha\beta}(\bm{\eta}) \sqrt{a} \dd y \quad\textup{ for all }\bm{\eta}=(\eta_i) \in \bm{V}(\omega).
	$$
	\qed
\end{theorem}

In order to fully justify Koiter's model for linearly elastic generalised membrane shells of the first kind subjected to remaining confined in a prescribed half-space, we are thus left to show that:
\begin{equation}
	\label{density}
	\bm{U}_M^\sharp(\omega) := \overline{\bm{U}(\omega)}^{|\cdot|_\omega^M}=\bm{U}^\sharp(\omega):=\overline{\bm{U}_K(\omega)}^{|\cdot|_\omega^M}.
\end{equation}

To this end, we will resort to the following density result, established by P. Doktor~\cite{Dok73} in a more general form (see also~\cite{DZ06}).

\begin{theorem}
	\label{DoktorDensity}
	Let $\omega \subset \mathbb{R}^N$, with $N=2$ or $N=3$ be a Lipschitz domain.
	Let $V:=\{\eta \in H^1(\omega);\eta=0 \textup{ on }\gamma_0\}$.
	Then the space $\{\eta \in \mathcal{C}^\infty(\overline{\omega});\eta=0 \textup{ on }\gamma_0\}$ is dense in $V$ with respect to the norm $\|\cdot\|_{H^1(\omega)}$.
	\qed
\end{theorem}

The next theorem, which constitutes the \emph{second main result of this paper}, is devoted to establishing that equality~\eqref{density} holds under the assumption that the domain $\omega$ is a \emph{rectangle}. Let $\hat{\omega}$ denote an extension of $\omega$ obtained by first reflecting (cf., e.g., Section~9.2 in~\cite{Brez11}) a strip of given width with respect to a pair of parallel edges and by then reflecting the strip of the same width associated with the obtained extension with respect to the remaining pair of parallel edges.
In this case, we can extend the injective immersion $\bm{\theta}$ to a new injective mapping $\hat{\bm{\theta}}:\overline{\hat{\omega}}\to\mathbb{E}^3$ as follows: First, we reflect $\bm{\theta}$ along the edges of the rectangle that belong to $\gamma\setminus\gamma_0$, and then we smoothly extend $\bm{\theta}$ via the Whitney extension theorem (cf., e.g., Theorem 2.3.6 of~\cite{Hormander1990} and Lemma~6.3 in~\cite{MeiPie2024} for a similar argument) along the edges that belong to $\gamma_0$. In light of the latter considerations, we have that $\hat{\bm{\theta}} \in \mathcal{C}^3(\overline{\hat{\omega}}\setminus\gamma;\mathbb{E}^3) \cap \mathcal{C}^0(\overline{\hat{\omega}};\mathbb{E}^3)$.

Likewise, we extend the contravariant basis $\{\bm{a}^i\}_{i=1}^3$ associated with the injective immersion $\bm{\theta}$ by reflection along the edges of the boundary $\gamma\setminus\gamma_0$ and by smooth extension along the edges of the rectangle that belong to $\gamma_0$. We denote the extension of the vector field $\bm{a}^i$ by $\hat{\bm{a}}^i$ and we further require that the unit-vector $\bm{q}$ we fixed satisfies
\begin{equation}
	\label{continuity}
	\hat{\bm{a}}^i(\hat{y}) \cdot \bm{q} = \bm{a}^i(y) \cdot\bm{q},
\end{equation}
for all $y \in \omega$ and $\hat{y}\in\hat{\omega}\setminus\omega$ that is antipodal to $y$ with respect to some point $y_0 \in \gamma\setminus\gamma_0$, and for each $1\le i \le 3$. A geometrical interpretation for the condition~\eqref{continuity} is illustrated in Figure~\ref{fig:0} below.
\begin{figure}[H]
	\centering
	\captionsetup[subfigure]{justification=centering}
	\includegraphics*[width=0.5\textwidth]{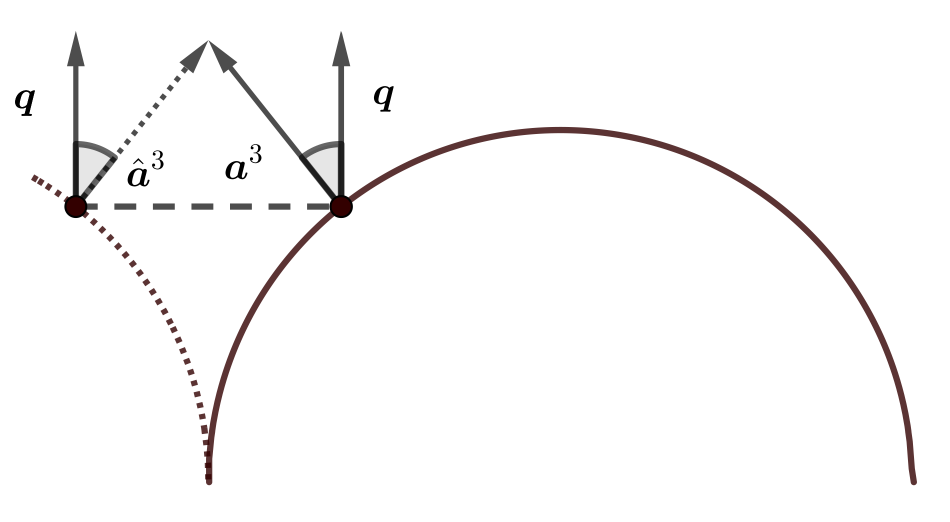}
	\caption{An example where the condition~\eqref{continuity} holds. The figure here represented is a cross section of a portion of a cylinder, whose extension by reflection along one of the straight edges of its boundary is denoted by the dotted pattern. The unit-vector $\bm{q}$ identifying the orthogonal complement to the half-space where the shell has to remain confined is given and it is equal to $(0,0,1)$. The vector $\bm{a}^3(y)$, denoted by a solid line, is reflected (in the sense of Section~9.2 of~\cite{Brez11}) onto the corresponding vector $\hat{\bm{a}}^3(\hat{y})$, denoted by a dotted line, where $\hat{y}\in\overline{\hat{\omega}}\setminus\omega$ is the unique antipodal point with respect to the straight edge of the boundary $\gamma$ that corresponds to an appropriate point $y \in\omega$. We observe that $\bm{a}^3(y)\cdot\bm{q}=\hat{\bm{a}}^3(\hat{y}) \cdot\bm{q}$ even though in general $\bm{a}^3(y) \neq \hat{\bm{a}}^3(\hat{y})$. The remaining vectors of the contravariant basis under consideration can be extended likewise. Other examples of parametrisation of surfaces $\bm{\theta}$ and unit-vectors $\bm{q}$ for which this extension machinery is possible are described in Figure~\ref{fig:3}.}
	\label{fig:0}
\end{figure}

Note that, in general, we have that $\bm{a}^i(y) \neq \hat{\bm{a}}^i(\hat{y})$, although the mappings
$$
\hat{y}\in\overline{\hat{\omega}} \mapsto \hat{\bm{a}}^i(\hat{y}) \cdot\bm{q},\quad 1\le i \le 3,
$$
are continuous.

We say that two points $y_1$, $y_2\in\mathbb{R}^N$ are \emph{antipodal with respect to} $y_0 \in \mathbb{R}^N$ if $y_0$ is the mid-point of the segment joining $y_1$ and $y_2$.

These observations and definitions lead to the formulation of the following \emph{density theorem}, which constitutes the second new result in this paper and that is critical for establishing the justification of Koiter's model for linearly elastic generalised membrane shells subjected to remaining confined in a half-space.

\begin{theorem}
	\label{th:density}
	Let $\bm{\theta} \in \mathcal{C}^3(\overline{\omega};\mathbb{E}^3)$ be an injective immersion such that $\min_{y \in \overline{\omega}}(\bm{\theta} \cdot \bm{q})=d>0$.
	
	Assume that $\omega$ satisfies one of the following, and note that $(a)$ is a special case of $(b)$ when the rectangle is fully clamped:
	\begin{itemize}
		\item[$(a)$] $\gamma_0=\gamma$;
		\item[$(b)$] $\omega$ is a rectangle and $\gamma_0$ consists of whole edges of $\gamma$.
	\end{itemize}
	
	Assume that the unit-vector $\bm{q} \in \mathbb{E}^3$ is such that~\eqref{continuity} holds, i.e.,
	\begin{equation*}
		\hat{\bm{a}}^i(\hat{y}) \cdot \bm{q} = \bm{a}^i(y) \cdot\bm{q},\quad\textup{ for all }1\le i \le 3,
	\end{equation*}
	for all $y \in \omega$ and $\hat{y}\in\hat{\omega}\setminus\omega$ that is antipodal to $y$ with respect to some point $y_0 \in \gamma\setminus\gamma_0$.
	Then
	\begin{equation*}
		\bm{U}_M^\sharp(\omega)=\bm{U}^\sharp(\omega).
	\end{equation*}
\end{theorem}
\begin{proof}
	The sought set equality will be shown by resorting to a density argument, where smaller and smaller dense subsets of $\bm{U}(\omega)$ with respect to the norm $\|\cdot\|_{\bm{H}^1(\omega)}$ are constructed. That the linearly elastic generalised membrane shell under consideration is of the first kind renders the semi-norm $|\cdot|_\omega^M$ a norm over the space $\bm{V}(\omega)$. Moreover, there exists a constant $C=C(\bm{\theta},\omega,\lambda,\mu)>0$ such that:
	\begin{equation}
		\label{norms}
		|\bm{\eta}|_\omega^M \le C \|\bm{\eta}\|_{\bm{H}^1(\omega)},\quad\textup{ for all }\bm{\eta} \in \bm{V}(\omega).
	\end{equation}
	
	The proof is broken into four parts numbered (i)--(iv).
	
	(i) \emph{The set $\bm{U}_M^{(1)}(\omega):=\bm{U}(\omega) \cap \bm{L}^\infty(\omega)$ is dense in the set $\bm{U}(\omega)$ with respect to the norm $\|\cdot\|_{\bm{H}^1(\omega)}$.}
	
	Let $\bm{\eta}\in\bm{U}(\omega)$ be given. For each integer $k\ge 1$, define the vector field $\bm{\eta}^{(k)}=\eta_i^{(k)}$ in such a way that its components satisfy the following relations:
	\begin{equation*}
		\eta_i^{(k)}(y)=
		\begin{cases}
			\eta_i(y)&, \textup{ for a.a. }y\in\omega \textup{ such that }|\eta_j(y)\bm{a}^j(y)|\le k,\\
			\\
			\dfrac{k\eta_i(y)}{|\eta_j(y)\bm{a}^j(y)|}&, \textup{ for a.a. }y\in\omega \textup{ such that }|\eta_j(y)\bm{a}^j(y)|> k.
		\end{cases}
	\end{equation*}
	
	Therefore, it results $|\eta_i^{(k)}\bm{a}^i|\le k$ if $|\eta_j\bm{a}^j|\le k$, and $|\eta_i^{(k)}\bm{a}^i|\equiv k$ if $|\eta_j\bm{a}^j|> k$.
	An application of the Cauchy-Schwarz inequality then gives
	\begin{equation*}
		|\eta_i^{(k)}|\le |\eta_j^{(k)}\bm{a}^j| |\bm{a}_i| \le k \max_{1\le \ell \le 3}\|\bm{a}_\ell\|_{\bm{\mathcal{C}}^1(\overline{\omega})},\quad\textup{ for a.a. }y\in\omega,
	\end{equation*}
	thus showing that $\eta_i^{(k)} \in L^\infty(\omega)$. If $y\in\omega$ is such that $|\eta_j(y)\bm{a}^j(y)|\le k$, then
	\begin{equation*}
		(\bm{\theta}(y)+\eta_i^{(k)}(y)\bm{a}^i(y))\cdot\bm{q}=(\bm{\theta}(y)+\eta_i(y)\bm{a}^i(y))\cdot\bm{q} \ge 0,
	\end{equation*}
	whereas, if $|\eta_j(y)\bm{a}^j(y)|>k$, then
	\begin{equation*}
		(\bm{\theta}(y)+\eta_i^{(k)}(y)\bm{a}^i(y))\cdot\bm{q}=\left(\bm{\theta}(y)+\dfrac{k}{|\eta_j(y)\bm{a}^j(y)|}\eta_i(y)\bm{a}^i(y)\right)\cdot\bm{q} \ge \left(1-\dfrac{k}{|\eta_j(y)\bm{a}^j(y)|}\right)\bm{\theta}\cdot\bm{q}> 0.
	\end{equation*}
	
	Observe that we can write
	\begin{equation*}
		\eta_i^{(k)}=g^{(k)}(u) \eta_i,
	\end{equation*}
	where $u:=|\eta_j\bm{a}^j|^2=\eta_j\eta_\ell a^{j\ell}$ and the function $g^{(k)}:[0,\infty)\to\mathbb{R}$ is defined by:
	\begin{equation*}
		g^{(k)}(u):=
		\begin{cases}
			1&, \textup{ if }0 \le u \le k^2,\\
			\\
			\dfrac{k}{\sqrt{u}}&, \textup{ if }u>k^2.
		\end{cases}
	\end{equation*}
	
	The function $g^{(k)}$ is continuous and it is differentiable in the classical sense in $(0,k^2) \cup (k^2,\infty)$, with derivative given by:
	\begin{equation*}
		\dfrac{\dd g^{(k)}}{\dd u}(u)=
		\begin{cases}
			0&, \textup{ if }0 < u < k^2,\\
			\\
			\dfrac{-k}{2u^{3/2}}&, \textup{ if }u>k^2.
		\end{cases}
	\end{equation*}
	
	Since $\eta_j \in H^1(\omega)$, the Sobolev embedding theorem in the limiting case (Corollary~9.11 in~\cite{Brez11}) ensures that $\eta_j \in L^r(\omega)$ for all $2\le r <\infty$. An application of the generalised H\"older's inequality (cf., e.g., Exercise~3.7 in~\cite{CanDAp}) thus shows that $u\in W^{1,q}(\omega)$, for any $1\le q <2$. Since 
	\begin{equation*}
		\left|\frac{\dd g^{(k)}}{\dd u}\right|<\frac{1}{2k^2},
	\end{equation*}
	we have obtained that $g^{(k)} \in W^{1,\infty}(0,\infty)$ and, therefore, that $g^{(k)}$ is Lipschitz continuous (cf., e.g., Proposition~8.4 of~\cite{Brez11}). Therefore, we are in a position to apply Theorem~2.2 in~\cite{MarMiz1972}, so as to obtain that:
	\begin{equation*}
		g^{(k)}(u) \in W^{1,q}(\omega),\quad 1<q <2.
	\end{equation*}
	
	Furthermore, the weak derivatives of $g^{(k)}(u)$ are given by the standard chain rule (cf., Theorem~2.1 in~\cite{MarMiz1972}):
	\begin{equation*}
		\partial_\beta(g^{(k)}(u))=
		\begin{cases}
			0&,\textup{ if }0<u<k^2,\\
			\\
			-\dfrac{k}{2u^{3/2}}\partial_\beta u&,\textup{ if }u>k^2.
		\end{cases}
	\end{equation*}
	
	In order to show that $\eta_i^{(k)} \in H^1(\omega)$, we will compute the distributional gradient of $\eta_i^{(k)}$ and show that it is square integrable. To this aim, we proceed by density. In light of Theorem~\ref{DoktorDensity}, we can find a sequence $\{\eta_{i,\ell}\}_{\ell\ge 1} \subset \mathcal{C}^\infty(\overline{\omega})$ such that $\eta_{i,\ell}=0$ on $\gamma_0$ and 
	\begin{equation}
		\label{conv1}
		\eta_{i,\ell}\to\eta_i,\quad\textup{ in }H^1(\omega) \textup{ as }\ell\to\infty.
	\end{equation}
	
	We also have
	\begin{equation*}
		\label{est:1}
		\|g^{(k)}(u)\eta_{j,\ell}-g^{(k)}(u)\eta_j\|_{L^2(\omega)} \le \|g^{(k)}(u)\|_{L^\infty(\omega)}\|\eta_{j,\ell}-\eta_j\|_{L^2(\omega)},
	\end{equation*}
	and we note that, thanks to~\eqref{conv1}, the right-hand side above approaches zero as $\ell\to\infty$. Consequently:
	\begin{equation}
		\label{conv4}
		g^{(k)}(u)\eta_{j,\ell} \to g^{(k)}(u)\eta_j,\quad\textup{ in }L^2(\omega) \textup{ as }\ell\to\infty.
	\end{equation}
	
	By H\"older's inequality, we have the following estimate:
	\begin{equation*}
		\|g^{(k)}(u)\partial_\beta\eta_{i,\ell}-g^{(k)}(u)\partial_\beta\eta_i\|_{L^2(\omega)}
			\le \|g^{(k)}(u)\|_{L^\infty(\omega)}\|\partial_\beta\eta_{i,\ell}-\partial_\beta\eta_i\|_{L^2(\omega)}.
	\end{equation*}
	
	Since the right-hand side above approaches zero as $\ell\to\infty$ thanks to~\eqref{conv1}, it follows that:
	\begin{equation}
		\label{est:3}
		g^{(k)}(u)\partial_\beta\eta_{i,\ell} \to g^{(k)}(u)\partial_\beta\eta_i,\quad\textup{ in }L^2(\omega) \textup{ as }\ell\to\infty.
	\end{equation}
	
	Noting that $\partial_\beta u \in L^q(\omega)$ for any $1<q<2$, we define the positive numbers $r:=3q/(2q-2)$ and $p:=3q/(2q+1)$ and note that $1<p<2$, and we note that
	\begin{equation*}
		\dfrac{1}{p}=\dfrac{1}{q}+\dfrac{1}{r}.
	\end{equation*}
	
	An application of the generalised H\"older's inequality then gives:
	\begin{equation}
		\label{est:4}
		\|\partial_\beta g^{(k)}(u)\eta_{i,\ell}-\partial_\beta g^{(k)}(u)\eta_i\|_{L^p(\omega)} \le \|\partial_\beta g^{(k)}(u)\|_{L^q(\omega)}\|\eta_{i,\ell}-\eta_i\|_{L^r(\omega)}\le \dfrac{\|\partial_\beta u\|_{L^q(\omega)}}{2k^2}\|\eta_{i,\ell}-\eta_i\|_{L^r(\omega)},
	\end{equation}
	and note that the right-hand side approaches zero as $\ell\to\infty$ thanks to~\eqref{conv1}. Consequently,
	\begin{equation}
		\label{conv6}
		\partial_\beta g^{(k)}(u)\eta_{i,\ell} \to \partial_\beta g^{(k)}(u)\eta_i,\quad\textup{ in }L^p(\omega) \textup{ as }\ell\to\infty.
	\end{equation}
	
	In light of~\eqref{conv4}--\eqref{conv6}, we have obtained that
	\begin{equation}
		\label{conv7}
		g^{(k)}(u)\eta_{j,\ell} \to g^{(k)}(u)\eta_j,\quad\textup{ in }W^{1,p}(\omega) \textup{ as }\ell\to\infty.
	\end{equation}
	
	Thanks to~\eqref{conv4}--\eqref{conv7}, we note that $\eta_i^{(k)} \in W^{1,p}(\omega)$ with $1<p<2$. Besides, we can express the weak derivatives $\partial_\beta\eta_i^{(k)}$ by means of the classical chain rule by exploiting the density recalled in Theorem~\ref{DoktorDensity}:
	\begin{equation}
		\label{wdk}
		\partial_\beta \eta_i^{(k)}=\partial_\beta(g^{(k)}(u)) \eta_i+g^{(k)}(u) \partial_\beta \eta_i.
	\end{equation}
	
	In light of~\eqref{conv7}, we can thus assert that:
	\begin{equation}
		\label{tr1}
		g^{(k)}(u) \eta_{i,\ell} \to g^{(k)}(u) \eta_i,\quad\textup{ in }L^p(\gamma_0) \textup{ as }\ell\to\infty.
	\end{equation}
	
	Since $g^{(k)}(u) \in W^{1,q}(\omega)$ for some $1<q<2$, there exists a sequence $\{\psi_n\}_{n\ge 0} \subset \mathcal{C}^\infty(\overline{\omega})$ such that $\psi_n \to g^{(k)}(u)$ in $W^{1,q}(\omega)$ as $n\to\infty$. Therefore:
	\begin{equation}
		\label{tr2}
		\psi_n \eta_{i,\ell} \to g^{(k)}(u) \eta_{i,\ell},\quad\textup{ in }W^{1,q}(\omega) \textup{ as }n\to\infty.
	\end{equation}
	
	Since the functions $\psi_n$ and $\eta_{i,\ell}$ are both smooth and since $\eta_{i,\ell}$ vanishes along $\gamma_0$, it follows that $\psi_n \eta_{i,\ell} = 0$ on $\gamma_0$. Therefore, combining \eqref{tr1} and~\eqref{tr2}, we obtain that $\eta_i^{(k)}=0$ on $\gamma_0$.
	
	Let us now show that $\partial_\beta\eta_i^{(k)} \in L^2(\omega)$. Thanks to the uniform boundedness of $g^{(k)}$, the following estimate holds:
	\begin{equation*}
		\left(\int_{\omega}|g^{(k)}(u) \partial_\beta\eta_i|^2\dd y\right)^{1/2}\le \|\partial_\beta\eta_i\|_{L^2(\omega)}<\infty.
	\end{equation*}
	
	A density argument like the ones exploited beforehand allows us to explicitly compute $\partial_\beta g^{k}(u)$ in terms of the components $\eta_i$ (cf. \cite{CiaMarPie2018}), and to obtain the following estimate:
	\begin{equation*}
		\begin{aligned}
			&\left(\int_{\omega}|\partial_\beta g^{(k)}(u) \eta_i|^2\dd y\right)^{1/2}\le\dfrac{\max_{1\le i \le 3}\left(\|\bm{a}_i\|_{\bm{\mathcal{C}}^0(\overline{\omega})}\right)}{2k}\|\partial_\beta a^{j\ell}\|_{\mathcal{C}^0(\overline{\omega})}\|\eta_j\eta_\ell\|_{L^2(\{|\eta_j\bm{a}^j|>k\})}\\
		&\quad+\max_{1\le i \le 3}\left(\|\bm{a}_i\|_{\bm{\mathcal{C}}^0(\overline{\omega})}\right) \|\bm{a}^j\|_{\bm{\mathcal{C}}^0(\overline{\omega})}\|\partial_\beta \eta_j\|_{L^2(\{|\eta_j\bm{a}^j|>k\})}<\infty.
		\end{aligned}
	\end{equation*}
	
	It remains to show that $\|\bm{\eta}^{(k)}-\bm{\eta}\|_{\bm{H}^1(\omega)} \to 0$ as $k \to \infty$.
	To begin with, we observe that the relations
	\begin{equation*}
		\begin{aligned}
		\eta_i^{(k)} - \eta_i = 0&,\quad \textup{ if } |\eta_j \bm{a}^j| \le k ,\\
		|\eta_i^{(k)} - \eta_i| = \left(1-\frac{k}{|\eta_j \bm{a}^j|}\right) |\eta_i| \le |\eta_i|&,\quad \textup{ if }|\eta_j\bm{a}^j|>k,
		\end{aligned}
	\end{equation*}
	imply that
	\begin{equation*}
		\|\eta_i^{(k)}-\eta_i \|_{L^2(\omega)} \le \|\eta_i\|_{L^2(\{|\eta_j\bm{a}^j|>k\})},
	\end{equation*}
	so that the following convergence holds:
	\begin{equation*}
		\|\eta_i^{(k)}-\eta_i\|_{L^2(\omega)} \to 0,\quad \textup{ as } k \to \infty.
	\end{equation*}
	
	Second, we observe that the definition of the functions $g^{(k)}$, the expressions of the weak derivatives $\partial_\beta \eta_i^{(k)}$, and the estimates established above, altogether show that
	\begin{equation}
		\label{chainrule}
		\partial_\beta \eta_i^{(k)} - \partial_\beta \eta_i = (g^{(k)}(u)-1) \partial_\beta \eta_i +\partial_\beta g^{(k)}(u)\eta_i.
	\end{equation}
	
	Since
	\begin{equation*}
		\|(g^{(k)}(u)-1) \partial_\beta \eta_i\|_{L^2(\omega)}=
		\left\|\left( \frac{k}{|\eta_j\bm{a}^j|}-1\right) \partial_\beta\eta_i\right\|_{L^2(\{|\eta_j\bm{a}^j|>k\})},
	\end{equation*}
	and since, in light of~\eqref{chainrule},
	\begin{equation*}
		\partial_\beta g^{(k)}(u)\eta_i\in L^2(\omega),
	\end{equation*}
	the convergence $\|\partial_\beta \eta_i^{(k)} - \partial_\beta \eta_i\|_{L^2(\omega)}\to 0$ as $k\to\infty$ easily follows.
	
	(ii) \emph{Define the set:}
	\begin{equation*}
		\begin{aligned}
			\bm{U}_M^{(2)}(\omega)&:=\{\bm{\eta}=(\eta_i)\in\bm{U}_M^{(1)}(\omega); \textup{ there exists a number }\delta(\bm{\eta})>0\\
			&\qquad\textup{ for which }\bm{\eta}(y)=\bm{0} \textup{ for a.a. } y\in\omega \textup{ such that }\textup{dist}(y,\gamma_0)<\delta(\bm{\eta})\\
			&\qquad\textup{ and such that }(\bm{\theta}+\eta_i\bm{a}^i)\cdot\bm{q}\ge d\delta(\bm{\eta}) \textup{ a.e. in }\omega\}.
		\end{aligned}
	\end{equation*}
	
	\emph{Then, the set $\bm{U}_M^{(2)}(\omega)$ is dense in the set $\bm{U}_M^{(1)}(\omega)$ with respect to the norm $\|\cdot\|_{\bm{H}^1(\omega)}$.}
	
	Let $\bm{\eta} \in \bm{U}_M^{(1)}(\omega)$ be given. For each integer $k\ge 1$, define the function $f^{(k)}:\overline{\omega}\to\mathbb{R}$ by:
	\begin{equation*}
		f^{(k)}(y):=
		\begin{cases}
			0&,\textup{ if }0\le \textup{dist}(y,\gamma_0)\le \dfrac{1}{k},\\
			\\
			k\textup{dist}(y,\gamma_0)-1&,\textup{ if }\dfrac{1}{k}\le \textup{dist}(y,\gamma_0)\le\dfrac{2}{k},\\
			\\
			1&,\textup{ if } \textup{dist}(y,\gamma_0)>\dfrac{2}{k}.
		\end{cases}
	\end{equation*}
	
	Clearly $f^{(k)} \in \mathcal{C}^0(\overline{\omega})$ and $0\le f^{(k)}(y)\le 1$, for all $y\in\overline{\omega}$. Since the function $\textup{dist}(\cdot,\gamma_0):\overline{\omega} \to \mathbb{R}$ is Lipschitz continuous with Lipschitz constant $L=1$, an application of the Rademacher theorem (cf., e.g., Theorem~9.2-2 in~\cite{Ciarlet2025}) gives:
	\begin{equation*}
		\left|\partial_\beta\left(\textup{dist}(\cdot,\gamma_0)\right)\right|\le 1,\quad\textup{ a.e. in }\omega.
	\end{equation*}
	
	As a result
	\begin{equation*}
		\partial_\beta f^{(k)}(y)=
		\begin{cases}
			0&,\textup{ if }0< \textup{dist}(y,\gamma_0)< \dfrac{1}{k},\\
			\\
			\le k&,\textup{ if }\dfrac{1}{k}< \textup{dist}(y,\gamma_0)<\dfrac{2}{k},\\
			\\
			0&,\textup{ if }\textup{dist}(y,\gamma_0)>\dfrac{2}{k},
		\end{cases}
	\end{equation*}
	which in turn implies that $f^{(k)}\in W^{1,\infty}(\omega)$ for each integer $k\ge 1$. Since the given vector field $\bm{\eta}=(\eta_i)$ is such that $\eta_i\in H^1(\omega)\cap L^\infty(\omega)$, we can apply the chain rule formula in Sobolev spaces (cf., e.g., Proposition~9.4 in~\cite{Brez11}) so as to infer that the functions
	\begin{equation*}
		\eta_i^{(k)}:=\left(1-\dfrac{1}{k}\right)f^{(k)}\eta_i \in H^1(\omega).
	\end{equation*}
	
	For each $s>0$, define the following set:
	\begin{equation*}
		\omega_s:=\{y\in\omega;\,\textup{dist}(y,\gamma_0)>s\}.
	\end{equation*}
	
	Clearly $\eta_i^{(k)}=0$ a.e. in $\omega\setminus\overline{\omega_{1/k}}$ and hence $\eta_i^{(k)}=0$ on $\gamma_0$ \emph{a fortiori}.
	It is also easy to show that, for a.a. $y\in\omega$, the geometrical constraint is valid:
	\begin{equation*}
		\begin{aligned}
			&(\bm{\theta}(y)+\eta_i^{(k)}(y)\bm{a}^{i}(y))\cdot\bm{q}=\dfrac{1}{k}\bm{\theta}(y)\cdot\bm{q}+\left(1-\dfrac{1}{k}\right)\left(1-f^{(k)}(y)\right)\bm{\theta}(y)\cdot\bm{q}\\
			&\qquad+\left(1-\dfrac{1}{k}\right)f^{(k)}(y)(\bm{\theta}(y)+\eta_i(y)\bm{a}^i(y))\cdot\bm{q}\ge \dfrac{d}{k}>0.
		\end{aligned}
	\end{equation*}
	
	Therefore, we have shown that $\bm{\eta}^{(k)} \in \bm{U}_M^{(2)}(\omega)$, for each integer $k\ge 1$. Let us now consider the difference
	\begin{equation*}
		\bm{\eta}^{(k)}-\bm{\eta}=
		\begin{cases}
			-\bm{\eta}&, \textup{ in }\omega\setminus\overline{\omega_{1/k}},\\
			\\
			\left[\left(1-\dfrac{1}{k}\right)f^{(k)}-1\right]\bm{\eta}&, \textup{ in }\omega_{1/k}\setminus\overline{\omega_{2/k}},\\
			\\
			-\dfrac{1}{k}\bm{\eta}&, \textup{ in }\omega_{2/k},
		\end{cases}
	\end{equation*}
	and let us observe that the following estimates hold:
	\begin{equation*}
		\begin{aligned}
			&\|\bm{\eta}-\bm{\eta}^{(k)}\|_{\bm{L}^2(\omega)}=\|\bm{\eta}\|_{\bm{L}^2(\omega\setminus\overline{\omega_{1/k}})}+\left\|\left[\left(1-\dfrac{1}{k}\right)f^{(k)}-1\right]\bm{\eta}\right\|_{\bm{L}^2(\omega_{1/k}\setminus\overline{\omega_{2/k}})}+\dfrac{1}{k}\|\bm{\eta}\|_{\bm{L}^2(\omega_{2/k})}\\
			&\le \|\bm{\eta}\|_{\bm{L}^2(\omega\setminus\overline{\omega_{1/k}})}+2\left(\int_{\omega}|\bm{\eta}|^2 \chi_{\omega_{1/k}\setminus\overline{\omega_{2/k}}}\dd y\right)^{1/2} +\dfrac{1}{k}\|\bm{\eta}\|_{\bm{L}^2(\omega_{2/k})}.
		\end{aligned}
	\end{equation*}
	
	Observe that the right-hand side of the previous estimate tends to zero as $k\to\infty$; indeed, the convergence of the third term is straightforward and the convergence of the first two terms follows directly from the dominated convergence theorem. This shows that $\bm{\eta}^{(k)} \to \bm{\eta}$ in $\bm{L}^2(\omega)$ as $k\to\infty$.
	Then, the functions
	\begin{equation*}
		\partial_\beta\eta_i^{(k)}-\partial_\beta\eta_i=
		\begin{cases}
			-\partial_\beta\eta_i&, \textup{ in }\omega\setminus\overline{\omega_{1/k}},\\
			\\
			\left[\left(1-\dfrac{1}{k}\right)f^{(k)}-1\right]\partial_\beta\eta_i+\left(1-\dfrac{1}{k}\right)(\partial_\beta f^{(k)}) \eta_i&, \textup{ in }\omega_{1/k}\setminus\overline{\omega_{2/k}},\\
			\\
			-\dfrac{1}{k}\partial_\beta\eta_i&, \textup{ in }\omega_{2/k},
		\end{cases}
	\end{equation*}
	satisfy
	\begin{equation}
		\label{est:5}
		\begin{aligned}
			&\|\partial_\beta\eta_i^{(k)}-\partial_\beta\eta_i\|_{L^2(\omega)}\le\|\partial_\beta\eta_i\|_{L^2(\omega\setminus\overline{\omega_{1/k}})}+\dfrac{1}{k}\|\partial_\beta\eta_i\|_{L^2(\omega_{2/k})}\\
			&\qquad+\left\|\left(\left(1-\dfrac{1}{k}\right)f^{(k)}-1\right)\partial_\beta\eta_i\right\|_{L^2(\omega_{1/k}\setminus\overline{\omega_{2/k}})}\\
			&\qquad+\left\|\left(1-\dfrac{1}{k}\right)(\partial_\beta f^{(k)})\eta_i\right\|_{L^2(\omega_{1/k}\setminus\overline{\omega_{2/k}})}\\
			&\le \|\partial_\beta\eta_i\|_{L^2(\omega\setminus\overline{\omega_{1/k}})}+\dfrac{1}{k}\|\partial_\beta\eta_i\|_{L^2(\omega_{2/k})}+2\|\partial_\beta\eta_i\|_{L^2(\omega_{1/k}\setminus\overline{\omega_{2/k}})}+k\|\eta_i\|_{L^2(\omega_{1/k}\setminus\overline{\omega_{2/k}})}.
		\end{aligned}
	\end{equation}
	
	It is clear that the first three terms in the right-hand side of~\eqref{est:5} converge to zero as $k\to\infty$ by the dominated convergence theorem.
	The analysis of the last term is more involving and hinges on the \emph{geometrical assumptions} on the set $\omega$ set forth in the statement of this theorem.
	
	In what follows, we limit ourselves to considering case~(b), i.e., the case where $\omega$ is a rectangle and $\gamma_0$ is either an edge of the rectangle or contains the union of two parallel edges. We also assume that the origin is placed at one of the vertices of $\gamma_0$. We observe that, for each integer $k\ge 1$, the set $\omega\setminus\overline{\omega_{1/k}}$ is the union of rectangles such that one of their edges is one of the edges constituting the portion of the boundary $\gamma_0$ where the boundary conditions of place are imposed. Examples of admissible configurations for the sets $\omega$ and $\omega\setminus\overline{\omega_{2/k}}$ are depicted in Figure~\ref{fig:1} below.
	\begin{figure}[H]
		\captionsetup[subfigure]{justification=centering}
		\centering
		\begin{subfigure}[b]{0.4\linewidth}
			\centering
			\includegraphics[width=0.4\textwidth]{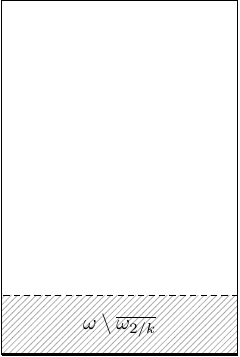}
			\subcaption{The set $\gamma_0$ consists of one edge.}
		\end{subfigure}%
		\begin{subfigure}[b]{0.4\linewidth}
			\centering
			\includegraphics[width=0.4\textwidth]{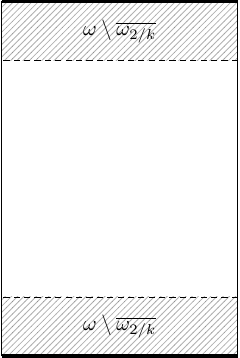}
			\subcaption{The set $\gamma_0$ is the union of two parallel edges.}
		\end{subfigure}%
	\end{figure}
	\begin{figure}[H]
		\ContinuedFloat
		\captionsetup[subfigure]{justification=centering}
			\centering
		\begin{subfigure}[b]{0.4\linewidth}
			\centering
			\includegraphics[width=0.4\textwidth]{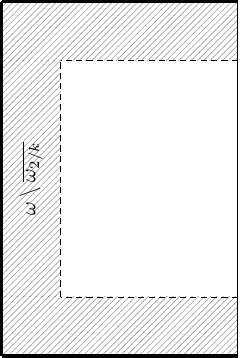}
			\subcaption{The set $\gamma_0$ is the union of three edges.}
		\end{subfigure}%
		\begin{subfigure}[b]{0.4\linewidth}
			\centering
			\includegraphics[width=0.4\textwidth]{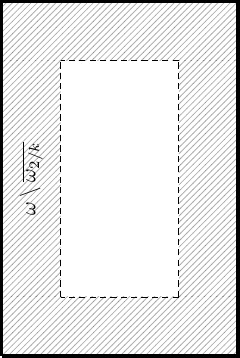}
			\subcaption{The set $\gamma_0$ coincides with $\gamma$.}
		\end{subfigure}%
		\caption{The sets $\omega$ and $\omega\setminus\overline{\omega_{2/k}}$ for different instances of $\gamma_0$, which is denoted here by a thick black line. In~(a), the set $\gamma_0$ coincides with one edge of the rectangular domain, in~(b) the set $\gamma_0$ is the union of two parallel edges, in~(c) the set $\gamma_0$ contains the union of two parallel edges and, finally, in~(d), the set $\gamma_0$ coincides with the entire boundary $\gamma$.}
		\label{fig:1}
	\end{figure}
	
	The assumption on the geometry of the domain will turn out to be essential for establishing the following inequality:
	\begin{equation}
		\label{Poincare}
		k\|\xi\|_{L^2(\omega\setminus\overline{\omega_{2/k}})} \le C\|\xi\|_{H^1(\omega\setminus\overline{\omega_{2/k}})},
	\end{equation}
	for all $\xi\in H^1(\omega)$ such that $\xi=0$ on $\gamma_0$, for some $C>0$ independent of $k$. To establish inequality~\eqref{Poincare}, we carry out some of the computations necessary to establish the Poincar\'e-Friedrichs inequality (cf., e.g., Theorem~8.3-3 in~\cite{Ciarlet2025}). For the sake of clarity, we present the main steps for establishing the sought inequality~\eqref{Poincare}.
	Fix $y_0\in\gamma_0$, and let $y\in\omega\setminus\overline{\omega_{2/k}}$ be such that $y=y_0-c\nu(y_0)$, for some $c>0$.
	To fix the ideas, we consider the strip of width $2/k$ in Figure~\ref{fig:1}(a).
	The absolute continuity along the lines (cf., e.g., \cite{Leoni2017}) enjoyed by the function $\xi$ and the fact that $\xi(y_0)=0$ give:
	\begin{equation*}
		\xi(y)=\int_{y_{0,2}}^{y_2} \partial_2\xi(y_1,t) \dd t,\quad\textup{ for a.a. } y_1\in(0,\textup{length }\gamma_0).
	\end{equation*}
	
	Consequently, we have that:
	\begin{equation*}
		\begin{aligned}
		&|\xi(y)|^2\le \left(\int_{y_{0,2}}^{y_2} |\partial_2\xi(y_1,t)| \dd t\right)^2\le(y_2-y_{0,2})\int_{y_{0,2}}^{y_2} |\partial_2\xi(y_1,t)|^2 \dd t\\
		&\le (y_2-y_{0,2})\int_{y_{0,2}}^{y_{0,2}-\frac{2}{k}\nu_2(y_0)} |\partial_2\xi(y_1,t)|^2 \dd t,
		\end{aligned}
	\end{equation*}
	where the second estimate is due to H\"older's inequality.
	The latter estimate in turn implies that
	\begin{equation*}
		\int_{y_{0,2}}^{y_{0,2}-\frac{2}{k}\nu_2(y_0)} |\xi(y)|^2 \dd y_2 \le \dfrac{2}{k^2} \int_{y_{0,2}}^{y_{0,2}-\frac{2}{k}\nu_2(y_0)} |\partial_2\xi(y_1,t)|^2 \dd t,\quad\textup{ for a.a. }y_1\in(0,\textup{length }\gamma_0),
	\end{equation*}
	so that it is immediate to establish the following estimate:
	\begin{equation*}
		\begin{aligned}
		&\int_{0}^{\textup{length }\gamma_0} \int_{y_{0,2}}^{y_{0,2}-\frac{2}{k}\nu_2(y_0)} |\xi(y)|^2 \dd y_2 \dd y_1 \le \dfrac{2}{k^2}\int_{0}^{\textup{length }\gamma_0} \int_{y_{0,2}}^{y_{0,2}-\frac{2}{k}\nu_2(y_0)} |\partial_2\xi(y_1,t)|^2 \dd t\\
		&= \dfrac{2}{k^2} \|\partial_2\xi\|_{L^2(\omega\setminus\overline{\omega_{2/k}})}^2.
		\end{aligned}
	\end{equation*}
	
	Therefore, the claimed inequality~\eqref{Poincare} holds with $C:=\sqrt{2}$. We are now in a position to analyse the last term in the right-hand side of~\eqref{est:5}:
	\begin{equation*}
		k\|\eta_i\|_{L^2(\omega\setminus\overline{\omega_{2/k}})}\le \sqrt{2}\|\eta_i\|_{H^1(\omega\setminus\overline{\omega_{2/k}})},
	\end{equation*}
	and we observe that the right-hand side of the latter term converges to zero as $k\to\infty$ by the dominated convergence theorem, thus proving that $\partial_\beta\eta_i^{(k)} \to\partial_\beta\eta_i$ in $L^2(\omega)$ as $k\to\infty$. Therefore, the sought density holds.
	
	The case corresponding to Figures~\ref{fig:1}(b)--(d) can be treated analogously, as well as the case where the domain has smooth boundary and the boundary condition of place is imposed along the entire boundary $\gamma$.
	
	(iii) \emph{Denote by $\mathcal{C}_{\gamma_0}^\infty(\overline{\omega})$ the space of continuously differentiable functions that vanish in a neighbourhood of $\gamma_0$.
		Then, the set $\bm{U}_M^{(3)}(\omega):=\bm{\mathcal{C}}_{\gamma_0}^\infty(\overline{\omega}) \cap \bm{U}(\omega)$ is dense in the set $\bm{U}_M^{(2)}(\omega)$ with respect to the norm $\|\cdot\|_{\bm{H}^1(\omega)}$.}
	
	Given two domains $\tilde{\Omega}\subset\mathbb{R}^N$ and $\Omega\subset\mathbb{R}^N$, $N \ge 1$, the notation $\tilde{\Omega}\subset\subset\Omega$ means that $\tilde{\Omega}\subset\Omega$ \emph{and} that $\textup{dist}(\partial\tilde{\Omega},\partial\Omega)>0$.
	
	To fix ideas, let us once again consider the case where $\omega$ is a rectangle and the boundary conditions under consideration are as in Figure~\ref{fig:1}(a), and let $\bm{\eta}\in\bm{U}_M^{(2)}(\omega)$ be given. Let $k>2/\delta(\bm{\eta})$ be an integer, where $\delta(\bm{\eta})$ is the positive number appearing in the definition of $\bm{U}_M^{(2)}(\omega)$. Let $\rho_{1/k}$ be a standard mollifier (cf., e.g., \cite{Brez11}) such that $\textup{supp }\rho_{1/k} \subset\subset B(0;1/k)$. Following the ideas leading to Figure~6 on page~275 of~\cite{Brez11}, we extend the original domain $\omega$ by reflection along the edges, so as to obtain a new rectangular domain $\hat{\omega}$ like the one depicted in Figure~\ref{fig:2} below.
	
	\begin{figure}[H]
		\captionsetup[subfigure]{justification=centering}
		\centering
		\begin{subfigure}[t]{0.3\linewidth}
			\centering
			\includegraphics[width=0.55\textwidth]{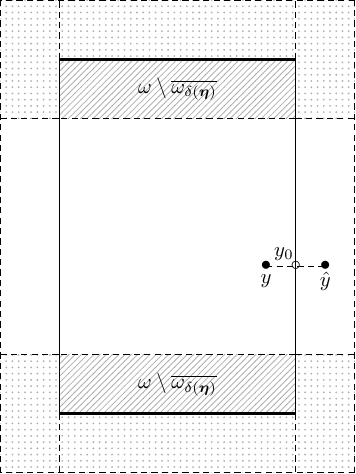}
			\subcaption{Extension in the case where the set $\gamma_0$ is the union of two parallel edges.}
		\end{subfigure}%
		\hspace{0.5cm}
		\begin{subfigure}[t]{0.3\linewidth}
			\centering
			\includegraphics[width=0.55\textwidth]{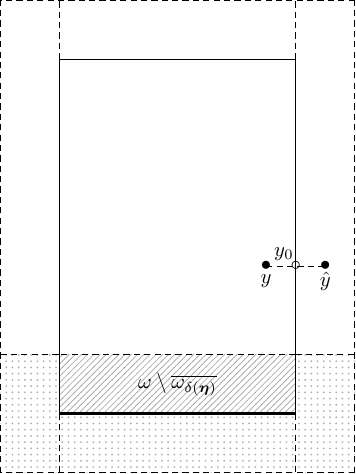}
			\subcaption{Extension in the case where the set $\gamma_0$ consists of one edge.}
		\end{subfigure}%
		\hspace{0.5cm}
		\begin{subfigure}[t]{0.3\linewidth}
			\centering
			\includegraphics[width=0.55\textwidth]{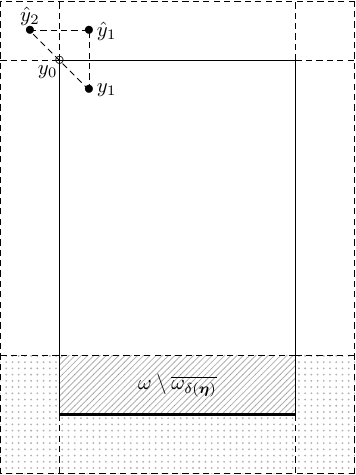}
			\subcaption{Extension in the case where the set $\gamma_0$ consists of one edge. Reflection in proximity of a corner of $\omega$.}
		\end{subfigure}%
		\caption{The set $\hat{\omega}$ is an extension by reflection of the original domain $\omega$. First, we reflect the strips of width $\delta(\bm{\eta})$ across the edge of $\gamma_0$ it is adjacent to. Second, we reflect the strips of width $\delta(\bm{\eta})$ across the edges of the intermediary extension that contain the portion of boundary $\gamma\setminus\gamma_0$. The hatched patterns denote the portion of $\omega$ where the given function $\bm{\eta}\in\bm{U}_M^{(2)}(\omega)$ vanishes. The dotted pattern denotes the points in the domain extension where the extension of $\bm{\eta}$ by reflection vanishes (cf., e.g., Theorem~9.7 in~\cite{Brez11}). The point $\hat{y}$, which is represented by a black dot, is the antipodal point corresponding to $y\in \omega$ with respect to $y_0\in\gamma$, denoted by a hollow dot. This means that there exists a number $c>0$ such that $\hat{y}=y+2c\nu(y_0)$. In the case illustrated in~(a), where the set $\gamma_0$ is the union of two parallel edges, the extension vanishes near the corners of $\omega$ and the extension by reflection of the contravariant basis with respect to the remaining edges is performed in the manner described at the beginning of section~\ref{Sec:5}. If the set $\gamma_0$ consists of one edge only, then the extension by reflection of the contravariant basis is performed in the usual manner when $y_0$ is not a corner point~(b). Otherwise, the extension by reflection of the contravariant basis is carried out with respect to two points $y_1$ and $\hat{y}_2$ that are antipodal with respect to the corner point $y_0$ under consideration~(c). The latter is equivalent to considering the extension by reflection of the contravariant basis at the point $y_1$ to the point $\hat{y}_1$ first, and then considering the extension by reflection of the contravariant basis at the point $\hat{y}_1$ to the point $\hat{y}_2$.}
		\label{fig:2}
	\end{figure}
	
	After multiplying the extension by reflection (cf. Section~9.2 of~\cite{Brez11}) of $\bm{\eta}$ to $\hat{\omega}$ (cf., e.g., Theorem~9.7 of~\cite{Brez11}) by an appropriate cut-off function which is equal to 1 in a set $\omega_1$ such that $\omega \subset\subset\omega_1 \subset\subset\hat{\omega}$, we infer that the regularisation of $\eta_i$, which are denoted and defined by
	\begin{equation*}
		\eta_i^{(k)}:=\rho_{1/k} \ast \eta_i,
	\end{equation*}
	are of class $\mathcal{C}^\infty(\overline{\omega})$ and satisfy $\eta_i^{(k)}=0$ in $\omega\setminus\overline{\omega_{\delta(\bm{\eta})/2}}$. Moreover, the regularisation theory (cf., e.g., \cite{Brez11}) gives that:
	\begin{equation*}
		\eta_i^{(k)} \to \eta_i,\quad\textup{ in }H^1(\omega)\textup{ as }k\to\infty.
	\end{equation*}
	
	To complete the proof, we need to verify that $\bm{\eta}^{(k)}$ satisfies the geometrical constraint ensuring that the deformed reference configuration of the shell is contained in the prescribed half-space. The case where $\textup{dist}(y,\gamma)>\delta(\bm{\eta})/2$ is straightforward to treat, as the support of the mollifier $\rho_{1/k}$ does not intersect the complement of $\omega$. To complete the proof, let us then consider the case where $y\in \omega$ such that $\textup{dist}(y,\gamma)\le\delta(\bm{\eta})/2$ and, in the same spirit as Figure~\ref{fig:2}, we consider the point $\hat{y}$ that is antipodal to $y$ with respect to the edge of $\gamma$ whose distance from $y$ is minimal. 
	
	The definition of $\bm{U}_M^{(2)}(\omega)$, the assumed properties for $\hat{\bm{\theta}}$, assumption~\eqref{continuity}, the property of the reflection and the fact that $\eta_i$ extends by zero in the vicinity of the edges where the boundary condition of place is imposed give:
	\begin{equation}
		\label{est:6}
		\begin{aligned}
			&\left(\bm{\theta}(y)+\eta_i^{(k)}(y)\bm{a}^i(y)\right)\cdot\bm{q}
			=\bm{\theta}(y)\cdot\bm{q}+\left(\int_{B(0;1/k)} \rho_{1/k}(z) \hat{\eta}_i(y-z) \dd z\right)\bm{a}^i(y)\cdot\bm{q}\\
			&=\int_{B(0;1/k)}\rho_{1/k}(z)\left(\hat{\bm{\theta}}(y-z)+\hat{\eta}_i(y-z)\hat{\bm{a}}^i(y-z)\right) \cdot\bm{q}\dd z\\
			&\quad+\int_{B(0;1/k)}\rho_{1/k}(z)(\hat{\bm{\theta}}(y)-\hat{\bm{\theta}}(y-z))\cdot\bm{q}\dd z
			+\int_{B(0;1/k)}\rho_{1/k}(z)\hat{\eta}_i(y-z)(\hat{\bm{a}}^i(y)-\hat{\bm{a}}^i(y-z))\cdot\bm{q}\dd z\\
			&\ge d\delta(\bm{\eta})-\max_{z\in B(0;1/k)}|\hat{\bm{\theta}}(y)-\hat{\bm{\theta}}(y-z)|-\|\hat{\eta}_i\|_{L^2(\mathbb{R}^2)}\left(\max_{\substack{1\le i\le3\\z\in B(0;1/k)}}\left|\hat{\bm{a}}^i(y)\cdot\bm{q}-\hat{\bm{a}}^i(y-z)\cdot\bm{q}\right|\right).
		\end{aligned}
	\end{equation}
	
	The continuity of $\hat{\bm{\theta}}$ and the continuity property assumed in~\eqref{continuity} in turn show that the second and third term in the right-hand side of~\eqref{est:6} tend to zero as $k\to\infty$. As a result, for $k$ sufficiently large, we obtain that:
	\begin{equation*}
		\left(\bm{\theta}(y)+\eta_i^{(k)}\bm{a}^i(y)\right)\cdot\bm{q}\ge 0,
	\end{equation*}
	so that $\bm{\eta}^{(k)} \in \bm{U}_M^{(3)}(\omega)$ by the arbitrariness of $y\in\omega$.
	
	(iv) \emph{Conclusion of the proof.}
	
	In parts~(i)-(iii), we have shown that, for a given $\bm{\eta}\in \bm{U}(\omega)$, there exists a sequence $\left\{\bm{\eta}^{(k)}\right\}_{k\ge 1} \subset \bm{\mathcal{C}}^\infty_{\gamma_0}(\overline{\omega}) \cap\bm{U}(\omega)$ such that:
	\begin{equation*}
		\bm{\eta}^{(k)} \to \bm{\eta},\quad\textup{ in }\bm{H}^1(\omega) \textup{ as }k\to\infty.
	\end{equation*}
	
	Thanks to the inequality~\eqref{norms}, we have that $\left|\bm{\eta}^{(k)}-\bm{\eta}\right|_\omega^M \to 0$, as $k\to\infty$.
	Therefore, we observe that for every $\bm{\eta}^\sharp \in \bm{U}_M^\sharp(\omega)$ and for every $\varepsilon>0$ there exists, by definition, an element $\bm{\eta}\in \bm{U}(\omega)$ such that $|\bm{\eta}^\sharp-\bm{\eta}|_\omega^M <\varepsilon/2$. Up to \emph{ad hoc} isometry (cf., e.g. Theorems~1.12-4 and~3.1-2 in~\cite{Ciarlet2025}), inequality~\eqref{norms} gives that in correspondence of the fixed $\varepsilon$ and of the element $\bm{\eta}\in\bm{U}(\omega)$, we can find another element $\tilde{\bm{\eta}}\in \bm{\mathcal{C}}^\infty_{\gamma_0}(\overline{\omega}) \cap \bm{U}(\omega)$ such that $|\bm{\eta}-\tilde{\bm{\eta}}|_\omega^M <\varepsilon/2$. In conclusion, we have that for every $\bm{\eta}^\sharp \in \bm{U}_M^\sharp(\omega)$ and for every $\varepsilon>0$ there exists $\tilde{\bm{\eta}}\in \bm{\mathcal{C}}^\infty_{\gamma_0}(\overline{\omega}) \cap \bm{U}(\omega)$ such that $|\bm{\eta}^\sharp-\tilde{\bm{\eta}}|_\omega^M <\varepsilon$, thus showing that:
	\begin{equation*}
		\overline{\bm{\mathcal{C}}^\infty_{\gamma_0}(\overline{\omega}) \cap \bm{U}(\omega)}^{|\cdot|_\omega^M}=\bm{U}_M^\sharp(\omega).
	\end{equation*}
	
	Therefore, the announced equality $\bm{U}_M^\sharp(\omega)=\bm{U}^\sharp(\omega)$ immediately follows by observing that the iterates constructed in~(iii) vanish near $\gamma_0$, thus satisfying a homogeneous boundary condition of Neumann type along $\gamma_0$.
\end{proof}

We now comment the assumption of \emph{geometrical nature} set forth in the statement of Theorem~\ref{th:density}.

To begin with, let us now consider the assumption on the shape of the domain $\omega$ and of the portion $\gamma_0$ of the boundary $\gamma$ where the boundary conditions of place are imposed. Observe that if $\gamma$ is smooth and $\textup{meas}(\gamma\setminus\gamma_0) >0$, then the set $\omega\setminus\overline{\omega_{2/k}}$ considered in step~(ii) of the proof would include a portion of the set $\gamma\setminus\gamma_0$. By doing so, it is not possible to carry out the calculations leading to the inequality~\eqref{Poincare} since the function under consideration does not vanish on $\gamma \setminus \gamma_0$.

Second, assuming that $\omega$ is a rectangle and $\gamma_0$ is the union of its edges allows us to apply the argument based on the absolute continuity along the lines for establishing part~(ii) of Theorem~\ref{th:density}. For general polygons and for instances of $\gamma_0$ that cover only a portion of an edge, these computations cannot be carried out.

Third, and finally, the fact that the linearly elastic generalised membrane shells we are considering are of the first kind renders the mapping $|\cdot|_\omega^M$ a norm on $\bm{V}(\omega)$, thus ensuring the claimed density.

A first example of middle surface satisfying the assumptions critical to establish Theorem~\ref{th:density} is, for instance, a portion of a spherical cap such that the portion of contour where the boundary conditions of place are imposed lies on a plane that is parallel to the plane normal to the \emph{given} unit-vector $\bm{q}$. For $\bm{q}=(0,0,1)$, a parametrisation for one such surface is given by
\begin{equation*}
	y=(y_1,y_2) \in [0,\pi] \times \left[0.1,c\dfrac{\pi}{2}\right] \mapsto \bm{\theta}(y):=\left(\cos y_1\sin y_2, \sin y_1 \sin y_2, \cos y_2\right),\quad\textup{ where }0<c<1,
\end{equation*}
and the portion $\gamma_0$ of the boundary $\gamma$ where the boundary conditions are imposed is given by $\gamma_0:=[0,\pi]\times \{0.1\}$.

A second example of middle surface satisfying the assumptions critical to establish Theorem~\ref{th:density} is, for instance, a portion of a cylinder whose straight edges lie on the same plane, and this plane is parallel to plane orthogonal to the \emph{given} unit-vector $\bm{q}$. For $\bm{q}=(0,0,1)$, a parametrisation for one such surface is given by
\begin{equation*}
	y=(y_1,y_2) \in [0.1,\pi-0.1] \times [0,2] \mapsto \bm{\theta}(y):=\left(\cos y_1,y_2,\sin y_1\right),
\end{equation*}
and the portion $\gamma_0$ of the boundary $\gamma$ where the boundary conditions are imposed is given by $\gamma_0:=[0.1,\pi-0.1]\times \{-2\}$.

A third example of middle surface satisfying the assumptions critical to establish Theorem~\ref{th:density} is, for instance, a portion of a one-sheet hyperboloid of revolution whose boundary is the union of two circles, and these circles lie on two distinct planes which are both either orthogonal to the \emph{given} unit-vector $\bm{q}$ or they are orthogonal to the orthogonal complement of the unit-vector $\bm{q}$.
For $\bm{q}=(0,1,0)$, a parametrisation for one such surface is given by
\begin{equation*}
	y=(y_1,y_2) \in [0.1,\pi-0.1] \times [-2,2] \mapsto \bm{\theta}(y):=\left(\sqrt{1+y_2^2} \cos y_1,\sqrt{1+y_2^2} \sin y_1, y_2\right),
\end{equation*}
and the portion $\gamma_0$ of the boundary $\gamma$ where the boundary conditions are imposed is given by $\gamma_0:=(\{0\} \times [-2,2]) \cup (\{\pi-0.1\}\times [-2,2])$.

Despite the assumptions of geometrical nature required to bring the proof of Theorem~\ref{th:density} to a completion, the surfaces satisfying these assumptions are non-trivial in number and in type.

\begin{figure}[H]
	\captionsetup[subfigure]{justification=centering}
	\centering
	\begin{subfigure}[b]{0.3\linewidth}
		\centering
		\includegraphics[width=0.6\textwidth]{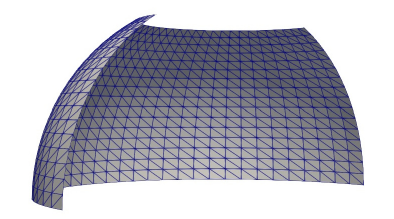}
	\end{subfigure}%
	\begin{subfigure}[b]{0.3\linewidth}
		\centering
		\includegraphics[width=0.6\textwidth]{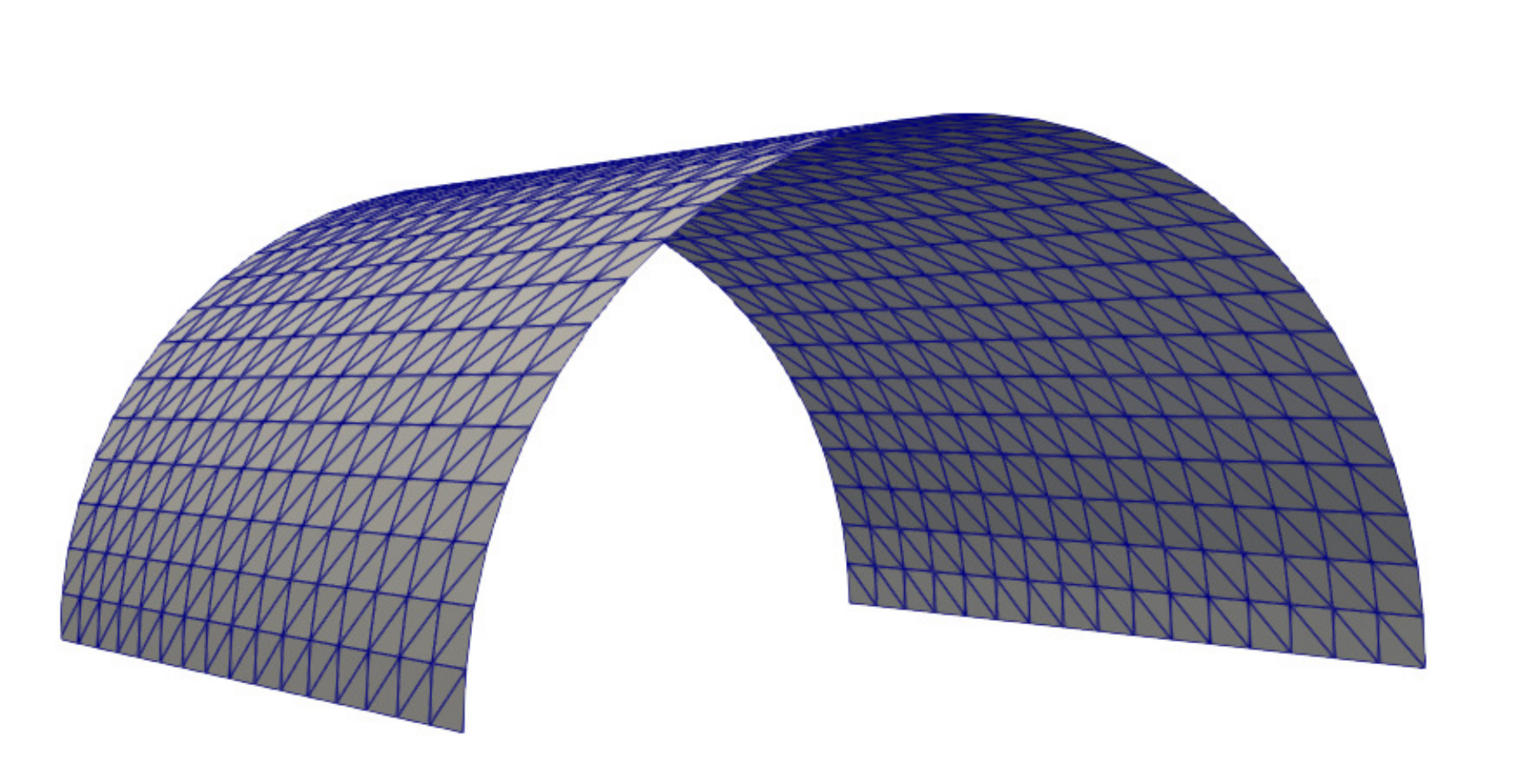}
	\end{subfigure}%
	\begin{subfigure}[b]{0.3\linewidth}
		\centering
		\includegraphics[width=0.6\textwidth]{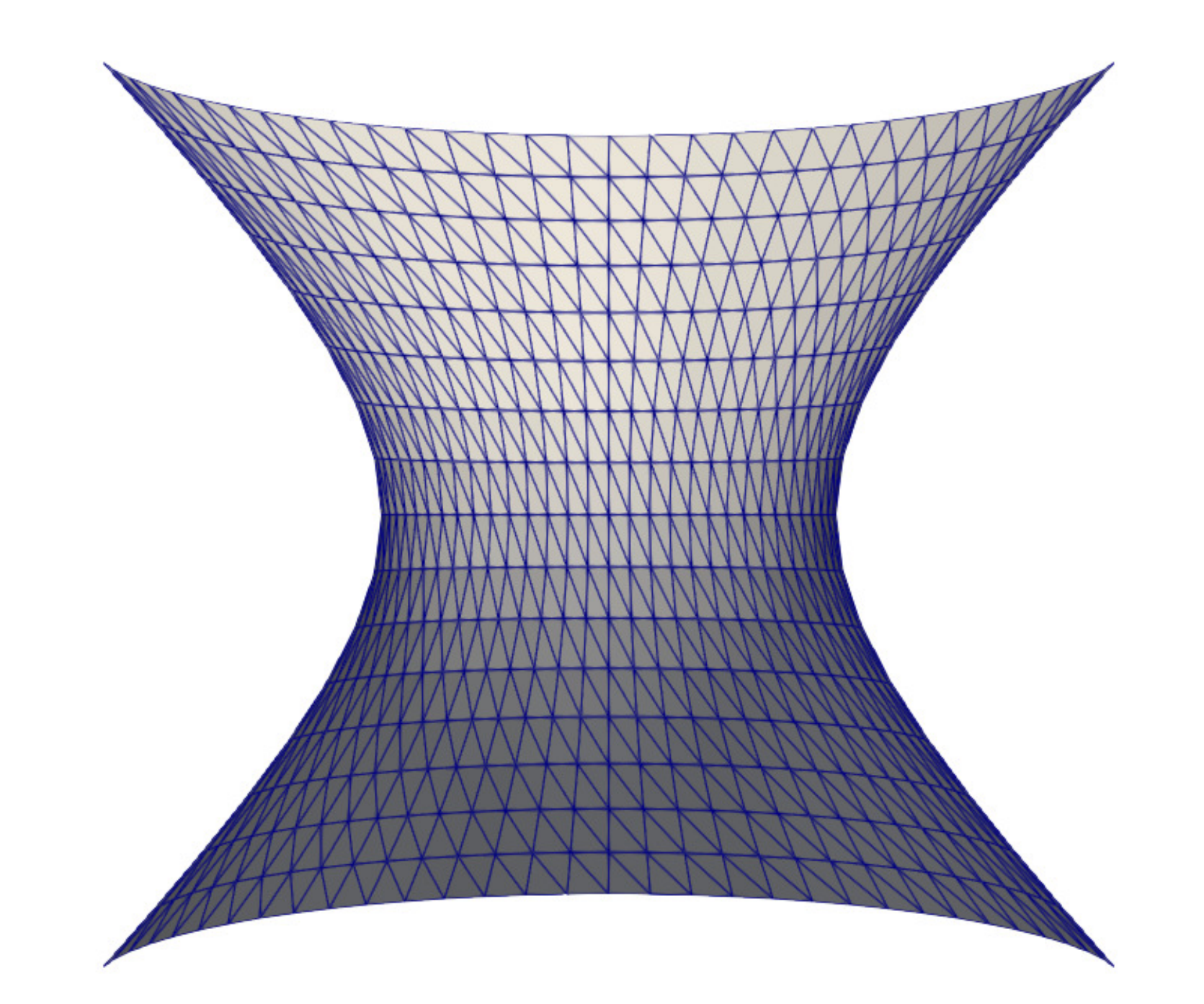}
	\end{subfigure}%
\caption{Three examples of surfaces defined over a rectangular domain $\omega \subset\mathbb{R}^2$ to which Theorem~\ref{th:density} can be applied. (a) A portion of a spherical cap such that the contour where the boundary conditions of place are imposed lies on a plane parallel to the orthogonal complement of $\bm{q}=(0,0,1)$; (b) A portion of a cylinder such that the image of the edges of $\gamma$ that are orthogonal to $\gamma_0$ lies on a plane that is parallel to the orthogonal complement of $\bm{q}=(0,0,1)$; (c) A portion of a hyperboloid of revolution such that the symmetry axis director is parallel to $\bm{q}=(0,1,0)$. In order to apply Theorem~\ref{asymptotics}, we need to require the validity of condition~\eqref{dpcmp}. Note that this is in line with the conclusion obtained for linearly elastic elliptic membrane shells~\cite{CiaMarPie2018}. Assuming the validity of~\eqref{dpcmp} certainly affects the choices for the middle surfaces for which the asymptotic analysis presented in Theorem~\ref{asymptotics} holds.}
\label{fig:3}
\end{figure}

\section*{Conclusion and final remarks}

In this paper we identified a set of two-dimensional variational inequalities that model the displacement of a linearly elastic generalised membrane shell of the first kind subjected to a confinement condition, expressing that all the points of the admissible deformed configurations remain confined in a prescribed half-space.

The starting point of the rigorous asymptotic analysis we carried out is a set of variational inequalities based on the classical energy of three-dimensional linearised elasticity, and posed over a non-empty, closed and convex subset of a suitable Sobolev space.
A rigorous asymptotic analysis departing from one such three-dimensional model led to the identification of a two-dimensional model posed over an abstract completion of a non-empty, closed and convex subset of the space $\bm{H}^1(\omega)$. We observed that the two-dimensional variational inequalities recovered as a result of the asymptotic analysis are the same as the one recovered in~\cite{CiaPie2018bCR,CiaPie2018b}, where the departure point was Koiter's model subjected to the same confinement condition as in this paper.

However, the sets where the solutions for the two-dimensional limit model are not the same, in the sense that the asymptotic analysis departing from the three-dimensional model \emph{with obstacle} and the asymptotic analysis departing from Koiter's model \emph{with obstacle} use in their definitions two sets, denoted by $\bm{U}_M^\sharp(\omega)$ and $\bm{U}^\sharp(\omega)$, respectively, which in general might not coincide.

The result announced in Theorem~\ref{th:density} establishes that, under suitable assumptions on the geometry of the linearly elastic generalised membrane shell of the first kind under consideration, these sets actually \emph{do coincide}, and we may thus regard the justification of Koiter's model for linearly elastic generalised membrane shells of the first kind subjected to remaining confined in a prescribed half-space as complete.

Finally, in connection with the results presented in~\cite{CiaPie2018b,CiaPie2018bCR,CiaMarPie2018b,CiaMarPie2018}, we observe that if the shell middle surface is elliptic and $\gamma_0=\gamma$, then the linearly elastic shell under consideration is a linearly elastic elliptic membrane shell so that, thanks to Korn's inequality for elliptic surfaces (cf., e.g., Theorem~2.7-3 in~\cite{Ciarlet2000}), it follows that $\bm{U}_M^\sharp(\omega)=\bm{U}_M(\omega)=\{\bm{\eta}=(\eta_i)\in H^1_0(\omega)\times H^1_0(\omega) \times L^2(\omega); (\bm{\theta}+\eta_i \bm{a}^i)\cdot\bm{q}\ge 0 \textup{ a.e. in }\omega\}$, provided that $\min_{y\in\overline{\omega}}(\bm{\theta}(y)\cdot\bm{q})>0$ and that $\min_{y\in\overline{\omega}}(\bm{a}^3(y)\cdot\bm{q})>0$. Note that, by contrast with the case of linearly elastic elliptic membrane shells, the condition $\min_{y\in\overline{\omega}}(\bm{a}^3(y)\cdot\bm{q})>0$, denoted here by~\eqref{dpcmp}, has to be assumed for carrying out the asymptotic analysis presented in Theorem~\ref{asymptotics} while, in~\cite{CiaMarPie2018b,CiaMarPie2018}, this condition has to be assumed to establish the validity of the ``density property''.
The most significant outcome of the alternative approach employing~\eqref{cc-new} is that the two-dimensional limit model obtained through a rigorous asymptotic analysis is exactly the \emph{same} as the model derived under the Kirchhoff-Love assumptions in Theorem~\ref{asymptotics}. This convergence of results from two distinct mathematical pathways strongly corroborates the genuineness and robustness of our main findings.
Additionally, note that Theorem~\ref{th:density} must still be established in the proposed simplified framework, underlining its importance in the justification of Koiter's model for linearly elastic generalised membrane shells of the first kind subjected to remaining confined in a prescribed half-space.

\section*{Declarations}

\subsection*{Authors' Contribution}

Not applicable.

\subsection*{Acknowledgements}

The Author is extremely grateful to the Anonymous Referees for the helpful comments and the suggested improvements.

\subsection*{Ethical Approval}

Not applicable.

\subsection*{Availability of Supporting Data}

Not applicable.

\subsection*{Competing Interests}

All authors certify that they have no affiliations with or involvement in any organisation or entity with any competing interests in the subject matter or materials discussed in this manuscript.

\subsection*{Funding}

The research of P.P. is currently being supported by the following agencies:
\begin{itemize}
	\item Natural Science Foundation of China (NSFC), grant number W2533011;
	\item ZJ Talent Program of the Guangdong Province, grant number 2024QN11X057;
	\item Peacock Grant - Type~C of the Shenzhen Municipality.
\end{itemize}
	
	\bibliographystyle{abbrvnat} 
	\bibliography{references.bib}	
	
\end{document}